\documentclass[11pt]{article}%{amsart}
\usepackage[numbers]{natbib}
\usepackage{a4wide}
\usepackage{amsmath}
\usepackage{amssymb}
\usepackage{latexsym}
\usepackage[pdftex]{graphicx}
\usepackage{subcaption}
\usepackage{color}
\usepackage[mathscr]{euscript}
\usepackage{float}
\usepackage{url}
\usepackage[utf8]{inputenc}
\usepackage[english]{babel}
\usepackage[pdftex,dvipsnames,table]{xcolor}
\usepackage{hyperref}
\usepackage{amssymb,amsmath, amsthm}
\usepackage{cleveref}
\usepackage{autonum}
\usepackage{enumitem}
\usepackage{graphicx}
\usepackage{epstopdf}
\usepackage[labelfont=bf,font={small},margin=1cm]{caption}
\usepackage{mathtools}
\usepackage{xargs}
\usepackage[colorinlistoftodos,prependcaption,textsize=tiny]{todonotes}
\usepackage{mathrsfs} 
\usepackage[toc,page]{appendix}
\usepackage{booktabs}
\usepackage{tikz}
\usepackage[pagestyles,extramarks]{titlesec}
\usepackage{geometry}
\newtheorem{thm}{Theorem}
\numberwithin{thm}{section} 
\newtheorem{lem}[thm]{Lemma}
\newtheorem{cor}[thm]{Corollary}
\newtheorem{prop}[thm]{Proposition}

\newtheorem*{claim*}{Claim}
\newtheorem*{con*}{Conjecture}

\theoremstyle{definition}

\newtheorem{rem}[thm]{Remark}

\newtheorem*{exa*}{Example}

\newcommand{\floor}[1]{\left\lfloor #1 \right\rfloor}
\newcommand{\ceil}[1]{\left\lceil #1 \right\rceil}
\newcommand{\on}[1]{\operatorname{#1}}
\renewcommand{\mid}{:}
\newcommandx{\bm}[1]{ \begin{bmatrix} #1 \end{bmatrix}  }
\newcommand{\be}{\coloneqq}

\renewcommand{\vec}[1]{\underline{#1}}

\renewcommand{\epsilon}{\varepsilon} 

\let\emptyset\varnothing

   \def\bE{{\mathbb E}}

   \def\bN{{\mathbb N}}   
\def\bP{{\mathbb P}}      \def\bR{{\mathbb R}}

   \def\bZ{{\mathbb Z}}

   \def\cN{{\mathcal N}}

  \renewenvironment{proof}{\vspace{1ex}\noindent{\bf Proof.}}{\hfill$\blacksquare$\vspace{1ex}}
  \newenvironment{proofof}[1]{\vspace{1ex}\noindent{\bf Proof of #1.}}{\hspace*{\fill}$\blacksquare$\vspace{1ex}}

\newcommand{\Xcal}[0]{\ensuremath{{\mathcal X}}}
\newcommand{\Ycal}[0]{\ensuremath{{\mathcal Y}}}

\newcommand{\eR}[0]{\ensuremath{\mathbb R}}

\newcommand{\eN}[0]{\ensuremath{ \mathbb N}}
\newcommand{\Zed}[0]{\ensuremath{ \mathbb Z}}

\newcommand{\Pee}[0]{\ensuremath{{\mathbb P}}}

\newcommand{\Ee}[0]{\ensuremath{{\mathbb E}}}

\newcommand{\isd}[0]{\hspace{.2ex} \raisebox{-.1ex}{$=$} \hspace{-1.5ex} 
\raisebox{1ex}{{$\scriptstyle d$}} \hspace{.8ex} }

 \newcommand{\eps}{\varepsilon}
\newcommand{\dtv}{d_{\text{TV}}}
\DeclareMathOperator{\Bi}{Bi}

\definecolor{orange}{RGB}{255,127,0}
\definecolor{pink}{RGB}{255,150,150}

\DeclareMathOperator{\inv}{inv}
\DeclareMathOperator{\Mallows}{Mallows}
\DeclareMathOperator{\Geo}{Geo}
\DeclareMathOperator{\TGeo}{TruncGeo}
\newcommand{\nulfieen}{\ensuremath{{}_{0}\phi_{1}}}
\DeclareMathOperator{\Var}{Var}
\DeclareMathOperator{\Cov}{Cov}
\title{Cycles in Mallows random permutations}

\author{
Jimmy He\thanks{Department of mathematics, MIT, USA.
E-mail: {\tt jimmyhe@mit.edu}}
\and
Tobias M\"uller\thanks{Bernoulli Institute, % for Mathematics, CS and AI, 
Groningen University, The Netherlands. 
E-mail: {\tt tobias.muller@rug.nl}.}%
\and 
Teun W.~Verstraaten\thanks{Bernoulli Institute, Groningen University, The Netherlands. 
E-mail: {\tt t.w.verstraaten@rug.nl}.}
}

\begin{document}

\maketitle

\begin{abstract}
We study cycle counts in permutations of $1,\dots,n$ drawn at random according to the Mallows distribution.
Under this distribution, each permutation $\pi \in S_n$ is selected with probability proportional to 
$q^{\inv(\pi)}$, where $q>0$ is a parameter and $\inv(\pi)$ denotes the number of inversions of $\pi$.
For $\ell$ fixed, we study the vector $(C_1(\Pi_n),\dots,C_\ell(\Pi_n))$ where $C_i(\pi)$ denotes the 
number of cycles of length $i$ in $\pi$ and $\Pi_n$ is sampled according to the Mallows distribution.
When $q=1$ the Mallows distribution simply samples a permutation
of $1,\dots,n$ uniformly at random.
A classical result going back to Kolchin and Goncharoff states that in this case, the vector
of cycle counts tends in distribution to a vector of independent Poisson random variables, with means 
$1,\frac12,\frac13,\dots,\frac{1}{\ell}$.

Here we show that if $0<q<1$ is fixed and $n\to\infty$ then there are positive constants $m_i$ such that 
each $C_i(\Pi_n)$ has mean $(1+o(1)) \cdot m_i\cdot n$ and the vector of cycle counts can be suitably rescaled to tend to a 
joint Gaussian distribution.
Our results also show that when $q>1$ there is a striking difference between the behaviour of the even and the odd 
cycles.
The even cycle counts still have linear means, and when properly rescaled tend to a multivariate Gaussian distribution. 
For the odd cycle counts on the other hand, the limiting behaviour depends on the parity of $n$ when $q>1$.
Both $(C_1(\Pi_{2n}),C_3(\Pi_{2n}),\dots)$ and 
$(C_1(\Pi_{2n+1}),C_3(\Pi_{2n+1}),\dots)$ have discrete limiting distributions -- they do not 
need to be renormalized -- but the two limiting distributions are distinct for all $q>1$.
We describe these limiting distributions in terms of Gnedin and Olshanski's bi-infinite extension of the Mallows model.

We investigate these limiting distributions further, and study the behaviour of the constants involved in the Gaussian limit laws.
We for example show that as $q\downarrow 1$ the expected number of 1-cycles tends to $1/2$ -- which, curiously, differs 
from the value corresponding to $q=1$.
In addition we exhibit an interesting ``oscillating'' behaviour in the limiting probability measures 
for $q>1$ and $n$ odd versus $n$ even.
\end{abstract}

\section{Introduction and statement of main results}

Let $S_n$ denote the set of permutations of $[n] := \{1,\dots,n\}$. 
For a permutation $\pi \in S_n$ the ordered pair $(i,j) \in [n]^2$ is an inversion of $\pi$ if $i<j$ and $\pi(i) > \pi(j)$. 
We denote the number of inversions of a permutation $\pi$ by $\inv(\pi)$. 
For $n \in \eN$ and $q>0$, the Mallows distribution $\Mallows(n,q)$ samples a random element $\Pi_n$ of $S_n$ in such a way that
each $\pi \in S_n$ has probability proportional to $q^{\inv(\pi)}$. 
That is, 

\begin{equation}\label{eq:Mallowsdef} 
\Pee( \Pi_n = \pi ) = \frac{q^{\inv(\pi)}}{\sum_{\sigma\in S_n} q^{\inv(\sigma)}}. 
\end{equation}

This distribution on $S_n$ was introduced in the late fifties by C.L. Mallows~\cite{Mallows} in the context of statistical ranking models.
It has since been studied in connection with a diverse range of topics, including 
mixing times of Markov chains~\cite{Benjamini2005, Diaconis2000},
finitely dependent colorings of the integers~\cite{HolroydHutchcroftLevy2020}, 
stable matchings~\cite{AngelEtAl}, random binary search trees~\cite{Louigi}, learning theory~\cite{BravermanMossel,Tang19}, 
 $q$-analogs of exchangeability~\cite{GnedinOlshanski2010,Gnedin}, determinantal point processes~\cite{BorodinDiaconisFulman2010},
statistical physics~\cite{Starr2009,StarrWalters2018} and genomics~\cite{FangEtAl2021}.

Aspects of the Mallows distribution that have been studied include the longest 
increasing subsequence~\cite{basu2016limit,BhatnagarPeled2015, MuellerStarr2013}, longest common subsequences~\cite{Jin2019}, 
pattern avoidance~\cite{CraneDesalvo2017, CraneDesalvoElizalde2018, Pinsky2021}, 
the number of descents~\cite{HeDescents} and the cycle structure~\cite{Peled}.

In the special case when $q=1$ the Mallows distribution coincides with the uniform distribution on $S_n$.
A classical result going back to Gontcharoff~\cite{Gontcharoff42} and Kolchin~\cite{Kolchin71}  
states that in this case, for every fixed $\ell$:

$$ \left( C_1( \Pi_n ), \dots, C_\ell(\Pi_n) \right) \xrightarrow[n\to\infty]{\text{d}}
\left(X_1,\dots,X_\ell\right), $$

\noindent
where $C_i(\pi)$ denotes the number of cycles of length $i$ in the permutation $\pi$, and $X_1,\dots,X_\ell$ are independent 
and $X_i$ is Poisson distributed with mean $1/i$ for each $i=1,\dots,\ell$.
In spite of the long history and considerable attention received by the Mallows distribution, 
until very recently the problem of determining analogues of this result for the $\Mallows(n,q)$ distribution with 
$q \neq 1$ seems to have escaped attention. In a recent paper, Gladkich and Peled~\cite{Peled} studied the cycle structure of 
the Mallows distribution when $q=q(n)$ depends on $n$ and approaches one as $n\to\infty$. 
Here we will focus instead on the limiting distribution of the cycle counts when $q\neq 1$ is fixed
and $n$ tends to infinity.

Our first result shows that for $0<q<1$ each $C_i(\Pi_n)$ has a mean that is linear in $n$, and that 
for every fixed $\ell$, the vector $(C_1(\Pi_n),\dots,C_\ell(\Pi_n))$ can be suitably rescaled 
so that it tends to a jointly normal limiting distribution.

\begin{thm}\label{thm:normal}
Fix $0<q<1$ and let $\Pi_n \sim \Mallows(n,q)$. 
There exist positive constants $m_1,m_2,\ldots $ and an infinite matrix $P\in \bR^{\bN \times \bN}$ such that for all 
$\ell\geq 1$ we have

\begin{equation}
\frac{1}{\sqrt{n}} \left(C_1(\Pi_n)-m_1 n, \ldots , C_\ell (\Pi_n) - m_{\ell} n \right) \quad 
\xrightarrow[n\to\infty]{d}\quad \cN_\ell(\vec{0}, P_\ell),
\end{equation}

\noindent
where $\cN_\ell(\cdot,\cdot)$ denotes the $\ell$--dimensional multivariate normal distribution and $P_\ell$ is 
the submatrix of $P$ on the indices $[\ell] \times [\ell]$.
\end{thm}

As it happens, for $q>1$, there is a major difference between the behaviour of even cycles and odd cycles.
For even cycle counts we have a result analogous to the previous theorem.

\begin{thm}\label{thm:clt_q_gt_one}
Fix $q>1$ and let $\Pi_n \sim \on{Mallows}(n,q)$. 
There exists constants $\mu_2,\mu_4,\ldots $ and an infinite matrix $Q\in \bR^{\bN \times \bN}$ such that for all $\ell\geq 1$ we have

\begin{equation}
\frac{1}{\sqrt{n}}(C_{2}(\Pi_n)  - \mu_{2}  n,\ldots, C_{2 \ell}(\Pi_n)  - \mu_{2\ell}  n) \quad 
\xrightarrow[n\to\infty]{d}\quad \cN_\ell(\vec{0}, Q_\ell),
\end{equation}

\noindent
where $\cN_\ell(\cdot,\cdot)$ denotes the $\ell$--dimensional multivariate normal distribution and $Q_\ell$ is the submatrix of 
$Q$ on the indices $[\ell] \times [\ell]$.
\end{thm}

We will describe the limiting distributions for odd cycles in the case when $q>1$ in terms of the 
bi-infinite analogue of the Mallows distribution that was introduced by Gnedin and Olshanski~\cite{Gnedin}. 
This is a random bijection $\Sigma : \Zed \to \Zed$, whose distribution we'll denote by $\Mallows(\Zed,q)$. 
See Section~\ref{sec:prelim} for more discussion and relevant facts.

Throughout the paper $r,\rho$ denote the bijections of $\Zed$ defined by $r(i) := -i$ and $\rho(i) := 1 - i$. 

\begin{thm}\label{thm:odd_cycles}
Let $q>1$ and $\Pi_{n}\sim \on{Mallows}(n, q)$ and $\Sigma \sim\on{Mallows}(\bZ, 1/q)$. 
We have 

$$ (C_1(\Pi_{2n+1}), C_3(\Pi_{2n+1}),\ldots )
 \xrightarrow[n\to\infty]{d}    \quad (C_1(r \circ \Sigma), C_3(r \circ \Sigma),\ldots ) 
 $$ 
 
\noindent
and 

$$ (C_1(\Pi_{2n}), C_3(\Pi_{2n}),\ldots )   \xrightarrow[n\to\infty]{d}    
\quad (C_1(\rho \circ \Sigma), C_3(\rho \circ \Sigma),\ldots ). 
$$ 

\noindent
Moreover, the two limiting distributions above are distinct for all $q>1$.
\end{thm}

The permutations $r \circ \Sigma$ and $\rho \circ \Sigma$ almost surely have only finitely many odd cycles, 
as we will see in more detail later on. 
%In particular the infinite dimensional vectors occurring as the limit 
%objects in Theorem~\ref{thm:odd_cycles} almost surely have bounded support. So, in a sense we can just think of them as discrete 
%random variables in disguise. 

Next, we study the properties of the constants $m_1,m_2,\dots$ occurring in Theorem~\ref{thm:normal}.
The first part of the next result gives an interpretation of these constants in terms of the $\Mallows(\Zed,q)$ distribution.

\begin{thm}\label{thm:mi}
Let $0<q<1$ and $\Sigma \sim \on{Mallows}(\bZ, q)$, and let $m_1,m_2,\dots$ be as provided by Theorem~\ref{thm:normal}. 
\begin{enumerate}
 \item\label{itm:m1part1} For $i=1,2,\dots$ we have %
 $$ m_i = (1/i) \cdot \Pee( 0 \text{ lies in an $i$-cycle of $\Sigma$ } ). $$%
 \noindent
 In particular %
 $$m_1 = \Pee( \Sigma(0)=0 ). $$%
 \item\label{itm:m1part2} $\sum_{i=1}^\infty i \cdot m_i = 1$.
 
 \item\label{itm:m1part3} We have %
 \begin{align}
(m_1,m_2,m_3,\ldots) &\to (1,0,0,\ldots)\qquad\text{ as }q\downarrow 0, \\
(m_1,m_2,m_3,\ldots) &\to (0,0,0,\ldots)\qquad\text{ as }q\uparrow 1,
\end{align}%
\noindent
where the convergence can be taken with respect to the $L_1$-norm.
Moreover,%
\begin{align}
m_1 &= 1- 2q + o(q) \quad & \text{ as }q\downarrow 0,\label{eq:akdhbjhe}\\
m_1 &= \frac{1-q}{4} + o(1-q) \quad & \text{ as }q\uparrow 1.\label{eq:asdasdas}
\end{align}%
\end{enumerate}
\end{thm}

Combining Part~\ref{itm:m1part1} of the above theorem with Theorem 5.1 of~\cite{Gnedin}, we can 
write 

$$ m_1 = \nulfieen(-;q;q,q^3) \cdot (1-q) \cdot \prod_{i=1}^\infty (1-q^i), $$

\noindent
where ${}_r\phi_s$ denotes the $q$-hypergeometric function -- see~\cite{GasperRahman} for the definition and background. 
An alternative expression, based on the work of Gladkich and Peled~\cite{Peled}, for $m_1$ is given in Lemma~\ref{lem:m_1_b} below.
Figure~\ref{fig:m1} shows a plot of $m_1$ versus $q$ together with the results of computer simulations.

\begin{figure}[h!]
\centering
\begingroup
  \fontfamily{enhanced}%
  \selectfont
  \makeatletter
  \providecommand\color[2][]{%
    \GenericError{(gnuplot) \space\space\space\@spaces}{%
      Package color not loaded in conjunction with
      terminal option `colourtext'%
    }{See the gnuplot documentation for explanation.%
    }{Either use 'blacktext' in gnuplot or load the package
      color.sty in LaTeX.}%
    \renewcommand\color[2][]{}%
  }%
  \providecommand\includegraphics[2][]{%
    \GenericError{(gnuplot) \space\space\space\@spaces}{%
      Package graphicx or graphics not loaded%
    }{See the gnuplot documentation for explanation.%
    }{The gnuplot epslatex terminal needs graphicx.sty or graphics.sty.}%
    \renewcommand\includegraphics[2][]{}%
  }%
  \providecommand\rotatebox[2]{#2}%
  \@ifundefined{ifGPcolor}{%
    \newif\ifGPcolor
    \GPcolorfalse
  }{}%
  \@ifundefined{ifGPblacktext}{%
    \newif\ifGPblacktext
    \GPblacktexttrue
  }{}%
  % define a \g@addto@macro without @ in the name:
  \let\gplgaddtomacro\g@addto@macro
  % define empty templates for all commands taking text:
  \gdef\gplbacktext{}%
  \gdef\gplfronttext{}%
  \makeatother
  \ifGPblacktext
    % no textcolor at all
    \def\colorrgb#1{}%
    \def\colorgray#1{}%
  \else
    % gray or color?
    \ifGPcolor
      \def\colorrgb#1{\color[rgb]{#1}}%
      \def\colorgray#1{\color[gray]{#1}}%
      \expandafter\def\csname LTw\endcsname{\color{white}}%
      \expandafter\def\csname LTb\endcsname{\color{black}}%
      \expandafter\def\csname LTa\endcsname{\color{black}}%
      \expandafter\def\csname LT0\endcsname{\color[rgb]{1,0,0}}%
      \expandafter\def\csname LT1\endcsname{\color[rgb]{0,1,0}}%
      \expandafter\def\csname LT2\endcsname{\color[rgb]{0,0,1}}%
      \expandafter\def\csname LT3\endcsname{\color[rgb]{1,0,1}}%
      \expandafter\def\csname LT4\endcsname{\color[rgb]{0,1,1}}%
      \expandafter\def\csname LT5\endcsname{\color[rgb]{1,1,0}}%
      \expandafter\def\csname LT6\endcsname{\color[rgb]{0,0,0}}%
      \expandafter\def\csname LT7\endcsname{\color[rgb]{1,0.3,0}}%
      \expandafter\def\csname LT8\endcsname{\color[rgb]{0.5,0.5,0.5}}%
    \else
      % gray
      \def\colorrgb#1{\color{black}}%
      \def\colorgray#1{\color[gray]{#1}}%
      \expandafter\def\csname LTw\endcsname{\color{white}}%
      \expandafter\def\csname LTb\endcsname{\color{black}}%
      \expandafter\def\csname LTa\endcsname{\color{black}}%
      \expandafter\def\csname LT0\endcsname{\color{black}}%
      \expandafter\def\csname LT1\endcsname{\color{black}}%
      \expandafter\def\csname LT2\endcsname{\color{black}}%
      \expandafter\def\csname LT3\endcsname{\color{black}}%
      \expandafter\def\csname LT4\endcsname{\color{black}}%
      \expandafter\def\csname LT5\endcsname{\color{black}}%
      \expandafter\def\csname LT6\endcsname{\color{black}}%
      \expandafter\def\csname LT7\endcsname{\color{black}}%
      \expandafter\def\csname LT8\endcsname{\color{black}}%
    \fi
  \fi
    \setlength{\unitlength}{0.0500bp}%
    \ifx\gptboxheight\undefined%
      \newlength{\gptboxheight}%
      \newlength{\gptboxwidth}%
      \newsavebox{\gptboxtext}%
    \fi%
    \setlength{\fboxrule}{0.5pt}%
    \setlength{\fboxsep}{1pt}%
\begin{picture}(7200.00,5040.00)%
    \gplgaddtomacro\gplbacktext{%
      \csname LTb\endcsname%%
      \put(814,704){\makebox(0,0)[r]{\strut{}$0$}}%
      \put(814,1527){\makebox(0,0)[r]{\strut{}$0.2$}}%
      \put(814,2350){\makebox(0,0)[r]{\strut{}$0.4$}}%
      \put(814,3173){\makebox(0,0)[r]{\strut{}$0.6$}}%
      \put(814,3996){\makebox(0,0)[r]{\strut{}$0.8$}}%
      \put(814,4819){\makebox(0,0)[r]{\strut{}$1$}}%
      \put(946,484){\makebox(0,0){\strut{}$0$}}%
      \put(1532,484){\makebox(0,0){\strut{}$0.1$}}%
      \put(2117,484){\makebox(0,0){\strut{}$0.2$}}%
      \put(2703,484){\makebox(0,0){\strut{}$0.3$}}%
      \put(3289,484){\makebox(0,0){\strut{}$0.4$}}%
      \put(3875,484){\makebox(0,0){\strut{}$0.5$}}%
      \put(4460,484){\makebox(0,0){\strut{}$0.6$}}%
      \put(5046,484){\makebox(0,0){\strut{}$0.7$}}%
      \put(5632,484){\makebox(0,0){\strut{}$0.8$}}%
      \put(6217,484){\makebox(0,0){\strut{}$0.9$}}%
      \put(6803,484){\makebox(0,0){\strut{}$1$}}%
    }%
    \gplgaddtomacro\gplfronttext{%
      \csname LTb\endcsname%%
      \put(209,2761){\rotatebox{-270}{\makebox(0,0){\strut{}$m_1$}}}%
      \put(3874,154){\makebox(0,0){\strut{}$q$}}%
      \csname LTb\endcsname%%
      \put(5816,4646){\makebox(0,0)[r]{\strut{}exact}}%
      \csname LTb\endcsname%%
      \put(5816,4426){\makebox(0,0)[r]{\strut{}simulations}}%
    }%
    \gplbacktext
    \put(0,0){\includegraphics{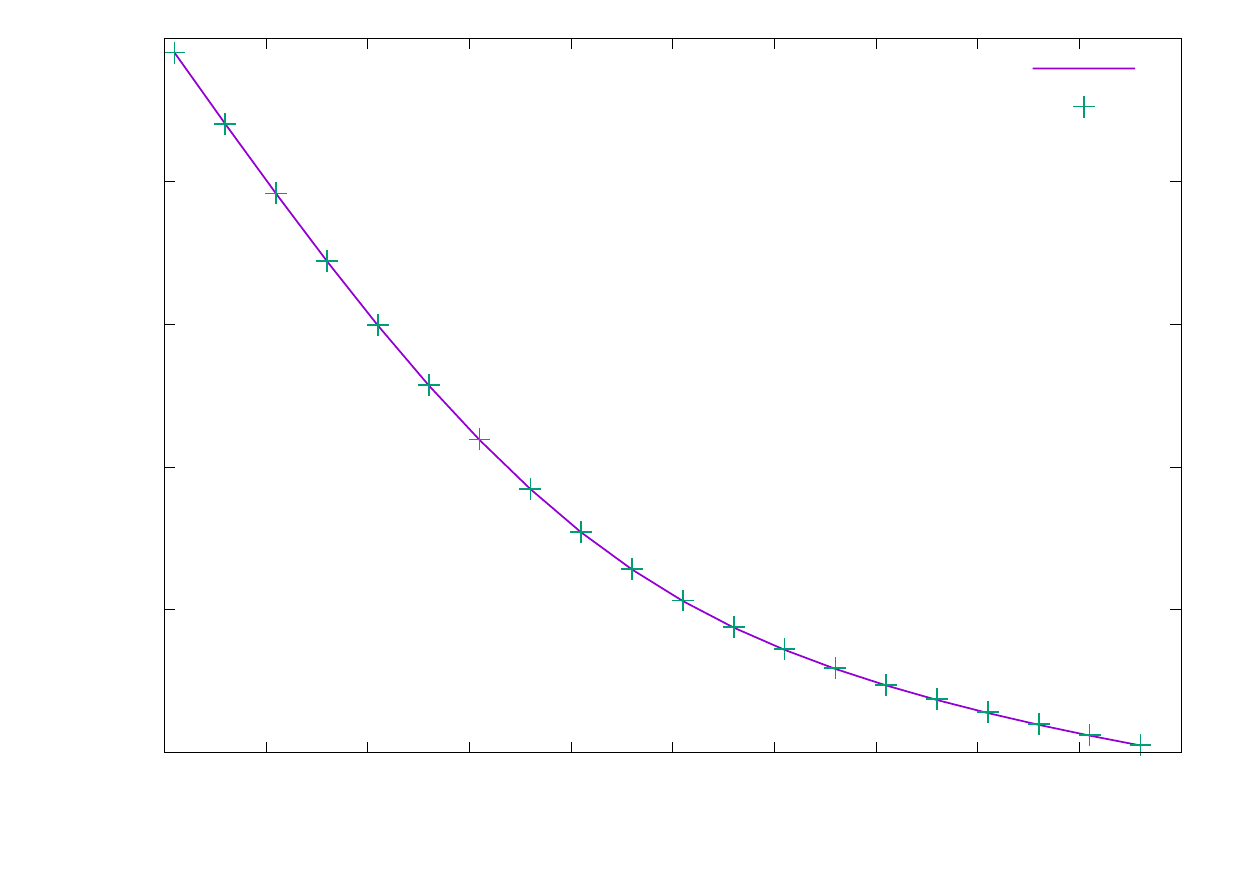}}%
    \gplfronttext
  \end{picture}%
\endgroup
\caption{\label{fig:m1} A plot of $m_1$ versus $q$. 
The crosses correspond to the average number of 1-cycles
in $10000$ samples of the $\on{Mallows}(1000,q)$ distribution.}
\end{figure}

\noindent
The next result provides similar results for the constants appearing in Theorem~\ref{thm:clt_q_gt_one}.

\begin{thm}\label{thm:mu2i}
Let $q>1$, let $\mu_2,\mu_4,\dots$ be as provided by Theorem~\ref{thm:clt_q_gt_one}, and let 
$\Sigma,\Sigma' \sim \on{Mallows}(\bZ, 1/q)$ be independent.
\begin{enumerate}	
\item\label{itm:mu2ipart1} For $i=1,2,\dots$ we have%
$$ \mu_{2i} = \frac{1}{2i} \cdot \Pee( \text{$0$ is in an $i$-cycle of $\Sigma'\circ\Sigma$} ), $$ %
and in particular%
$$ \mu_2 = \frac12 \cdot \sum\limits_{i\in\Zed} \bP \left[ \Sigma(0) = i \right]^2. $$ %
\item\label{itm:mu2ipart2} $\sum_{i=1}^{\infty} 2i \cdot \mu_{2i} = 1$. 
\item\label{itm:mu2ipart3} We have %
\begin{align}
(\mu_2,\mu_4,\mu_6,\ldots)&\to (0,0,0,\ldots )\qquad\text{ as }q\downarrow 1,\\
(\mu_2, \mu_4,\mu_6,\ldots)&\to (1/2,0,0,\ldots )\qquad\text{ as }q\to \infty,
\end{align} %
where the convergence can be taken with respect to the $L_1$-norm. 
\end{enumerate}
\end{thm}

Again, combining Part~\ref{itm:mu2ipart1} with Theorem 5.1 in~\cite{Gnedin} gives an expression for $\mu_2$ as an explicit function 
of $q$. %%%% nog een expliciete expressie geven ...
Figure~\ref{fig:mu2} provides a plot of $\mu_2$ versus $q$ together with the results of computer simulations.
We mention that Pitman and Tang (\cite{PitmanTang}, Proposition 3.3) give a result for so-called regenerative random permutations, 
that is closely related to parts~\ref{itm:m1part1} and~\ref{itm:m1part2} of Theorems~\ref{thm:mi} and~\ref{thm:mu2i}.

\begin{figure}[h!]
\centering
\begingroup
  \fontfamily{enhanced}%
  \selectfont
  \makeatletter
  \providecommand\color[2][]{%
    \GenericError{(gnuplot) \space\space\space\@spaces}{%
      Package color not loaded in conjunction with
      terminal option `colourtext'%
    }{See the gnuplot documentation for explanation.%
    }{Either use 'blacktext' in gnuplot or load the package
      color.sty in LaTeX.}%
    \renewcommand\color[2][]{}%
  }%
  \providecommand\includegraphics[2][]{%
    \GenericError{(gnuplot) \space\space\space\@spaces}{%
      Package graphicx or graphics not loaded%
    }{See the gnuplot documentation for explanation.%
    }{The gnuplot epslatex terminal needs graphicx.sty or graphics.sty.}%
    \renewcommand\includegraphics[2][]{}%
  }%
  \providecommand\rotatebox[2]{#2}%
  \@ifundefined{ifGPcolor}{%
    \newif\ifGPcolor
    \GPcolorfalse
  }{}%
  \@ifundefined{ifGPblacktext}{%
    \newif\ifGPblacktext
    \GPblacktexttrue
  }{}%
  % define a \g@addto@macro without @ in the name:
  \let\gplgaddtomacro\g@addto@macro
  % define empty templates for all commands taking text:
  \gdef\gplbacktext{}%
  \gdef\gplfronttext{}%
  \makeatother
  \ifGPblacktext
    % no textcolor at all
    \def\colorrgb#1{}%
    \def\colorgray#1{}%
  \else
    % gray or color?
    \ifGPcolor
      \def\colorrgb#1{\color[rgb]{#1}}%
      \def\colorgray#1{\color[gray]{#1}}%
      \expandafter\def\csname LTw\endcsname{\color{white}}%
      \expandafter\def\csname LTb\endcsname{\color{black}}%
      \expandafter\def\csname LTa\endcsname{\color{black}}%
      \expandafter\def\csname LT0\endcsname{\color[rgb]{1,0,0}}%
      \expandafter\def\csname LT1\endcsname{\color[rgb]{0,1,0}}%
      \expandafter\def\csname LT2\endcsname{\color[rgb]{0,0,1}}%
      \expandafter\def\csname LT3\endcsname{\color[rgb]{1,0,1}}%
      \expandafter\def\csname LT4\endcsname{\color[rgb]{0,1,1}}%
      \expandafter\def\csname LT5\endcsname{\color[rgb]{1,1,0}}%
      \expandafter\def\csname LT6\endcsname{\color[rgb]{0,0,0}}%
      \expandafter\def\csname LT7\endcsname{\color[rgb]{1,0.3,0}}%
      \expandafter\def\csname LT8\endcsname{\color[rgb]{0.5,0.5,0.5}}%
    \else
      % gray
      \def\colorrgb#1{\color{black}}%
      \def\colorgray#1{\color[gray]{#1}}%
      \expandafter\def\csname LTw\endcsname{\color{white}}%
      \expandafter\def\csname LTb\endcsname{\color{black}}%
      \expandafter\def\csname LTa\endcsname{\color{black}}%
      \expandafter\def\csname LT0\endcsname{\color{black}}%
      \expandafter\def\csname LT1\endcsname{\color{black}}%
      \expandafter\def\csname LT2\endcsname{\color{black}}%
      \expandafter\def\csname LT3\endcsname{\color{black}}%
      \expandafter\def\csname LT4\endcsname{\color{black}}%
      \expandafter\def\csname LT5\endcsname{\color{black}}%
      \expandafter\def\csname LT6\endcsname{\color{black}}%
      \expandafter\def\csname LT7\endcsname{\color{black}}%
      \expandafter\def\csname LT8\endcsname{\color{black}}%
    \fi
  \fi
    \setlength{\unitlength}{0.0500bp}%
    \ifx\gptboxheight\undefined%
      \newlength{\gptboxheight}%
      \newlength{\gptboxwidth}%
      \newsavebox{\gptboxtext}%
    \fi%
    \setlength{\fboxrule}{0.5pt}%
    \setlength{\fboxsep}{1pt}%
\begin{picture}(7200.00,5040.00)%
    \gplgaddtomacro\gplbacktext{%
      \csname LTb\endcsname%%
      \put(814,704){\makebox(0,0)[r]{\strut{}$0$}}%
      \put(814,1527){\makebox(0,0)[r]{\strut{}$0.1$}}%
      \put(814,2350){\makebox(0,0)[r]{\strut{}$0.2$}}%
      \put(814,3173){\makebox(0,0)[r]{\strut{}$0.3$}}%
      \put(814,3996){\makebox(0,0)[r]{\strut{}$0.4$}}%
      \put(814,4819){\makebox(0,0)[r]{\strut{}$0.5$}}%
      \put(946,484){\makebox(0,0){\strut{}$0$}}%
      \put(2117,484){\makebox(0,0){\strut{}$5$}}%
      \put(3289,484){\makebox(0,0){\strut{}$10$}}%
      \put(4460,484){\makebox(0,0){\strut{}$15$}}%
      \put(5632,484){\makebox(0,0){\strut{}$20$}}%
      \put(6803,484){\makebox(0,0){\strut{}$25$}}%
    }%
    \gplgaddtomacro\gplfronttext{%
      \csname LTb\endcsname%%
      \put(209,2761){\rotatebox{-270}{\makebox(0,0){\strut{}$\mu_2$}}}%
      \put(3874,154){\makebox(0,0){\strut{}$q$}}%
      \csname LTb\endcsname%%
      \put(2530,4646){\makebox(0,0)[r]{\strut{}exact}}%
      \csname LTb\endcsname%%
      \put(2530,4426){\makebox(0,0)[r]{\strut{}simulations}}%
    }%
    \gplbacktext
    \put(0,0){\includegraphics{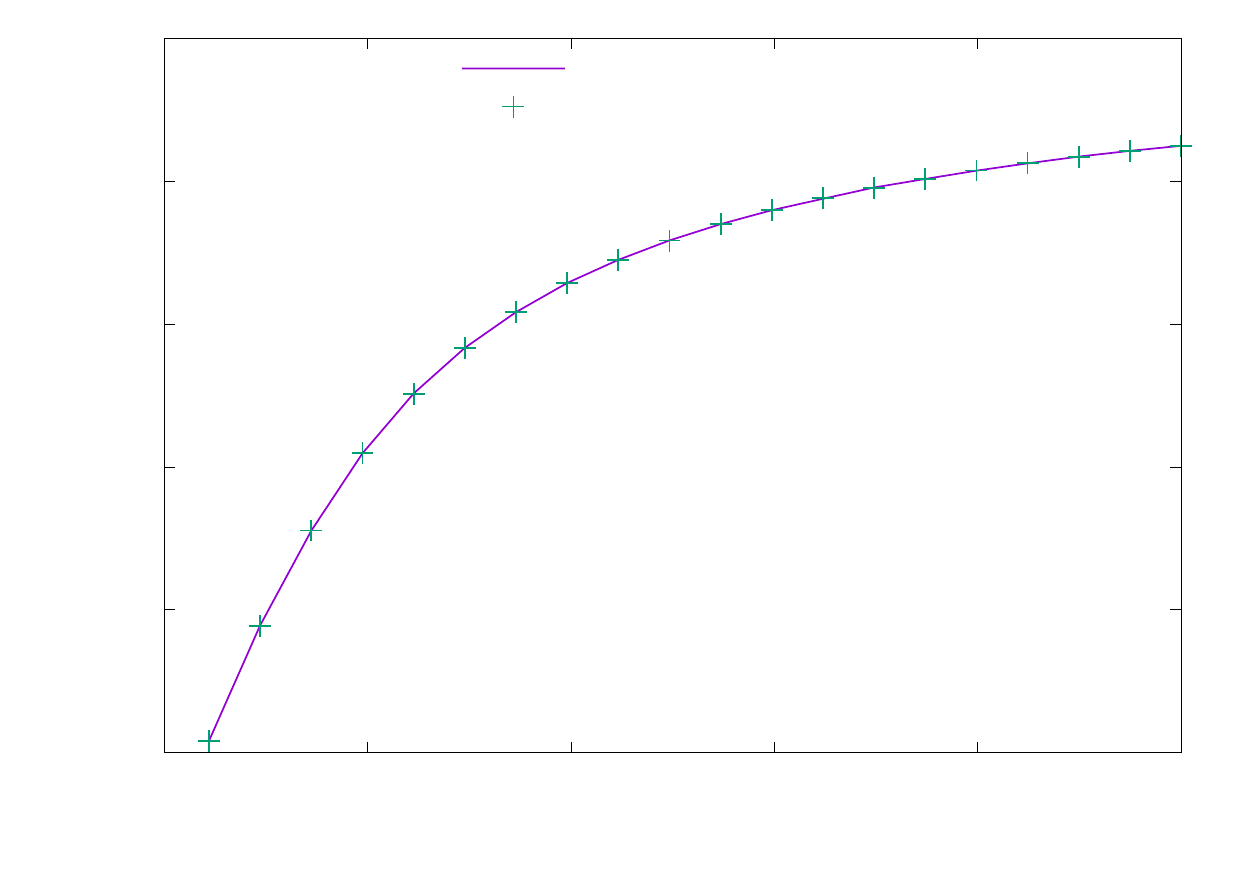}}%
    \gplfronttext
  \end{picture}%
\endgroup
\caption{A plot of $\mu_2$ versus $q$. 
The simulations were done sampling a $\on{Mallows}(1000,q)$ distribution $10000$ times, and taking the average number of 
2-cycles.\label{fig:mu2}}
\end{figure}

Next we provide some results on the asymptotic expected number of 1-cycles when $q>1$.
For notational convenience let us write

\begin{equation}\label{eq:cecodef}
c_e := \Ee C_1( \rho\circ\Sigma ), \quad %\lim_{n\to \infty} \bE C_1(\on{Mallows}(2n,q)),\qquad\text{ and }\\
c_o := \Ee C_1( r\circ\Sigma ), %\lim_{n\to \infty} \bE C_1(\on{Mallows}(2n + 1,q)).
\end{equation}%
where again $\Sigma\sim\Mallows(\Zed,1/q)$ and $r,\rho$ are given by $r(i)=-i,\rho(i)=1-i$.

\begin{thm}\label{thm:ceco}
	Let $q > 1$ and $c_e,c_o$ as given by~\eqref{eq:cecodef} and $\Sigma \sim \on{Mallows}(\bZ, 1/q)$ . 
\begin{enumerate}
 \item\label{itm:cecopart1} We have
	\begin{align}
		c_e = \bP \left[ \Sigma(0) \textup{ odd} \right] \qquad \text{and}\qquad
		c_o = \bP \left[ \Sigma(0) \textup{ even} \right].
	\end{align}
\item\label{itm:cecopart2} We have% 
%
%$$ \frac{1}{1+q} \leq c_e,c_o \leq \frac{q}{1+q}, $$
%
%and in particular
\begin{align}
\lim_{q\downarrow 1 } c_e &=	\lim_{q\downarrow 1 }c_o = \frac12, \\
\lim_{q\to\infty} c_e&=0, \\ 
\lim_{q\to\infty} c_o&=1.
\end{align}
Moreover, as $q \to \infty$ we have% 
$$ c_o = 1 - 2/q + o(1/q) \quad \text{ and } \quad c_e = 2/q + o(1/q) \quad \text{ as $q\to\infty$. } $$% 
\end{enumerate}
\end{thm}

We note that Part~\ref{itm:cecopart1} of the above theorem gives that in particular 

$$c_e + c_o = 1,$$ 

\noindent
for all $q>1$. 
Again, Theorem 5.1 in~\cite{Gnedin} allows us to convert the probabilities given in Part~\ref{itm:cecopart1} of the above theorem
into explicit functions of $q$.
Plots of $c_e$ and $c_o$ as a function of $q$, together with the results of 
computer simulations are shown in Figure \ref{fig:fixed}.

\begin{figure}
\centering
\begingroup
  \fontfamily{enhanced}%
  \selectfont
  \makeatletter
  \providecommand\color[2][]{%
    \GenericError{(gnuplot) \space\space\space\@spaces}{%
      Package color not loaded in conjunction with
      terminal option `colourtext'%
    }{See the gnuplot documentation for explanation.%
    }{Either use 'blacktext' in gnuplot or load the package
      color.sty in LaTeX.}%
    \renewcommand\color[2][]{}%
  }%
  \providecommand\includegraphics[2][]{%
    \GenericError{(gnuplot) \space\space\space\@spaces}{%
      Package graphicx or graphics not loaded%
    }{See the gnuplot documentation for explanation.%
    }{The gnuplot epslatex terminal needs graphicx.sty or graphics.sty.}%
    \renewcommand\includegraphics[2][]{}%
  }%
  \providecommand\rotatebox[2]{#2}%
  \@ifundefined{ifGPcolor}{%
    \newif\ifGPcolor
    \GPcolorfalse
  }{}%
  \@ifundefined{ifGPblacktext}{%
    \newif\ifGPblacktext
    \GPblacktexttrue
  }{}%
  % define a \g@addto@macro without @ in the name:
  \let\gplgaddtomacro\g@addto@macro
  % define empty templates for all commands taking text:
  \gdef\gplbacktext{}%
  \gdef\gplfronttext{}%
  \makeatother
  \ifGPblacktext
    % no textcolor at all
    \def\colorrgb#1{}%
    \def\colorgray#1{}%
  \else
    % gray or color?
    \ifGPcolor
      \def\colorrgb#1{\color[rgb]{#1}}%
      \def\colorgray#1{\color[gray]{#1}}%
      \expandafter\def\csname LTw\endcsname{\color{white}}%
      \expandafter\def\csname LTb\endcsname{\color{black}}%
      \expandafter\def\csname LTa\endcsname{\color{black}}%
      \expandafter\def\csname LT0\endcsname{\color[rgb]{1,0,0}}%
      \expandafter\def\csname LT1\endcsname{\color[rgb]{0,1,0}}%
      \expandafter\def\csname LT2\endcsname{\color[rgb]{0,0,1}}%
      \expandafter\def\csname LT3\endcsname{\color[rgb]{1,0,1}}%
      \expandafter\def\csname LT4\endcsname{\color[rgb]{0,1,1}}%
      \expandafter\def\csname LT5\endcsname{\color[rgb]{1,1,0}}%
      \expandafter\def\csname LT6\endcsname{\color[rgb]{0,0,0}}%
      \expandafter\def\csname LT7\endcsname{\color[rgb]{1,0.3,0}}%
      \expandafter\def\csname LT8\endcsname{\color[rgb]{0.5,0.5,0.5}}%
    \else
      % gray
      \def\colorrgb#1{\color{black}}%
      \def\colorgray#1{\color[gray]{#1}}%
      \expandafter\def\csname LTw\endcsname{\color{white}}%
      \expandafter\def\csname LTb\endcsname{\color{black}}%
      \expandafter\def\csname LTa\endcsname{\color{black}}%
      \expandafter\def\csname LT0\endcsname{\color{black}}%
      \expandafter\def\csname LT1\endcsname{\color{black}}%
      \expandafter\def\csname LT2\endcsname{\color{black}}%
      \expandafter\def\csname LT3\endcsname{\color{black}}%
      \expandafter\def\csname LT4\endcsname{\color{black}}%
      \expandafter\def\csname LT5\endcsname{\color{black}}%
      \expandafter\def\csname LT6\endcsname{\color{black}}%
      \expandafter\def\csname LT7\endcsname{\color{black}}%
      \expandafter\def\csname LT8\endcsname{\color{black}}%
    \fi
  \fi
    \setlength{\unitlength}{0.0500bp}%
    \ifx\gptboxheight\undefined%
      \newlength{\gptboxheight}%
      \newlength{\gptboxwidth}%
      \newsavebox{\gptboxtext}%
    \fi%
    \setlength{\fboxrule}{0.5pt}%
    \setlength{\fboxsep}{1pt}%
\begin{picture}(7200.00,5040.00)%
    \gplgaddtomacro\gplbacktext{%
      \csname LTb\endcsname%%
      \put(814,704){\makebox(0,0)[r]{\strut{}$0$}}%
      \put(814,1527){\makebox(0,0)[r]{\strut{}$0.2$}}%
      \put(814,2350){\makebox(0,0)[r]{\strut{}$0.4$}}%
      \put(814,3173){\makebox(0,0)[r]{\strut{}$0.6$}}%
      \put(814,3996){\makebox(0,0)[r]{\strut{}$0.8$}}%
      \put(814,4819){\makebox(0,0)[r]{\strut{}$1$}}%
      \put(946,484){\makebox(0,0){\strut{}$1$}}%
      \put(1922,484){\makebox(0,0){\strut{}$5$}}%
      \put(3142,484){\makebox(0,0){\strut{}$10$}}%
      \put(4363,484){\makebox(0,0){\strut{}$15$}}%
      \put(5583,484){\makebox(0,0){\strut{}$20$}}%
      \put(6803,484){\makebox(0,0){\strut{}$25$}}%
    }%
    \gplgaddtomacro\gplfronttext{%
      \csname LTb\endcsname%%
      \put(209,2761){\rotatebox{-270}{\makebox(0,0){\strut{}}}}%
      \put(3874,154){\makebox(0,0){\strut{}$q$}}%
      \csname LTb\endcsname%%
      \put(1342,4646){\makebox(0,0)[r]{\strut{}$c_e$}}%
      \csname LTb\endcsname%%
      \put(1342,4426){\makebox(0,0)[r]{\strut{}$c_o$}}%
    }%
    \gplbacktext
    \put(0,0){\includegraphics{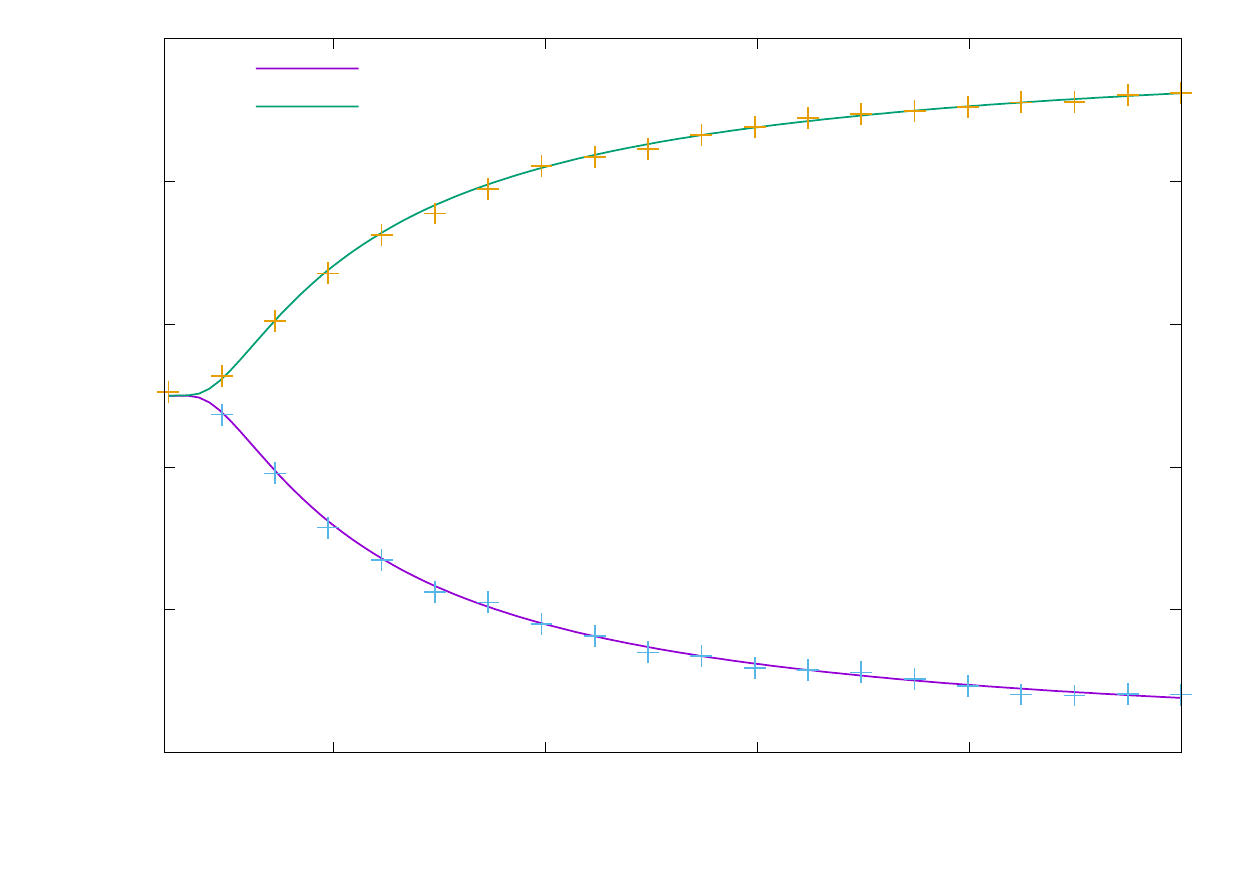}}%
    \gplfronttext
  \end{picture}%
\endgroup
\caption{\label{fig:fixed} The graph of $c_o$ and $c_e$ for $q>1$. 
Simulations were done for $n=1000$ and $n=1001$, each sampled $10000$ times.}
\end{figure}

As mentioned previously, when $q=1$ we retrieve the uniform distribution on $S_n$, for which 
the expected number of 1-cycles equals one.
So the fact that the limits for $q\downarrow 1$ of $c_e,c_o$ equal $1/2$ is pretty curious.
Of course there is no contradiction, since our results apply to the situation where $q>1$ is fixed and we send $n$ to infinity. 
Our results however do suggest something interesting must be going on 
in the ``phase change'' when $q=q(n)$ is a function of $n$ that approaches one from above as $n\to\infty$.

Our final (main) result highlights an interesting ``oscillating'' behaviour in the probability measures 
corresponding the limit of $C_1(\Pi_{2n})$, respectively $C_1(\Pi_{2n+1})$, when $q>1$.
The probability that $\Pi_n$ has at least $m$ one cycles is a lot larger when the parities of $m$ and $n$ agree
than when they don't (for $m$ large but fixed and $n\to\infty$).

\begin{thm}\label{thm:asymp}
For $0<q<1$ we have, as $k\to\infty$

$$ \begin{array}{rcl} 
\Pee\left( C_1( \rho \circ\Sigma) \geq 2k \right) & \ll & \Pee\left( C_1( r \circ\Sigma) \geq 2k \right), \\[2ex]
\Pee\left( C_1( \rho \circ\Sigma) \geq 2k+1 \right) & \gg & \Pee\left( C_1( r \circ\Sigma) \geq 2k+1 \right). 
\end{array} $$

\end{thm}

\noindent
(The notation $g(k) \ll f(k)$ means that the ratio $g(k)/f(k)$ tends to zero as $k\to\infty$.)

\subsection*{Sketches of some ideas used in the proofs.}

The proofs of Theorems~\ref{thm:normal} and Theorem~\ref{thm:clt_q_gt_one} are adaptations 
of a proof technique developed by Basu and Bhatnagar~\cite{basu2016limit} to prove a Gaussian 
limit law for the length of the longest monotone subsequence of a Mallows permutation, and in fact Theorem~\ref{thm:normal} closely follows the original proof.
The intuition behind it is that if $\Pi_n \sim\Mallows(n,q)$ with $0<q<1$ 
then given that $\Pi_n[ \{1,\dots,j\} ] = \{1,\dots,j\}$ the remainder 
of the permutation behaves like a Mallows random permutation of length $n-j$.
As it turns out, there will typically be linearly many such $j$.

A very rough sketch of the argument giving Theorem~\ref{thm:normal} is as follows.
If $T_1<\dots<T_k=n$ are such that $\Pi_n[ \{1,\dots,T_i\} ] = \{1,\dots,T_i\}$ 
then each cycle must lie completely in $\{T_{i-1}+1,\dots,T_i\}$ for some $i$
(setting $T_0=0$).
This allows us to show that the cycle counts behave approximately like a stopped two-dimensional random walk.
This refers to the situation where $(X_1,Y_1), (X_2, Y_2), \dots$ are i.i.d.~and we
are interested in $\sum_{i=1}^\tau Y_i$ where 
$\tau := \inf\{ k : X_1 + \dots + X_k > n \}$. 
Here the $X_i$ correspond to $T_i-T_{i-1}$ and $Y_i$ counts the number of 
cycles contained in the interval $\{T_{i-1}+1,\dots,T_i\}$.
A convenient result of Gut and Janson~\cite{GutStopped} allows us to derive that 
the mentioned sum is approximately Gaussian after suitable rescaling. 
The same argument applies to arbitrary linear combinations of the cycle counts, so that 
we can employ the Cramer-Wold device to deduce that the -- suitably rescaled -- vector of cycle counts is 
multivariate Gaussian.

The proof of Theorem~\ref{thm:clt_q_gt_one} goes along the same lines. 
Now it turns out that there are linearly many $T_1<\dots<T_k$ such that 
$\Pi_n[ \{1,\dots,T_i\} ] = \{n+1-T_i, \dots, n\}$ and 
$\Pi_n[ \{n+1-T_i, \dots, n\}] = \{1,\dots,T_i\}$.
(Almost) every even cycle must then be contained in some set 
$\{T_{i-1}+1,\dots,T_i\} \cup \{n+1-T_i, \dots, n-T_{i-1}\}$, and we
can adapt the proof strategy that gave Theorem~\ref{thm:normal} to work also here. We mention that Theorem \ref{thm:clt_q_gt_one} can also be proved by starting from the center of the permutation, rather than the sides, see \cite{HeCycles}.

In the proof of Theorem~\ref{thm:odd_cycles}, we rely on results of Gnedin and Olshanski~\cite{Gnedin} that show
that ``locally'', for $0<q<1$, the finite Mallows permutation resembles the bi-infinite Mallows permutation $\Sigma$ 
defined and analyzed in~\cite{Gnedin}.
An elementary, but crucial, observation is that $\Pi_n \isd \Mallows(n,q)$ if and only if $r_n\circ\Pi_n \isd \Mallows(n,1/q)$
where $r_n(i) := n+1-i$. (See the next section for the explanation.) 
%This allows us to relate the case of $q>1$ to $0<q<1$.
Note that if $n$ is odd then $r_n$ leaves $(n+1)/2$ invariant, but when $n$ is even no element of $\{1.\dots,n\}$ 
is invariant ($n/2$ and $n/2+1$ are flipped). 
For $q>1$, the relation with the $\Mallows(n,1/q)$ distribution translates to 
$\Pi_n \sim \Mallows(n,q)$ being ``approximated'' by $r \circ \Sigma$ with $r(i):=-i$ when $n$ is odd; and 
$\rho\circ\Sigma$ with $\rho(i) := 1-i$ when $n$ is even.
In particular we for instance have that the number of 1-cycles (fixed points) of $\Pi_n$ 
approximately behaves like the number of $i\in\Zed$ such that $\Sigma(i)=-i$ when $n$ is odd, and the 
number of $i\in\Zed$ such that $\Sigma(i) = 1-i$ when $n$ is even. 

For the proof of Theorem~\ref{thm:mi} we again use that the $\Mallows(n,q)$ distribution locally 
looks like the bi-infinite Mallows model. 
The main intuition is the elementary observation that the number of $i$-cycles 
equals $1/i$ times the number of points that are in $i$-cycles.
That $\sum i m_i = 1$ is then more or less immediate from the observation that, almost surely, all cycles of 
$\Sigma$ have finite length. The statements about the limits as $q\downarrow0$ and $q\uparrow 1$ can be derived
by using explicit expressions for the expected number of 1-cycles that follow by combining 
our work with results of Gnedin and Olshanski~\cite{Gnedin} and Gladkich and Peled~\cite{Peled}.

The idea behind the proof of Theorem~\ref{thm:mu2i} is very similar, but more technical. 
When $q>1$ then the $\Mallows(n,q)$ model
is approximated well locally by the composition of two independent bi-infinite Mallows models.

The first part of Theorem~\ref{thm:ceco} will follow from the aforementioned fact that, when $q>1$, the expected number
of 1-cycles is well-approximated by the number of $i\in\Zed$ such that $\Sigma(i)=-i$ when $n$ is odd, and the 
number of $i\in\Zed$ such that $\Sigma(i) = 1-i$ when $n$ is even.
The limit of $1/2$ for the expected number of $1$-cycles when $q\downarrow 1$ will follow from the 
relatively elementary observation that $1/q < \Pee( \Sigma(0)=j+1 ) / \Pee( \Sigma(0) = j ) < q$ for all $j \in\Zed$.
For the other limits we again analyze the various explicit expressions in $q$.

The proof of Theorem~\ref{thm:asymp} is technically involved, but the intuition behind it is relatively easy to explain.
When $k$ is large the ``most likely'' way in which $r \circ \Sigma$ will have at least $2k+1$ fixed points
is if $\Sigma(-k)=k, \dots, \Sigma(k)=-k$, or some minor perturbation of this configuration.
For $\rho\circ\Sigma$ the ``most likely'' way to have $2k+1$ fixed points is something 
like $\Sigma(-k)=1+k,\Sigma(-k+1)=k, \dots, \Sigma(k)=1-k$, or some minor perturbation of that situation.
However, as shown by Gnedin and Olshanski~\cite{Gnedin}, $\Sigma$ is almost surely balanced : 
the number of $i<0$ with $\Sigma(i)\geq 0$ is finite and equals the number of $i\geq 0$ with $\Sigma(i)<0$. 
This forces the existence of one more $i \in\Zed$ with $|\Sigma(i)-i| = \Omega(k)$ -- which makes the
probability exponentially smaller.
The intuition for $2k$ fixed points similar.

\begin{rem}
	Let us note that Theorems \ref{thm:normal}, \ref{thm:clt_q_gt_one}, and \ref{thm:odd_cycles} all hold for a more general class of permutation statistics.
	
	Let $w\in S_n$, and suppose that $w=w_1w_2$ where $w_1$ sends $[1,i]$ to itself and $[i+1,n]$ to itself. Say that a function $f:\cup _n S_n\to \mathbf{R}^d$ is \emph{additive} if $f(w)=f(w_1)+f(w_2)$ for all $w=w_1w_2$ decomposing in this way, where $f(w)$ for $w$ a permutation on an interval $[i,j]$ is defined by shifting down the permutation to $[1,j-i+1]$. Then Theorem \ref{thm:normal} holds for any (non-trivial) additive function satisfying $|f(w)|\leq Cn^k$ for $w\in S_n$, where $C$ and $k$ are constants. The proof is exactly the same, with the key being that $f(w)$ decomposes into a sum of independent pieces in the same way as the number of cycles. The other assumptions are to ensure that the moments are finite, and that the variance is non-zero.
	
	Let $w\in S_n$, and suppose that $w=w_1w_2$, where $w_1$ sends $[i,n-i]$ to itself and $w_2$ sends $[1,i]$ to $[n-i+1,n]$ and vice versa. Say that $f:\cup_n S_n\to\mathbf{R}^d$ is \emph{anti-additive} if $f(w)=f(w_1)+f(w_2)$ for all $w=w_1w_2$ decomposing in this way. Then Theorems \ref{thm:clt_q_gt_one} and \ref{thm:odd_cycles} holds for any (non-trivial) anti-additive function satisfying $|f(w)|\leq Cn^k$ for $w\in S_n$, where $C$ and $k$ are constants. Again, the proof is the same, writing $f(w)$ as a sum of independent pieces.
	
	Here, non-trivial means that $f$, when restricted to permutations of size $n$ for which no non-trivial decomposition of the form $w=w_1w_2$ exists, is non-constant. This assumption is needed to ensure that the variances $\beta_{ii}$ and $\beta_{ii}'$ are non-zero. Almost any reasonable function satisfies this, but note in particular that $f(w)=n$ the size of the permutation does not satisfy this.
	
	Note that if $f(w)$ is either additive or anti-additive, then so is $f(w^{-1})$. Thus, the theorems also apply to joint statistics for a Mallows permutation and its inverse, giving another proof of Theorem 1.2 of \cite{HeDescents}.
\end{rem}

\section{Notation and preliminaries\label{sec:prelim}}

Here we collect some notation, definitions and results from the literature that we will use in our proofs.
Throughout the paper we use $[n]:=\{1,\dots,n\}$ to denote the set consisting of the first $n$ natural numbers, and $[a,b]:=\{a,\dotsc,b\}$ for $a<b$.
If $f(n), g(n)$ are two functions depending on the parameter $n$, we will use 
$f(n) = o\left( g(n) \right)$ to denote that $f(n)/g(n) \to 0$, we will use 
$f(n) = O\left( g(n) \right)$ to denote that there exists a constant $C>0$ such that $f(n) \leq C\cdot g(n)$, 
we will use $f(n) = \Omega( g(n) )$ to denote 
that there exist a constant $c>0$ such that $f(n) > c \cdot g(n)$, and 
$f(n) = \Theta\left( g(n)\right)$ to denote that $f(n) = O( g(n) )$ and $f(n) = \Omega( g(n) )$.
We will use $\Bi(n,p)$ to denote the binomial distribution with parameters $n$ and $p$ and we use $\Geo(p)$ to the note the 
geometric distribution with parameter $p$.
So $X\sim\Geo(p)$ means that

$$\Pee(X=k) = p (1-p)^{k-1}, $$

\noindent 
for all $k\in\eN$. We use $\TGeo(n,p)$ to denote the {\em truncated} geometric distribution, truncated at $n$.
That is, if $Y \sim \TGeo(p)$ and $X \sim\Geo(p)$ then

$$ \Pee( Y = k ) = \Pee( X = k | X \leq n ) = \frac{p(1-p)^{k-1}}{1-(1-p)^n} \quad (k=1,\dots, n). $$

As is usual in the literature on the Mallows distribution, we denote by

$$ Z(n,q) := \sum_{\sigma\in S_n} q^{\inv(\sigma)}, $$

\noindent
the denominator in~\eqref{eq:Mallowsdef}. By a standard result in enumerative combinatorics (see Corollary 1.3.13 in~\cite{StanleyVol1}) 
we have

$$ Z(n,q) = \prod_{i=1}^n \frac{1-q^i}{1-q}. $$

% \noindent
% (An alternative, more probabilistic, way to see this identity is via the iterative procedure for 
% generating the Mallows random permutation $\Pi_n$ that we will describe shortly.)
An elementary observation is that the indices $i,j \in [n]$ form an inversion for $\pi \in S_n$ 
if and only if $\pi(i), \pi(j)$ form an inversion for $\pi^{-1}$.
In particular 

$$ \inv(\pi^{-1}) = \inv(\pi). $$

\noindent
Similarly, letting $r_n \in S_n$ denote the ``reversal map'' given by $r_n(i) := n+1-i$,
we have that $i,j \in [n]$ are an inversion in $\pi$ if and only if 
they are not an inversion in $r_n \circ \pi$. The same holds true for $\pi\circ r_n$.
In other words 

$$ \inv(r_n\circ\pi) = \inv(\pi\circ r_n) = {n\choose 2} - \inv(\pi), $$  

\noindent
and hence also 

$$ \inv(r_n\circ \pi \circ r_n ) = \inv(\pi). $$

\noindent
As a direct consequence of these observations and the definition of the Mallows probability measure, we have:

\begin{cor}
Let $q>1$ and $\Pi_n\sim\Mallows(n,q)$ and let $r_n$ be given by $r_n(i) = n+1-i$. The following hold.
\begin{enumerate}
 \item $\Pi_n^{-1} \isd \Pi_n$, and;
 \item $r_n\circ\Pi_n\circ r_n \isd \Pi_n$, and;
 \item $r_n\circ \Pi_n \isd \Mallows(n,1/q)$, and;
 \item $\Pi_n\circ r_n \isd \Mallows(n,1/q)$.
\end{enumerate}
\end{cor}

\noindent
(To see the third and fourth parts of the lemma, note
that $\Pee( r_n\circ\Pi_n = \pi ) = \Pee( \Pi_n = r_n \circ \pi ) 
= q^{{n\choose 2}-\inv(\pi)} / Z(n,q)$ is proportional
to $(1/q)^{\inv(\pi)}$, and similarly for $\Pi_n \circ r_n$.)
The last two parts of the corollary provide a way to 
express the Mallows distribution with $q>1$ in terms of the Mallows distribution with $0<q<1$. 
We will rely on this a lot in our proofs of the results for $q>1$.

For $0<q<1$, there is an iterative procedure for generating $\Pi_n \in \Mallows(n,q)$, going 
back to the work of Mallows~\cite{Mallows}.
We let $Z_1,\dots,Z_n$ be independent with $Z_i \sim \TGeo(n+1-i,1-q)$. %% (so in particular $Z_n = 1$ almost surely).
We now set %
\begin{align}\label{eq:MallowsIter}	
\Pi_n(1) &= Z_1,\\
\Pi_n(i)&= \text{the $Z_i$-th smallest number in the set }[n]\setminus \{ \Pi(1),\ldots, \Pi(i-1)\}, \qquad \text{ for $1<i\leq n$}.
\end{align}%

Put differently, having determined  $\Pi_n(1),\ldots, \Pi_n(i-1)$, we determine $\Pi_n(i)$ by writing
$[n] \setminus \{\Pi_n(1),\ldots, \Pi_n(s-1)\}$ in increasing order as $\{j_1 , j_2, \ldots , j_{n-s+1} \}$, and setting 
set $\Pi_n(i) := j_{Z_i}$. %
To see that this procedure indeed generates a random element of $S_n$ chosen according to the $\Mallows(n,q)$ distribution, 
we can argue as follows. 
We first note that for each $\pi\in S_n$ there is exactly one choice of $(k_1,\dots,k_n)\in[n]\times[n-1]\times\dots\times[1]$ 
such that setting $Z_1=k_1,\dots,Z_n=k_n$ results in $\Pi_n = \pi$
(and vice versa each choice of $k_1,\dots,k_n$ determines a unique element of $S_n$). 
In particular 

$$ \Pee( \Pi_n = \pi ) = \prod_{i=1}^n \Pee( Z_i = k_i ) \propto q^{\sum_{i=1}^n (k_i-1)}, $$

\noindent
where the symbol $\propto$ denotes ``proportional to'', and hides a multiplicative term not depending on $k_1,\dots,k_n$.
We now note that for each $i \in \eN$ we must have% 
$$ \left| \{ j : i < j \text{ and } \pi(i) > \pi(j) \}\right| = k_i - 1. $$
%
%in $\pi$ there must be precisely $k_i-1$ inversions involving the element $i$ and one of 
%the larger elements $i+1,\dots,n$. 
In other words, $\inv(\pi) = k_1+\dots+k_n-n$, which shows that we've indeed sampled according to the $\Mallows(n,q)$-distribution.
% \noindent
%A two--sided adaptation of the above sampling procedure is given in Section \ref{sec:proof_q_qt_one}. \par

\begin{figure}
	\centering
	\includegraphics{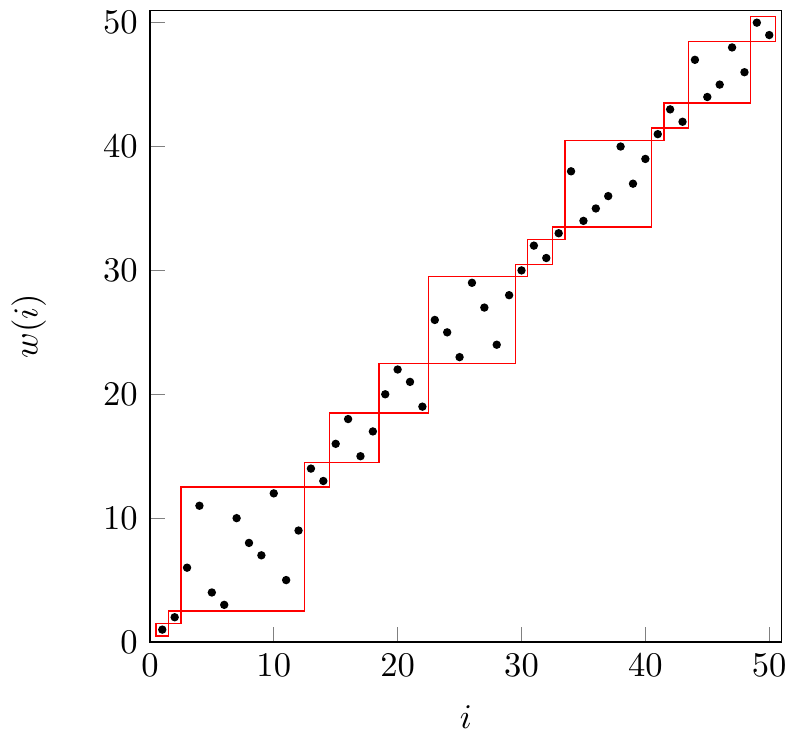}
	\caption{Sample of $\Mallows(\eN,q)$. The red squares indicate the first time that an interval is sent to itself, and the lengths and contents of the red squares are independent and identically distributed.}
	\label{fig:mallows process}
\end{figure}

There is a natural extension of the Mallows model to random functions $\Pi : \eN\to\eN$, called the 
{\em Mallows process} by some authors. 
Similarly to the iterative procedure for generating $\Pi_n\sim\Mallows(n,q)$ described above, we 
let $Z_1,Z_2,\ldots$ be an infinite sequence of i.i.d. $\on{Geom}(1-q)$ random variables
and we iteratively construct an infinite sequence $\Pi$ of natural numbers by setting
\begin{align}\label{eq:akjkjfekew}	
\Pi(1) &= Z_1,\\
\Pi(i)&= \text{the $Z_i$-th smallest number in the set }\bN \setminus \{ \Pi(1),\ldots, \Pi(i-1)\}, \qquad \text{ for $i > 1$}.
\end{align}
We denote the probability distribution of $\Pi$ generated in this manner by $\on{Mallows}(\bN, q)$ (see also Figure \ref{fig:mallows process}).

For a non-empty, finite $A \subseteq \eR$ and a bijection $\pi : A \to A$ we define $\inv(\pi)$ in exactly the 
same way as for bijections from $[n]$ to itself. The distribution of the random bijection $\Pi_A : A \to A$ satisfying 

$$ \Pee( \Pi_A = \pi ) = \frac{q^{\inv(\pi)}}{Z(n,q)}, $$

\noindent
will be denoted by $\Mallows(A,q)$. 
For $\sigma : B \to B$ a bijection with $B \subseteq \eR$ and $A \subseteq B$ finite (but $B$ may be infinite) we
denote by $\sigma_A$ the bijection we obtain by setting 

$$ \sigma_A(a) := a_{(i)} \text{ if $\sigma(a)$ is the $i$-th smallest element of $\sigma[A]$, } $$

\noindent
where $a_{(i)}$ is the $i$-th smallest element of $A$.

As shown by Basu and Bhatnagar (\cite{basu2016limit}, Lemma 2.1) and independently Crane and DeSalvo (\cite{CraneDesalvo2017}, Lemma 5.2), 
if $\Pi \sim \Mallows(\eN,q)$ with $0<q<1$ and $I = \{a,\dots,a+n\} \subseteq \eN$ 
is a finite ``interval'' of consecutive integers then

$$ \Pi_{I} \isd \Mallows(I,q). $$

\noindent
Moreover, as is easily seen from the definitions, if $I= \{a+1,\ldots, a+n\} \subseteq \Zed$ is such an interval 
then

$$ \Sigma \sim \on{Mallows}(I,q) \quad \text{ if and only if } \quad 
s^{(a)} \circ \Sigma \circ s^{(-a)} \isd \on{Mallows}(n,q). $$

\noindent 
Here and in the rest of the paper $s$ denotes the {\em shift map} given by $i\mapsto i+1$ and 
$f^{(n)}$ denotes the $n$-fold composition of the function $f$ with itself
and $f^{(-n)}$ denotes the $n$-fold composition of the inverse $f^{-1}$ of $f$ with itself.
(So in particular $s^{(k)}$ is the map $i \mapsto i+k$ and $s^{(-k)}$ is the map $i\mapsto i-k$.)

We next discuss Gnedin and Olshanski's bi-infinite Mallows model.
For every $0<q<1$, this is a random bijection of $\Zed$, whose distribution we will denote by $\Mallows(\Zed,q)$.
The work of Gnedin and Olshanski provides several definitions, but all of them are
rather involved. So we refer the reader to the original paper~\cite{Gnedin} for the precise definition and mention
only the properties and relevant facts we will be using in what follows.

We will need the following notion of convergence. For a sequence $\sigma_1 : I_1\to I_1,\sigma_2: I_2\to I_2,\ldots$ of bijections
of subsets of $\bZ$ with the property that $\Zed = \bigcup_n I_n$, we write 
$\sigma_n \to \sigma$ if for every $i\in \bZ$ there is an $n=n(i)$ such that $n'\geq n$ implies $\sigma_n(i)=\sigma(i)$. 

% Again for $\sigma :\Zed \to \Zed$ a bijection and $I \subseteq \Zed$ finite, we define $\sigma_I$ 
% to be the permutation of $I$ that inherits the relative order of $\sigma$ on $I$.

As in \cite{Gnedin}, we call a permutation $\pi$ of $\bZ$ balanced if 
\begin{equation}
	|\{   (i ,\pi(i)) \mid i < 1 \leq \pi(i)    \}| = |\{   (i ,\pi(i)) \mid \pi(i) < 1 \leq i    \}| < \infty.
\end{equation}
As noted in \cite{Gnedin}, we may replace the $1$ above by any fixed number in $\bZ$ and obtain an equivalent 
definition of balanced permutations. 
The random permutation $\Sigma\sim\on{Mallows}(\bZ, q)$ is almost surely balanced.

The $q$--Pochhammer symbols $(a;q)_n$ and $(a;q)_{\infty}$ are defined as
\begin{align}\label{eq:pochhammer}
	(a;q)_n = \prod_{i=1}^n (1- aq^{i-1})\qquad\text{and} \qquad	(a;q)_\infty &= \prod_{i=1}^\infty (1- aq^{i-1}),
\end{align}

\noindent
Recall that, thoughout the paper, we will use $s,r,\rho$ to denote the maps given by

$$ s(i) := i+1, \quad r(i) := -i, \quad \rho(i) := 1-i. $$

\noindent
The following lemma lists the facts on the $\Mallows(\Zed,q)$ distribution we will be relying on in our proofs below.

\begin{lem}\label{lem:gnedin}
Let $0<q<1$ and $\Sigma \sim \on{Mallows}(\bZ, q)$. We have
\begin{enumerate}
		% \item \label{item:gnedin_exch} The distribution $\on{Mallows}(\bZ, q)$ is $q$--exchangeable, (\cite{Gnedin}, Thm 3.3);
\item\label{itm:gnedin_Z_finite} If $I\subseteq\Zed$ is a finite set of consecutive integers then $\Sigma_{I} \isd \Mallows(I,q)$
(\cite{Gnedin}, comments below Theorem 6.1);
\item\label{item:gnedin_converge} If $I_1\subseteq I_2\subseteq ..$ are finite sets of consecutive integers with 
$\bigcup_n I_n = \bZ$, then 
$\Sigma_{I_n} \to \Sigma$ almost surely (\cite{Gnedin}, Proposition 7.6); 
\item \label{item:gnedin_inverse}$\Sigma^{-1}\isd \on{Mallows}(\bZ, q)$ (\cite{Gnedin}, Corollary 3.4);
\item \label{item:gnedin_shift}$s \circ \Sigma \circ s^{-1} \isd \on{Mallows}(\bZ, q)$
(\cite{Gnedin}, Lemma 4.4);
\item \label{item:gnedin_reflect} $\rho\circ\Sigma\circ\rho \isd \Mallows(\Zed,q)$ (\cite{Gnedin}, Corollary 3.5);
\item \label{item:gnedin_disp} $\bP \left[ | \Sigma(0)| > m \right] = \Theta(q^m)$ 
(\cite{Gnedin}, Remark 5.2);
\item\label{itm:gnedin_displacements} For each $d \in \Zed$ we have 

$$ \Pee\left[ \Sigma(0)=d  \right] = (1-q) (q;q)_\infty \sum\limits_{r,l\geq 0;\atop r-l = d} \frac{q^{rl + r + l}}{ (q;q)_r (q;q)_l }. $$

(\cite{Gnedin}, Theorem 5.1).
\end{enumerate}
\end{lem}

\noindent
We remark that it follows from Part~\ref{item:gnedin_shift} that 
also 

$$s^{(k)} \circ \Sigma \circ s^{(-k)} \isd \Sigma, $$%
for all $k \in \Zed$. It now also follows that

$$ \Sigma(i)-i \isd \Sigma(0), $$ 

\noindent
for all $i \in \Zed$.
Similarly, since $\rho = r \circ s^{-1} = s \circ r$, we can combine~\ref{item:gnedin_shift} and~\ref{item:gnedin_reflect} to derive that  

$$ r \circ \Sigma \circ r = s^{-1} \circ \rho \circ \Sigma \circ \rho \circ s
\isd s^{-1} \circ \Sigma \circ s \isd \Sigma. $$
Finally, let us also remark that the $\Mallows(\Zed,q)$ process can be thought of as a stationary version of the $\Mallows(\eN,q)$ process, as described in \cite{PitmanTang}. In particular, we will use the following result on renewal processes, applied to the Mallows process (\cite{G13}, Equation 5.72).

\begin{prop}
    \label{prop: waiting time}
    Let $X_i$ be a sequence of independent and identically distributed random variables taking values in $\eN$, with $\Ee X_i=\mu>0$, and let $\tau(n)=\inf\{t\colon \sum_{i=1}^t X_i>n\}$. Then $X_{\tau(n)}\to X^*$ in distribution, where $X^*$ has the size-bias distribution of the $X_i$, which means $\Pee(X^*=x)=x\Pee(X_i=x)/\mu$.
\end{prop}

We will also use some of the tools developed in \cite{Peled}. 
In particular we will use the arc chain $\{\kappa_t \}_{t= 0}^n$ corresponding to $\Pi_n\sim\Mallows(n,q)$ with $0<q<1$, 
defined by
\begin{equation}
	\kappa_t := |\{i \in [t] \mid \Pi_n(i) > t \}|.
\end{equation}
We speak of the $(n,q)$ arc chain. We have the following.

\begin{lem}[\cite{Peled}, Proposition 3.3]\label{lem:peled_kappa}
Let $\Pi_n \sim \on{Mallows}(n,q)$ with $0<q<1$. The arc chain $(\kappa_t)_{t=0,\dots,n}$ of $\Pi_n$ is a time-inhomogeneous Markov chain
with transition probabilities

$$ \Pee( \kappa_{t+1} = j | \kappa_t=k ) =
\begin{cases}
 \left(  \frac{1-q^{k}}{1 - q^{n-t}} \right)^2 & \text{ if $j=k-1$, } \\[2ex]
 \frac{q^{k} - q^{n-t}}{1 - q^{n-t}}\cdot \frac{ 2 - q^{k} - q^{k + 1} }{ 1-q^{n-t} } & \text{ if $j=k$, }\\[2ex]
 \frac{q^{k} - q^{n-t}}{1 - q^{n-t}} \cdot \frac{q^{k+1} - q^{n-t}}{1 - q^{n-t}} & \text{ if $j=k+1$, } \\[2ex]
 0 & \text{ otherwise.}
\end{cases}.
$$
% 
% \begin{align}
% \Pee\left[ \kappa_{t+1} = k - 1\,|\,\kappa_t=k \right] &= \left(  \frac{1-q^{k}}{1 - q^{n-t}} \right)^2,\\ 
% \label{eq:kappa_transition}\Pee\left[ \kappa_{t+1} = k \,| \kappa_t=k \right] &=  
% \frac{q^{k} - q^{n-t}}{1 - q^{n-t}}\cdot \frac{ 2 - q^{k} - q^{k + 1} }{ 1-q^{n-t} }  ,\\ 
% \Pee\left[ \kappa_{t+1} = k + 1\,|\,\kappa_t=k \right] &= \frac{q^{\kappa_t} - q^{n-t}}{1 - q^{n-t}} \cdot \frac{q^{\kappa_t+1} - q^{n-t}}{1 - q^{n-t}}.
% 	\end{align}
\end{lem}

\begin{lem}[\cite{Peled}, Lemma 3.4]\label{lem:prop_fixed_kappa}
	Let $\Pi_n \sim \on{Mallows}(n,q)$ and $\kappa$ be its arc chain. Then
	\begin{equation}
		\bP [\Pi_n(t+1) = t+1 \ | \kappa_t=k] = 
		\frac{   (q^{k} - q^{k+1} )(q^{k} - q^{n-t}) }{ \left( 1-q^{n-t} \right)^2  },\qquad 0\leq t\leq n.
	\end{equation}
\end{lem}

\noindent
(We mention that we have slightly adapted the statements from~\cite{Peled} in the above two lemmas.)

Analogously to the arc chain $\kappa_t$ for $\Pi_n\sim\Mallows(\eN,q)$, we can define the arc-chain for 
$\Pi\sim\Mallows(\eN,q)$ with $0<q<1$ setting 

$$ \hat\kappa_t := |\{ i \in [t] : \Pi(i) > t \}|. $$

\noindent
We speak of the $(\infty,q)$-arc chain.
It is straightforward to verify 
that $(\hat\kappa_t)_{t\geq 0}$ forms a Markov chain with the following transition probabilities.

$$ \Pee( \hat{\kappa}_{t+1} = j | \hat{\kappa}_t = k )
= \begin{cases}
   (1 - q^{k})^2 & \text{ if $j=k-1$, } \\
   2q^{k} -q^{2 k} - q^{2 k + 1} & \text{ if $j=k$, } \\
   q^{2k + 1} & \text{ if $j=k+1$, } \\
   0 & \text{otherwise.}
  \end{cases}
$$

\noindent
Alternatively, this can be seen by combining Lemma~\ref{lem:peled_kappa} with Lemma~\ref{lem:PinPicpl} 
below. As a side remark, we mention that Gladkich and Peled~\cite{Peled} 
{\em defined} the $(\infty,q)$-arc chain directly via the transition probabilities given here, without 
explicitly mentioning the connection to the definition we give here.

As explained in~\cite{Peled} (Section 3.2), the Markov chain $\hat{\kappa}_t$ has a 
unique stationary distribution $\nu$ given by%
\begin{equation}\label{eq:nu}
\nu_s := \frac{  \prod_{i=1}^s  \frac{q^{2i-1}}{ (1-q^i)^2  }    }{ \sum_{t\geq 0 } \prod_{i=1}^t  \frac{q^{2i-1}}{ (1-q^i)^2  } }.
\end{equation}%
Here, as usual, an empty product equals $1$. 

The following result provides a useful link between the $(n,q)$ and the $(\infty,q)$ arc chains.

\begin{prop}[\cite{Peled}, Proposition 3.8, ]\label{prop:convergence_kappa}
	Set $t=t(n)$. 
	If both $t\to \infty$ and $n-t \to \infty$ then the law of $\kappa_t$ converges 
	to the stationary distribution $\nu_s$ as $n$ tends to infinity with $q$ fixed.
\end{prop}

Given two discrete probability distributions $\mu_1$ and $\mu_2$ on a countable set $\Omega$, their total variation distance 
is defined as

$$ d_{\on{TV}}(\mu_1,\mu_2) = \max_{ A \subseteq \Omega  }| \mu_1(A) - \mu_2(A) |. $$

\noindent
A useful alternative expression is 

\begin{equation}\label{eq:dtvalt} 
d_{\on{TV}}(\mu_1,\mu_2) = \frac12 \sum_{x\in \Omega} |  \mu_1(x) - \mu_2(x) |
= \sum_{x : \mu_1(x)>\mu_2(x) } \mu_1(x)-\mu_2(x). 
\end{equation}

\noindent
(For a proof, see for instance Proposition 4.2 in \cite{Peres}.)
As is common, we will interchangeably use the notation 
$\dtv(X,Y) := \dtv(\mu,\nu)$ if $X\sim\mu_1$ and $Y\sim\mu_2$.

A {\em coupling} of two probability measures $\mu, \nu$ is a joint probability measure for a pair of random variables $(X,Y)$ 
satisfying $X\isd\mu, Y\isd\nu$.
We will also speak of a coupling of $X,Y$ as being a probability space for $(X',Y')$ with $X'\isd X, Y'\isd Y$.
Another useful characterization of the total variational distance is as follows.

\begin{lem}\label{lem:coupling}
Let $\mu$ and $\nu$ be two probability distributions on the same countable set $\Omega$. Then
\begin{equation}
d_{TV}(\mu, \nu) = \inf \{ \bP \left[ X\neq Y\right] \mid (X,Y) \text{ is a coupling of $\mu$ and $\nu$}  \}.
\end{equation}
There is a coupling that attains this infimum. 
\end{lem}

\noindent
(For a proof, see for instance~\cite{Peres}, Proposition 4.7 and Remark 4.8.)

For the proofs of the normal limiting laws in Theorem~\ref{thm:normal} and Theorem~\ref{thm:clt_q_gt_one} we
will make use of a result on stopped two-dimensional random walks by Gut and Janson~\cite{GutStopped}
that seems tailor made for our purposes.
Here we consider an i.i.d. sequence $(X_1,Y_1),(X_2,Y_2),\ldots$ 
and for $t>0$ we define $\tau(t)$ as the first $k$ such that $X_1+\dots+X_k$ exceeds $t$:%
\begin{equation}\label{eq:taundef} 
\tau(t) = \inf \{ k \geq 1 : X_1+\dots+X_k > t\}. 
\end{equation}%
\noindent
The result of Gut and Janson we'll use states that:

\begin{thm}[\cite{GutStopped}, Theorem 3]\label{thm:regen_clt}
Let $(X_1,Y_1), (X_2,Y_2),\dots$ be an i.i.d.~sequence and let
$\tau(n)$ be as given by~\eqref{eq:taundef}.
Suppose that $\Ee X_1>0$, that $\Var(X_1), \Var(Y_1)<\infty$ and $\on{Var}(  Y_1\bE X_1 -  X_1  \bE Y_1 ) >0$. Then 
	\begin{equation}
		\frac{   \sum_{i=1}^{\tau(t)} Y_i - (\frac{\bE Y_1}{\bE X_1})t}{  \sqrt{t}  } \quad 
		\xrightarrow[t\to\infty]{d}\quad  
		\cN\left(0,\frac{\Var(  Y_1\bE X_1 -  X_1  \bE Y_1 )}{(\bE X_1)^{3}}\right) .
	\end{equation}
\end{thm}

For the proofs of the normal limiting laws in Theorem~\ref{thm:normal} and Theorem~\ref{thm:clt_q_gt_one} it will 
be convenient to use the Cramer-Wold device.
A proof can for instance be found in~\cite{Billingsley} (Theorem 29.4).

\begin{thm}[Cramer-Wold device]\label{thm:cramer_wold}
	For random vectors $X_n = (X_{n,1},\ldots, X_{n,k})$ and $Y = (Y_1,\ldots, Y_k)$, a necessary and sufficient 
	condition for $X_n \overset{d}{\longrightarrow} Y$ is that 
	$\sum_{u=1}^k t_u X_{n,u} \overset{d}{\longrightarrow} \sum_{u=1}^k t_u Y_u$ for each $(t_1,\ldots, t_k)\in \bR^k$. 
\end{thm}

Several times, we are going to rely on the following result of Basu and Bhatnagar~\cite{basu2016limit}.

\begin{lem}[\cite{basu2016limit}, Lemma 5.5]\label{lem:max_iid}
	Let $W_1,W_2,\ldots$ be an i.i.d. sequence of random variables with $\bE W_i^2 <\infty$. Then
	\begin{equation}
		\frac{\max_{1\leq i \leq n} W_i}{\sqrt{n}} \quad \to\quad  0
	\end{equation}
	in probability.
\end{lem}

We will also make use of the following fact. Even though it seems pretty standard we have not been able to find a convenient reference.
We therefore provide a short proof.

\begin{lem}\label{lem:MCret}
Suppose that $(W_t)_{t\geq 0}$ is a Markov chain with state space $\{0\}\cup\eN$, started in state $W_0=0$, 
and whose transition probabilities satisfy $p_{i,j} = 0$ if and only if $|i-j|\neq 1$ and 
$\displaystyle \liminf_{i\to\infty} p_{i,i-1} > 1/2$.
Let 

$$ T := \inf\{ t \geq 1 : W_t = 0 \}. $$

\noindent
Then $\Ee T^k < \infty$ for all $k \in \eN$.
\end{lem}

\begin{proof}
Let $T_i$ denote the number of steps to reach $i-1$, in the chain starting from $W_0=i$.
Let $i_0 \in \eN, p>1/2$ be such that $p_{i,i-1} > p$ for all $i\geq i_0$.
We have 

$$ \Pee( T_i > t ) \leq \Pee( \Bi( t, p ) \leq t/2 ) = \exp\left[ - \Omega(t) \right], $$

\noindent
for each $i\geq i_0$, using the Chernoff inequality (see for instance~\cite{PurpleBook}, Corollary 2.3).
This implies

$$ \Ee T_i^k \leq \sum_t t^k\cdot\Pee( T_i \geq t ) < \infty, $$

\noindent
for all $i \geq i_0$ and $k\in \eN$.

Starting from $W_0=0$, we of course move to state 1 with probability one in the first step, 
giving

$$ \Ee T^k = \Ee\left(1+T_1\right)^k = \sum_{j=0}^k {n\choose j} \Ee T_1^j, $$

\noindent
for all $k \in \eN$. In particular, it suffices to show $\Ee T_1^k < \infty$ for all $k \in \eN$.

Similarly, by considering the first step of the chain we see for each $i\geq 1$ and $k\in\eN$:

\begin{equation}\label{eq:eenhoorn} \begin{array}{rcl} 
\Ee T_i^k 
& = & \displaystyle 
p_{i,i-1} + p_{i,i-1} \cdot \Ee\left( 1 + T_i' + T_{i+1}' \right)^k\\[2ex]
& = & \displaystyle 
p_{i,i-1} + p_{i,i+1} \cdot \sum_{0\leq k_1,k_2\leq k, \atop k_1+k_2 \leq k} 
{k \choose k_1,k_2,k-(k_1+k_2)} \Ee T_i^{k_1} \Ee T_{i+1}^{k_2}, 
\end{array} \end{equation}

\noindent
where we take $T_i'\isd T_{i-1}, T_{i+1}'\isd T_{i+1}$ independent in the first line.
(To see the first inequality, note that in the first step we move to $i-1$ with probability $p_{i,i-1}$. 
If, on the other hand, we move to $i+1$ in the first step then we first have to wait until we reach state $i$ again, 
and then we have to wait until we reach $i-1$ from $i$).
Rewriting~\eqref{eq:eenhoorn}, we obtain

$$ \Ee T_i^k = \frac{1}{1-p_{i,i-1}} \cdot 
\left( p_{i,i-1} + p_{i,i+1} \cdot \sum_{{0\leq k_1,k_2\leq k, \atop k_1+k_2 \leq k,}\atop k_1\neq k} 
{k \choose k_1,k_2,k-(k_1+k_2)} \Ee T_i^{k_1} \Ee T_{i+1}^{k_2}\right). $$

\noindent
We can thus apply induction on $k$ to show that $\Ee T_{i_0-1}^k < \infty$ for all $k \in \eN$.
Repeating the argument, we also have $\Ee T_i^k < \infty$ for $i=i_0-2, i_0-3,\dots$ and so on until $i=1$.
\end{proof}

\section{The proof of Theorem~\ref{thm:normal}} 

We define a sequence of {\em regeneration times }$T_0 <T_1 < T_2<\ldots$ as follows:
\begin{align}
	T_0 &:=0,\\
	T_i &:=  \inf\{ j > T_{i-1} \text{ s.t. }  \Pi( [j] )=[j] \} \qquad (i=1,2,\dots),.
\end{align}

In Section 4 of \cite{basu2016limit}, Basu and Bhatnagar show that $T_1$ has finite second moment. 

\begin{lem}\label{lem:Rplus} $\Ee T_1^2 < \infty$.
\end{lem}

\noindent
(We combine Lemmas 4.1 and 4.5 of~\cite{basu2016limit}.)
We also define the interarrival times

$$ X_i = T_i - T_{i-1}, $$ 

\noindent
and 

$$ \tau(n) = \inf\{ t : T_t > n \}. $$ 

\noindent 
Looking at the description of $\Pi$ in Section~\ref{sec:prelim}, it is not difficult to see that 
conditional on the event $T_1 = t$, the bijection $i \mapsto \Pi(i+t)-t$ is distributed like $\Pi$.
It follows that the interarrival times $X_1, X_2, \dots$ are i.i.d.
Moreover, writing $\Xcal_i := \{T_{i-1}+1,\dots,T_i\}$ we see that 
$\Pi$ maps $\Xcal_i$ bijectively onto $\Xcal_i$, and in fact the permutations $\Sigma_1 : [X_1]\to[X_1],\Sigma_2: [X_2]\to[X_2], \dots$ 
given by 

$$ \Sigma_i(j) := \Pi(T_{i-1}+j) - T_{i-1} \quad \text{ for $j=1,\dots,X_i$, } $$

\noindent
are i.i.d.~as well. %%(To be completely clear: the domain of $\Sigma_i$ is itself random.)

With this regenerative structure, the following lemma follows.

\begin{lem}\label{lem:PinPicpl}
	For $0<q<1$ and $n\in\eN$, let $\Pi_n \sim \Mallows(n,q)$ and $\Pi\sim\Mallows(\eN,q)$.
	There exists a coupling of $\Pi_n$ and $\Pi$ satisfying
	
	$$ \Pee( \Pi_n(i) = \Pi(i) \text{ for all $1\leq i \leq n-\log^2 n$ } ) = 1 - o(1),$$
	and in fact, $\log^2 n$ can be replaced with any function going to $\infty$ with $n$.
\end{lem}

\begin{proof}
	Let $\Pi\sim \Mallows(\eN,q)$, and recall that $\Pi_{[n]}\sim \Mallows(n,q)$. We claim that this coupling works. Indeed, it's clear that $\Pi(i)=\Pi_{[n]}(i)$ up until $T_{\tau(n)-1}$. But $n-T_{\tau(n)-1}\leq X_{\tau(n)}$, and $X_{\tau(n)}$ converges to a limiting distribution by Proposition \ref{prop: waiting time}. In particular, the probability that it is larger than $\log^2 n$ (or any function going to $\infty$) goes to $0$.	
\end{proof}

By this last lemma, with probability $1-o(1)$, the 
number of $i$-cycles in $\Pi_n$ differs by at most $2\log^2 n$ 
from the number of $i$-cycles of $\Pi$ that are completely contained in $[n]$ (for each $i=1,\dots,\ell$).

Fix an $\ell \in \bN$, and let $a_1,\ldots, a_\ell$ be a sequence of real numbers, not all zero. 
For $\pi$ a permutation, we define $\varphi(\pi) := \sum_{j=1}^\ell a_j C_j(\pi)$ 
and let $Y_i=\varphi(\Sigma_i)$.

We plan to apply Theorem~\ref{thm:regen_clt} to the i.i.d.~sequence $(X_1,Y_1), (X_2,Y_2), \dots$.
For this we first need to establish the conditions of that theorem are met.

\begin{lem}\label{lem:other_conditions}
We have $\Ee X_1 > 0$ and $\Var(X_1), \Var(Y_1) < \infty$. 
\end{lem}

\begin{proof}
Since $X_1 \geq 1$ by definition, we trivially have $\Ee X_1 > 0$.
As $|\varphi(\Sigma_1)| \leq  \max_i |a_i| \cdot X_1$, it suffices to show that $\on{Var}(X_1) < \infty$.
But this has already been established by Lemma~\ref{lem:Rplus}.
\end{proof}

\begin{lem}\label{lem:varZt}
$\on{Var}(Y_1 \Ee X_1  -  X_1\Ee Y_1) >0$.
\end{lem}

\begin{proof}
We first note that, for each $i \in \eN$, there is a positive
probability that $\Sigma_1$ consists of a single $i$-cycle.
(This happens for instance when $Z_1=i, Z_2=1, \dots Z_i=i-1$.)

Aiming for a contradiction, assume that $Y_1 \Ee X_1  -  X_1\Ee Y_1$ is almost surely constant.
Whenever $\Sigma_1$ consists of a single $>\ell$ cycle, we have $Y_1=0$.
In particular $Y_1 \Ee X_1  -  X_1\Ee Y_1$ can equal both $(\ell+1) \Ee Y_1$ and $(\ell+2) \Ee Y_1$
with positive probability. The quantity $Y_1 \Ee X_1  -  X_1\Ee Y_1$ being an almost sure constant
now implies $\Ee Y_1 = 0$.

Let $1\leq i \leq \ell$ be such that $a_i \neq 0$. 
There is a positive probability that $Y_1 \Ee X_1 = 0$
and a positive probability that $Y_1 \Ee X_1 = a_i \Ee X_1$.
But that implies $a_i = 0$, contradicting the choice of $i$.

It follows $Y_1 \Ee X_1  -  X_1\Ee Y_1$ is not almost surely constant. 
In other words, $\Var( Y_1 \Ee X_1  -  X_1\Ee Y_1 ) > 0$.
\end{proof}

Having established Lemmas~\ref{lem:other_conditions} and~\ref{lem:varZt}, we can apply 
Theorem~\ref{thm:regen_clt} to conclude that 

$$ \frac{ \sum_{i=1}^{\tau(n)} Y_i -  \left(\frac{\bE Y_1}{ \bE X_1}\right)  n }{ \sqrt{n} } \quad 
\xrightarrow[n\to\infty]{d}\quad \cN(0,  (\bE X_1)^{-3}  \on{Var}( Y_1 \bE X_1  -  X_1\bE Y_1)). $$

By Lemma~\ref{lem:PinPicpl} and the definition of $\tau(n)$ we have

\begin{equation}\label{eq:C_k_squeeze} \left| \varphi(\Pi_n) - \sum_{j=1}^{\tau(n)}Y_j\right| \leq 
\max_{i=1,\dots,\ell} |a_i| \cdot \left(  X_{\tau(n)} + \log^2(n) \right),
\end{equation}

\noindent
with probability $1-o(1)$, under the coupling provided by Lemma~\ref{lem:PinPicpl}.
Moreover, applying Lemmas~\ref{lem:max_iid} and~\ref{lem:Rplus} we have that, with probability $1-o(1)$, the RHS of~\eqref{eq:C_k_squeeze}
is $o(\sqrt{n})$.
We can conclude:

$$ \frac{ \varphi(\Pi_n) -  \left(\frac{\bE Y_1}{ \bE X_1}\right)  n }{ \sqrt{n} } \quad 
\xrightarrow[n\to\infty]{d}\quad 
\cN\left(0, \frac{\on{Var}\left( Y_1\bE X_1 - X_1\bE Y_1\right)}{\left(\Ee X_1\right)^3} \right). $$

\noindent
Recalling that $Y_1 = \sum_{i=1}^\ell a_i C_i(\Sigma_1)$ and setting $m_i := \frac{\Ee C_i(\Sigma_1)}{\Ee X_1}$, we can write 

$$ \frac{ \varphi(\Pi_n) -  \left(\frac{\bE Y_1}{ \bE X_1}\right)  n }{ \sqrt{n} }
= a_1 \frac{C_1(\Pi_n) - m_1 n}{\sqrt{n}} + \dots 
+ a_\ell \frac{C_\ell(\Pi_n) - m_\ell n}{\sqrt{n}}. $$

\noindent
Setting 

$$ U_i := \frac{C_i(\Sigma_1)\Ee X_1 - X_1\Ee C_i(\Sigma_1)}{\left(\Ee X_1\right)^{3/2}}, $$

\noindent 
we see that 

$$ \frac{Y_1\Ee X_1-X_1\Ee Y_1}{\left(\Ee X_1\right)^{3/2}} = a_1 U_1 + \dots + a_\ell U_\ell. $$

Therefore, if we set $P_{ij} := \Cov( U_i, U_j )$ then 
%$$ P_{ij} := \Cov\left( \frac{C_i(\Sigma_1)\Ee X_1 - X_1\Ee C_i(\Sigma_1)}{\left(\Ee X_1\right)^{3/2}}, 
%\frac{C_j(\Sigma_1)\Ee X_1 - X_1\Ee C_j(\Sigma_1)}{{\left(\Ee X_1\right)^{3/2}}}\right), $$
%
%\noindent
%we see that 

$$ \frac{\on{Var}\left( Y_1 \bE X_1  -  X_1\bE Y_1\right)}{\left(\Ee X_1\right)^3}  = \sum_{i=1}^\ell \sum_{j=1}^\ell a_ia_j P_{ij}. $$

% $$ %(\bE X_1)^{-3} \on{Var}( Y_1 \bE X_1  -  X_1\bE Y_1)) =
% \begin{array}{rcl}
% \displaystyle 
% \frac{\on{Var}\left( Y_1 \bE X_1  -  X_1\bE Y_1\right)}{\left(\Ee X_1\right)^3} 
% & = & \displaystyle 
% \sum_{i=1}^\ell \sum_{j=1}^\ell a_ia_j \Cov\left( \frac{C_i(\Sigma_1)\Ee X_1 - X_1\Ee C_i(\Sigma_1)}{\left(\Ee X_1\right)^{3/2}}, 
% \frac{C_j(\Sigma_1)\Ee X_1 - X_1\Ee C_j(\Sigma_1)}{{\left(\Ee X_1\right)^{3/2}}}\right) \\
% & =: & \displaystyle 
% \sum_{i=1}^\ell \sum_{j=1}^\ell a_ia_j P_{ij}, 
% \end{array} $$

\noindent
This shows that if $(N_1,\dots,N_\ell) \sim \cN_\ell( \underline{0}, P_\ell )$ then 

$$ a_1 N_1 + \dots + a_\ell N_\ell \isd 
\cN\left(0, \frac{\on{Var}\left( Y_1 \bE X_1  -  X_1\bE Y_1\right)}{\left(\Ee X_1\right)^3} 
\right). $$

\noindent
We've thus shown that

$$ a_1 \frac{C_1(\Pi_n) - m_1 n}{\sqrt{n}} + \dots 
+ a_\ell \frac{C_\ell(\Pi_n) - m_\ell n}{\sqrt{n}}
\quad \xrightarrow[n\to\infty]{\text{d}} \quad a_1 N_1 + \dots + a_\ell N_\ell, $$

\noindent
for all $a_1,\dots,a_\ell$. An application of Theorem~\ref{thm:cramer_wold} now allows us to conclude

$$ \left( \frac{C_1(\Pi_n)-m_1 n}{\sqrt{n}}, \dots, \frac{c_\ell(\Pi_n)-m_\ell n}{\sqrt{n}} \right) \quad 
\xrightarrow[n\to\infty]{d} \quad 
(N_1,\dots,N_\ell), $$

\noindent
completing the proof of Theorem~\ref{thm:normal}. \hfill $\blacksquare$

\section{The proof of Theorem~\ref{thm:clt_q_gt_one}\label{sec:normal_q_geq_one}}

The proof is very similar to the proof of Theorem~\ref{thm:normal}.
We first introduce a two--sided sampling procedure in the case $0<q<1$ for a $\on{Mallows}(n,q)$ distributed permutation $\Pi_n$ 
taking $\floor{\frac{n}{2}}$ iterations. 
During iteration $i\geq 1$ we determine the images of $i$ and $n-i+1$. 
Again we take $Z_1,\dots,Z_n$ independent with $Z_i \sim \TGeo(n+1-i,q)$.
In the first iteration we set%
\begin{align}	
\Pi_n(1) &= Z_1,\\
\Pi_n(n)&= \text{the $Z_2$-th largest number in the set }[n]\setminus \{ \Pi(1) \}. 
\end{align}%
In the $i$-th iteration we set 
\begin{align}\label{eq:MallowsIter}	
\Pi_n(i)&= \text{the $Z_{2i-1}$-th smallest element of }[n]\setminus \{ \Pi_n(j) : j<i \text{ or } j > n+1-i \}, \\
\Pi_n(n+1-i)&=\text{the $Z_{2i}$-th largest element of }[n]\setminus \{ \Pi_n(j) : j\leq i \text{ or } j > n+1-i \}
\end{align}%
(If $n$ is odd then after $\floor{\frac{n}{2}}$ iterations, the image of $\lceil\frac{n}{2}\rceil$ has formally speaking not yet been 
determined, but of course there will be only one possible element of $[n]$ left.)

That this adapted procedure indeed produces a random permutation sampled according to the $\Mallows(n,q)$ measure follows
analogously the the corresponding argument for the original sampling procedure : 
For every $\pi \in S_n$ there is a choice of 
$(k_1,\dots,k_n) \in [n]\times[n-1]\times\dots\times[1]$ such that 
$\{\Pi_n = \pi\} = \{ Z_1=k_1,\dots,Z_n=k_n\}$. Again
$\Pee( Z_1=k_1,\dots,Z_n=k_n ) \propto q^{k_1+\dots+k_n-n}$. We also again have
$\inv(\pi) = k_1+\dots+k_n-n$, because when we are determining $\Pi_n(i)$ with $i \leq n/2$ then 
the number of $i<j<n+1-i$ such that $i,j$ form an inversion is precisely
$Z_{2i-1}-1$, and similarly for $\Pi_n(n+1-i)$.

Recall that we use $r_n$ to denote the map $i \mapsto n+1-i$.
Analogously to Lemma~\ref{lem:PinPicpl}, we have

\begin{lem}\label{lem:PinPicpl2sided}
Let $0<q<1$ and $\Pi_n \sim \Mallows(n,q)$ and let $\Pi,\Pi' \sim\Mallows(\eN,q)$ be independent.
There is a coupling for $\Pi_n, \Pi,\Pi'$ such that 

$$ \Pee\left( \begin{array}{l} 
\Pi_n(i) = \Pi(i) \text{ for all $1\leq i \leq n/2-\log^2 n$, and } \\
\Pi_n(i) = (r_n \circ \Pi' \circ r_n)(i)  \text{ for all $n/2+\log^2 n\leq i \leq n$}       
              \end{array}
 \right)
= 1 - o(1). $$
Moreover, the $\log^2 n$ can be replaced with any function going to $\infty$ with $n$.
\end{lem}

\begin{proof}
The proof is similar to the proof of Lemma~\ref{lem:PinPicpl}. By Lemma~\ref{lem:PinPicpl} applied to $\Pi$ and $\Pi'$, we can couple $\Pi,\Pi'$ with independent $\Mallows(n/2,q)$ permutations, $\Pi_{n/2},\Pi'_{n/2}$ (for simplicity, we write $n/2$ even if $n$ is odd, where it should be rounded either up or down as needed), such that $\Pee(\Pi_{n/2}(i)=\Pi(i),\Pi_{n/2}'(i)=\Pi'(i)\text{ for all }1\leq i\leq n/2-\log^2 n)=1-o(1)$. We now claim that we can couple $\Pi_n$ with $\Pi_{n/2}, \Pi_{n/2}'$ such that  
$$ \Pee\left( \begin{array}{l} 
\Pi_n(i) = \Pi_{n/2}(i) \text{ for all $1\leq i \leq n/2-\log^2 n$, and } \\
\Pi_n(i) = (r_n \circ \Pi'_{n/2} \circ r_n)(i)  \text{ for all $n/2+\log^2 n\leq i \leq n$}       
\end{array}
\right)
= 1 - o(1), $$
and this would immediately finish the proof, since any coupling with these bivariate marginals would satisfy the lemma.

To see the claim, note that $(\Pi_n)_{[n/2]}$ and $(\Pi_n)_{[n/2+1,n]}$ are independent $\Mallows(n/2,q)$ (see e.g. Lemma 2.3 of \cite{HeDescents}), and so can be coupled to perfectly agree with $\Pi_{n/2}$ and $r_n \circ \Pi'_{n/2} \circ r_n$ on the intervals $[n/2]$ and $[n/2,n]$. Now if $\Pi_n([k])=[k]$ for $k\leq n/2$, then $\Pi_n$ and $\Pi_{n/2}$ agree on that interval, since then $\Pi_n(i)=(\Pi_n)_{[1,n/2]}(i)=\Pi_{n/2}(i)$ for $i\in [k]$, and similarly for $r_n \circ \Pi'_{n/2} \circ r_n$ and intervals of the form $[k,n]$. Thus, the number of $i$ with disagreements is bounded by the length of the smallest interval $[a,b]$ with $a\leq n/2$, $b\geq n/2$, and $\Pi_n([a])=[a]$ and $\Pi_n([b,n])=[b,n]$. However, this random variable is stochastically dominated by $X_{\tau(n/2)}$, since if we couple $\Pi_n$ with a $\Mallows(\eN,q)$ process, then the length of the interval is $b-a=X_{\tau(n/2)}$ unless $X_{\tau(n/2)}>n$, in which case $X_{\tau(n/2)}$ is strictly larger. But now we are done, since $X_{\tau(n/2)}$ converges to a limiting distribution by Proposition \ref{prop: waiting time}, and so the probability that $\Pi_n$ disagrees with either $\Pi_{n/2}$ or $r_n \circ \Pi'_{n/2} \circ r_n$ for a growing number of locations goes to $0$.
\end{proof}

Recall that $r_n \circ \Pi_n \isd \Mallows(n,1/q)$ if $\Pi_n\sim\Mallows(n,q)$.
We have the following corollary.

\begin{cor}\label{cor:PinPicpl2sidedq_geq_one}
Let $q>1$, $\Pi_n \sim \on{Mallows}(n,q)$ and $\Pi,\Pi' \sim \on{Mallows}(\bN, 1/q)$ be independent. 
Let $ m = \floor{\frac{n}{2} - (\log n)^2}$. There exists a coupling between $\Pi_n$ and $\Pi, \Pi'$ such that that
\begin{equation}\label{eq:coupling_bound}
\bP \left[ (r_n\circ \Pi_n,  \Pi_n \circ r_n) \text{ and } (\Pi, \Pi') 
\text{ agree on $1,\ldots,m$}\right] =1- o(1). 
\end{equation} 
\end{cor}

\begin{proof}
We let $\Pi_n^* := r_n \circ \Pi_n$. Then $\Pi_n^* \isd \Mallows(n,1/q)$ and we 
can couple $\Pi_n^*$ with $\Pi,\Pi'\sim\Mallows(\eN,1/q)$ independent as in the previous lemma.
With probability $1-o(1)$, we have 

$$ (r_n\circ\Pi_n)(i) = \Pi_n^*(i) = \Pi(i), $$

\noindent 
and

$$(r_n\circ \Pi_n \circ r_n)(i) = \Pi_n^*(r_n(i)) = (r_n\circ\Pi'\circ r_n)(r_n(i))
= (r_n\circ\Pi')(i), $$

\noindent 
for all $i\leq m$. 
Now notice that $(r_n\circ \Pi_n \circ r_n)(i) = (r_n\circ \Pi')(i)$ if and only if
$(\Pi_n\circ r_n)(i) = \Pi'(i)$.
\end{proof}

For $q>1$, we let $\Pi_n\sim\Mallows(n,q)$ and $\Pi, \Pi'\sim\Mallows(\eN,1/q)$ be coupled as in 
Corollary~\ref{cor:PinPicpl2sidedq_geq_one}. 
We define 

\begin{align}
	T_0 &= 0,\\
	T_{i}& =\inf\{  j > T_{i-1} : \Pi([j]) = \Pi'([j]) = [j]  \} \qquad (i=1,2,\dots),\\
\end{align}

\noindent
We next show the analogue of Lemma~\ref{lem:Rplus}.

\begin{lem}\label{lem:Rplus1}
$\Ee T_1^2 < \infty.$
\end{lem}

\begin{proof}
Let $\hat{\kappa}, \hat{\kappa}'$ denote the arc-chains for $\Pi$, respectively $\Pi'$, as defined in Section~\ref{sec:prelim}.
Note that $\Pi([j])=[j]$ if and only if $\hat{\kappa}_{j} = 0$ and 
$\Pi'([j]) = [j]$ if and only if $\hat{\kappa}_{j}' = 0$.
Thus $T_1$ is the first return to the origin of the two-dimensional random walk $(\hat{\kappa}_t,\hat{\kappa}_t')$, started
at $(\hat{\kappa}_0,\hat{\kappa}_0')=(0,0)$. 
For convenience, let us write $\lambda_t := \max(\hat{\kappa}_t, \hat{\kappa}_t')$.
We observe that 

$$ \Pee( \lambda_{t+1} = i-1 | \lambda_t = i, \lambda_{t-1} = i_{t-1}, \dots,
\lambda_1 = i_1 ) \geq \Pee( \hat{\kappa}_{t+1} = i-1 | \hat{\kappa}_t = i )^2 =
\left(1-(1/q)^i\right)^4, $$

\noindent
for all $i\geq 1$ and $i_1,\dots, i_{t-1} \in \{0\}\cup \eN$ and all $t$.
So, while $(\lambda_t)_{t\geq 0}$ is itself not necessarily a Markov chain, it is stochastically dominated by the 
chain $(W_t)_{t\geq 0}$ on $\{0\}\cup\eN$ with starting state
$W_0=0$, and transition probabilities 

$$ \Pee( W_{t+1} = j | W_{t}=i ) = \begin{cases}
                                    \left(1-(1/q)^i\right)^4 & \text{ if $j=1-1$ and $i\geq 1$, } \\
                                    1 - \left(1-(1/q)^i\right)^4 & \text{ if $j=1+1$ and $i \geq 0$, }\\
                                    0 & \text{ otherwise.}
                                   \end{cases}
$$

\noindent 
The result now immediately follows from Lemma~\ref{lem:MCret}.
\end{proof}

We also define $X_i := T_i - T_{i-1}$ and 
$\tau(t) = \inf\{j : T_j > t\}$ for all $t>0$.
Again it can be easily seen from the iterative procedure generating $\Pi$ and $\Pi'$ than 
$X_1, X_2,\dots$ are i.i.d.
Moreover, if we define the maps $\Sigma_1, \Sigma_1' : [X_1]\to[X_1], \Sigma_2,\Sigma_2' : [X_2]\to[X_2], \dots$ by setting

$$ \Sigma_i(j) := \Pi( T_{i-1} + j )-T_{i-1}, \quad \Sigma_i'(j) :=  \Pi'( T_{i-1} + j )-T_{i-1}, $$

\noindent
then $(\Sigma_1,\Sigma_1'), (\Sigma_2,\Sigma_2'),\dots$ are i.i.d.~as well.

We write $\Xcal_i := \{ T_{i-1}+1,\dots, T_i\}$.
Observe that, with probability $1-o(1)$, for each $i$ such that $T_i < n/2-\log^2 n$ we have 

$$ \Pi_n\left[\Xcal_i\right] = r_n\left[ \Pi\left[\Xcal_i\right] \right] = r_n\left[\Xcal_i\right], $$%
$$ \Pi_n\left[ r_n\left[\Xcal_n\right]\right] = \Pi'\left[\Xcal_i\right] = \Xcal_i, $$

\noindent
by Corollary~\ref{cor:PinPicpl2sidedq_geq_one}. In other words, writing 
$\Ycal_i := \Xcal_i \cup r_n\left[\Xcal_i\right]$, we have 

$$ \Pi_n\left[ \Ycal_i \right] = \Ycal_i, $$

\noindent
for each $i$ such that $T_i < n/2-\log^2 n$.
In particular, every cycle of $\Pi_n$ is either 
completely contained in one of $\Ycal_1, \dots, \Ycal_{\tau(n/2)-1}$ or it contains some number 
between $\min(n/2-\log^2 n, T_{\tau(n/2)-1})$ and $\max(n/2+\log^2 n, n+1-T_{\tau(n/2)})$.
We observe that the number of cycles of $\Pi_n$ length $2i$ contained in $\Ycal_i$
equals the number of cycles of $\Pi_n\circ\Pi_n$ of length $i$ contained in $\Xcal_i$.
Now note that on $\Xcal_i$ we have $\Pi_n\circ\Pi_n = (\Pi_n\circ r_n) \circ (r_n\circ \Pi_n) 
= \Pi'\circ\Pi$. In particular, the number cycles of $\Pi_n\circ\Pi_n$ of length $i$ contained in $\Xcal_i$
equals the number of cycles of $\Sigma_i'\circ\Sigma_i$ of length $i$.

We fix $a_{1},\dots,a_\ell \in \eR$, not all zero, and set

$$ \varphi(\pi) := \sum_{i=1}^\ell a_i C_{2i}(\pi), \quad \psi(\pi) := \sum_{i=1}^\ell a_i C_i(\pi), $$

$$ Y_i := \psi( \Sigma_i'\circ\Sigma_i ). $$

By the previous

\begin{equation}\label{eq:aapie} 
\left| \varphi(\Pi_n) - \sum_{i=1}^{\tau(n/2)} Y_i \right| \leq 
2 \max_i |a_i| \cdot \left( X_{\tau(n/2)} + \log^2 n \right) = o\left( \sqrt{n} \right), 
\end{equation}

\noindent
with probability $1-o(1)$ (where the last equality holds 
because $X_{\tau(n/2)} = o(\sqrt{n})$ with 
probability $1-o(1)$ by Lemma~\ref{lem:max_iid} and Lemma~\ref{lem:Rplus1}).
Completely analogously to Lemmas~\ref{lem:other_conditions} and~\ref{lem:varZt} we have

\begin{lem}
$\Ee X_1 > 0$ and $\Var X_1, \Var Y_1 < \infty$.\hfill $\blacksquare$
\end{lem}

\begin{lem}
$\Var\left( Y_1\Ee X_1 - X_1\Ee Y_1\right) > 0$. \hfill $\blacksquare$
\end{lem}

We can thus conclude from Theorem~\ref{thm:regen_clt} that

$$ \frac{ \sum_{i=1}^{\tau(n/2)} Y_i - \left(\frac{\Ee Y_1}{\Ee X_1}\right) \cdot (n/2)}{\sqrt{n/2}} 
\xrightarrow[n\to\infty]{\text{d}} \cN\left( 0, \frac{\Var\left( Y_1\Ee X_1 - X_1\Ee Y_1\right)}{\left(\Ee X_1\right)^3}\right). 
$$

\noindent
Setting $\mu_{2i} := \frac{\Ee C_i(\Sigma_1'\circ \Sigma_1)}{2\Ee X_1}$ and using~\eqref{eq:aapie}
this gives

$$ a_1\cdot\frac{C_{2}(\Pi_n) - \mu_2 n}{\sqrt{n}}+\dots+a_\ell\cdot\frac{C_{2\ell}(\Pi_n) - \mu_{2\ell} n}{\sqrt{n}}
\xrightarrow[n\to\infty]{\text{d}}
a_1 N_1 + \dots + a_\ell N_\ell, $$

\noindent
where $(N_1,\dots,N_\ell) \isd \cN(\underline{0}, P)$ with 

$$ P_{ij} := \Cov\left( \frac{C_i(\Sigma_1'\circ\Sigma_1)\Ee X_1 - X_1\Ee C_i(\Sigma_1'\circ\Sigma_1)}{
\sqrt{2} \cdot \left(\Ee X_1\right)^{3/2}}, 
\frac{C_j(\Sigma_1'\circ\Sigma_1)\Ee X_1 - X_1\Ee C_j(\Sigma_1'\circ\Sigma_1)}{{\sqrt{2}\cdot\left(\Ee X_1\right)^{3/2}}}\right). $$

\noindent
Again the result follows by an application of the Cramer-Wold device.

\section{The proof of Theorem~\ref{thm:odd_cycles}}

%%The following quantitative version of part~\ref{item:gnedin_converge} of Lemma~\ref{lem:gnedin} will be very helpful for us.
Part \ref{item:gnedin_converge} of Lemma \ref{lem:gnedin} says that, almost surely, $\Sigma_{I_n} \to \Sigma$ for $I_n := \{-n,\dots,n\}$.
It however leaves open how fast the convergence is. 
The following lemma shows that in fact, with high probability, for the vast majority of elements of $I_n$, the values
of $\Sigma$ and $\Sigma_{I_n}$ agree. This will be very helpful for us.

\begin{lem}\label{lem:convergence_infinite}
Let $0<q<1$ and $\Sigma \sim \on{Mallows}(\bZ, q)$ and $I_n := \{-n,\ldots,n\}, J_n := \{-n+1,\dots,n\}$. 
Almost surely there exists a (random) $N \in \bN$ such that 
$\Sigma_{I_n}(i)=\Sigma_{J_n}(i)=\Sigma(i)$ for all $n\geq N$ and $i$ with $|i| \leq n - \log^2 n$.
\end{lem}

\begin{proof}
We define the event $B_n$ by %
\begin{equation}
B_n := \left\{ \begin{array}{l}|\Sigma(i) - i | < 5\log_{1/q} n \quad\text{for all $i \in I_{n}$, and }\\
|\Sigma(j)-j| \leq 5\log_{1/q} n + |j| - n\quad \text{for all }j\notin  I_n	               
\end{array} \right\}.
\end{equation}%
By Lemma \ref{lem:gnedin}, Part \ref{item:gnedin_disp} we have
	\begin{equation}\label{eq:event_BN}
\bP \left[ {B^c_n} \right] = O \left(   (2n+1) q^{5\log _{1/q} n } + 2 \sum_{j>n}q^{5\log _{1/q}n + j - n} \right) = O(n^{-4}).
	\end{equation}

We will show that $B_n$ implies the conclusion of the lemma. 
Let $a\geq 0$ and consider the intervals $I_{n+a}$. 
Let $i$ be such that $|i| < n - 10 \log_{1/q} n$. If $B_n$ holds, then for all $j < -n$ we have 
$\Sigma(j) < -n + 5 \log_{1/q} n \leq \Sigma(i)$. 
Similarly $j> n$ implies $\Sigma(j) > n - 5 \log_{1/q} n \geq \Sigma(i)$. Thus $B_n$ implies
	\begin{equation}
		|\{ j\in I_{n+a} \mid \Sigma (j)  < \Sigma (i)  \}| = a + |\{ j\in I_n \mid \Sigma (j)  < \Sigma (i)  \}|.
	\end{equation}
	As 
	\begin{equation}
		\Sigma_{I_{n+a}} (i) = -n-a + |\{ j\in I_{n+a} \mid \Sigma (j)  < \Sigma (i)  \}|,
	\end{equation}
the sequence $\Sigma_{I_{n+a}}(i)$ is constant for all $a\geq 0$. By Lemma \ref{lem:gnedin}, Part \ref{item:gnedin_converge}, 
with probability one there is some $n'$ such that for all $n''\geq n'$ we have $\Sigma_{I_{n''}}(i)=\Sigma(i)$. 
There is some $a\geq 0$ such that $n' \leq n+a$, so that in particular we must have $\Sigma_{I_n}(i) = \Sigma(i)$.
Similarly, $\Sigma_{J_n}(i) = \Sigma(i)$.
By the Borel--Cantelli Lemma $B_n$ holds for all but finitely many $n$. 
\end{proof}

We also require the following Markov chain representation for the times $T_i$. Let $\Pi\sim\Mallows(\eN,q)$. Consider the process
$$
M_n=\max_{1\leq i\leq n}\Pi(i)-n
$$
on $\eN$. This is a positive recurrent Markov process -- see \cite{basu2016limit}. 
The Markov process can be described in terms of the geometric random variables defining 
the Mallows process. Specifically, the walk can be described as moving from $M_n$ to
$M_{n+1}=\max(M_n,Z_{n})-1$ where the $Z_n$ are independent geometric random variables. 
Let $R_i$ denote the hitting time of $i$ and let $R_i^+$ denote the return time at $i$. Then if the chain is started from $0$, $R_0^+$ is distributed as the size of an excursion in the Mallows process. This Markov chain was introduced in \cite{basu2016limit} to study the moments of the $T_i$. We are now ready for the proof of Theorem~\ref{thm:odd_cycles}.

We start by considering $\Pi_{2n+1}$. We let $\Sigma\sim \Mallows(\bZ,1/q)$, and 
let $S_0$ be the smallest integer (we will show that this exists) such that $\Sigma([-S_0,S_0])=[-S_0,S_0]$, 
and then let $S_i$ be defined inductively as the smallest number 
larger than $S_{i-1}$ such that $\Sigma$ preserves $[-S_i,S_i]$. 
To see that all these values are finite almost surely, we note that the times 
$T_i$ for $i\in \bZ$ for which $\Sigma([T_i+1,T_{i+1}])=[T_{i}+1,T_i]$, with the convention that $T_0$ 
contains $0$, forms a stationary renewal process, with the $T_i$ for $i\neq 0$ the same as for 
the $\Mallows(\eN,1/q)$ process, and $T_0$ having its size-bias distribution 
(see Theorem 3.2 of \cite{PitmanTang}). 
Then $T_0$ is finite almost surely, and given $T_0$, the two sides are independent and behave 
like $\Mallows(\eN,1/q)$ processes. Then the $S_i$ correspond to a simultaneous renewal on both sides 
of this process, with $S_0$ being the first time this occurs, which are the return times in a product of 
two independent copies of the positive recurrent Markov chain $M_i$ defined above, which is thus also 
positive recurrent. Thus, all $S_i$ are finite almost surely. By definition, $\Sigma\circ r$ preserves $[-S_0,S_0]$ 
and exchanges $[S_{i-1}+1,S_i]$ and $[-S_i,-S_{i-1}-1]$.

We then immediately see that there are no infinite cycles in $\Sigma\circ r$, and the odd cycles must be contained 
in the interval $[-S_0,S_0]$. Thus, the $C_{2i+1}(\Pi_{[-n,n]})\to C_{2i+1}(\Pi_{[-S_0,S_0]})$ almost surely.
The result for $\Pi_{2n+1}$ follows by noting that $C_{2i+1}(\Pi_{2n+1})\isd C_{2i+1}(\Pi_{[-n,n]})$
by Lemma~\ref{lem:gnedin} and the observations that precede it.
	
The proof of the result for $\Pi_{2n}$ follows in exactly the same manner, except centered at $\frac{1}{2}$ rather 
than $0$, and using $\rho$ instead of $r$.\par

Technically speaking, the proof of Theorem~\ref{thm:odd_cycles} is not yet complete, as we have not
yet shown the random vectors $(C_1(r\circ\Sigma), C_3(r\circ\Sigma),\dots)$ and 
 $(C_1(\rho\circ\Sigma), C_3(\rho\circ\Sigma),\dots)$ have distinct probability distributions.
 This will however follow immediately from Theorem~\ref{thm:asymp}, which we will prove in a little while.

We also note the following consequence which will be useful later.
\begin{lem}\label{cor:no_infinite}
	Let $\Sigma \sim \on{Mallows}(\Zed,q)$ for $0<q<1$. Almost surely, $\Sigma$ has no cycles of infinite length.
\end{lem}

\begin{proof}
	Any cycle must be contained in an interval $[-S_0,S_0]$, $[S_{i-1}+1,S_i]$, or $[-S_i,-S_{i-1}-1]$, all of which are finite almost surely.
\end{proof}

\section{The proof of Theorem~\ref{thm:mi}}

\subsection{The proof of Part~\ref{itm:m1part1} of Theorem~\ref{thm:mi}}

% 
% 
% Recall that $f^{(n)}$ denotes the $n$-fold composition of $f$ with itself.
% The following observation will help to study (long) cycles of $\Sigma\sim\Mallows(\Zed,q)$.
% 
% \begin{lem}\label{lem:cycle_bound}
% Let $\Sigma \sim \on{Mallows}(\Zed,q)$ with $0<q<1$, and let $i \in \bN$. Then 
% \begin{equation}
% \bP \left[ \max\{\Sigma(0), \Sigma^{(2)}(0), \dots,\Sigma^{(i)}(0)\} \geq m \right]  = O(m \,q^{m/i}),
% \end{equation}
% as $m\to\infty$. 
% \end{lem}
% 
% \begin{proof}
% Let $I_m = \{-m,\ldots , m\}$. Let $A$ be the event that for all $j\in I_m$ we have $|\Sigma(j) -j| \leq m /i$. By the union bound we have 
% \begin{equation}
% \bP \left[ A^c \right] \leq  O(|I_m|\ q^{m/i}), 
% \end{equation}
% by Lemma \ref{lem:gnedin}, Part \ref{item:gnedin_disp}. If $A$ holds then $\max_{j \in [i]} \{\Sigma^{(j)}(0)\} \leq m$. 
% \end{proof}

Let $0<q<1$, let $\Pi_{n} \sim \Mallows(n,q)$ and $\Sigma \sim \on{Mallows}(\Zed,q)$  and let $i \in \eN$ be fixed. 
%Note that $i C_i(\Pi_n)$ counts the number of points of $[n]$ that lie in an $i$-cycle of $\Pi_n$.
Using that $0 \leq C_i(\Pi_n)/n \leq 1$, it follows from Theorem~\ref{thm:normal} that 

\begin{equation}\label{eq:nootie} \lim_{n\to\infty} \frac{\Ee C_i(\Pi_n)}{n} = m_i. \end{equation}

Setting $I_n := \{-n,\ldots, n\}$, by Part~\ref{itm:gnedin_Z_finite} of Lemma~\ref{lem:gnedin}, we have 
$C_i(\Pi_{2n+1}) \isd C_i(\Sigma_{I_n})$.
We define the events

$$ \begin{array}{l}
E := \{\text{$\Sigma_{I_n}(j) = \Sigma(j)$ for all $-n+\log^2 n\leq j \leq n-\log^2 n$}\}, \\
F := \{\text{$|\Sigma(j)-j| \leq \log^2 n$ for all $j \in I_n$ }\}. 
\end{array} $$

\noindent
Then $\Pee(E) = 1-o(1)$ by Lemma \ref{lem:convergence_infinite}.
By Part~\ref{itm:gnedin_displacements} of Lemma~\ref{lem:gnedin} and the remarks that 
follow that lemma $\Pee\left(F^c\right) = O\left( 
n \cdot q^{\log^2 n} \right) = o(1)$.

It follows that for all $j$ with $|j| < n-i \cdot \log^2 n$ we have 

$$ \left| \Pee( \text{$j$ in an $i$-cycle of $\Sigma_{I_n}$} ) -
\Pee( \text{$j$ in an $i$-cycle of $\Sigma$} ) \right| \leq \Pee( E^c) + \Pee( F^c ) = o(1). $$

The number of elements of $I_n$ in an $i$-cycle of $\Sigma_{I_n}$ equals $i \cdot C_i( \Sigma_{I_n} )$.
We see that

\begin{equation}\label{eq:miesie} \begin{array}{rcl} i \cdot \Ee C_i( \Pi_{2n+1} ) 
& = & \displaystyle 
\sum_{-n\leq j \leq n} \Pee( \text{$j$ in an $i$-cycle of $\Sigma_{I_n}$} ) \\[2ex]
& = & \displaystyle 
\sum_{-n \leq j \leq n } 
\Pee( \text{$j$ in an $i$-cycle of $\Sigma$} ) + o(n) \\[2ex]
& = & \displaystyle
(2n+1) \cdot \Pee( \text{$0$ in an $i$-cycle of $\Sigma$} ) + o(n), 
\end{array} \end{equation}

\noindent
where we use Part~\ref{item:gnedin_shift} of Lemma~\ref{lem:gnedin} (together with the remarks following the lemma) 
for the last line.
Dividing the LHS and RHS of~\eqref{eq:miesie} 
by $2n+1$ and sending $n\to\infty$, and recalling~\eqref{eq:nootie}, proves the result.

\subsection{The proof of Part~\ref{itm:m1part2} of Theorem~\ref{thm:mi}}

By Lemma \ref{cor:no_infinite}, $\Sigma\sim\on{Mallows}(\Zed,q)$ for $0<q<1$ almost surely has no infinite cycles. By Part~\ref{itm:m1part1} of Theorem \ref{thm:mi} we have 
\begin{equation}
\sum_{i=1}^{\infty} i m_i = \sum\limits_{i=1}^\infty \bP \left[ 0 \text{ in an $i$--cycle of $\Sigma$} \right] = 
\bP \left[ \text{$0$ lies in a finite cycle of $\Sigma$} \right] = 1,
\end{equation}
the last equality due to Lemma \ref{cor:no_infinite}.

\subsection{The proof of Part~\ref{itm:m1part3} of Theorem~\ref{thm:mi}}

We start by giving an alternative expression for $m_1$ by employing the tools developed by Gladkich and Peled~\cite{Peled}.

\begin{lem}\label{lem:m_1_b}
	\begin{equation}\label{eq:m_1}
		m_1 = \sum\limits_{s\geq 0} \nu_s q^{2s}(1-q),
	\end{equation}
	with $\nu$ as given in \eqref{eq:nu}. 
\end{lem}

\begin{proof}
Let $\Pi_{2n+1} \sim \on{Mallows}(2n+1,q)$. By Lemma \ref{lem:convergence_infinite} we have that
	\begin{equation}
		\bP \left[ \Pi_{2n+1} (n+1) =  n + 1 \right] \to 
		\bP \left[ \Sigma(0) = 0 \right] ,\qquad \text{ as $n\to \infty$},
	\end{equation}
	where $\Sigma \sim \on{Mallows}(\bZ, q)$. Now,%
\begin{equation}\label{eq:frhrgrberjh}
\begin{array}{c}
\displaystyle 
\lim_{n\to \infty}\bP \left[ \Pi_{2n+1}(n+1) = n + 1 \right] \\
 = \\
 \displaystyle 
 \lim_{n\to \infty} \sum\limits_{s \geq 0} \bP \left[ \kappa_{n+1} = s  \right] \cdot 	
 \bP \left[ \Pi_{2n+1}(n+1) = n + 1\, |\, \kappa_{n+1} = s \right].
\end{array}
\end{equation}%
By Lemma \ref{lem:prop_fixed_kappa} and Proposition \ref{prop:convergence_kappa} we have for all $s \geq 0$ that
	\begin{align}
		\lim_{n\to \infty} \bP \left[ \Sigma_{2n+1}(n+1) = n + 1\, |\, \kappa_{n+1} = s \right] &= q^{2s}(1-q),\\
		\lim_{n\to \infty} \bP \left[ \kappa_{n+1} = s  \right] &= \nu_s. 
	\end{align}
	So the summands on the right hand side of \eqref{eq:frhrgrberjh} converge pointwise to $\nu_s q^{2s}(1-q)$ as $n\to \infty$, and are uniformly bounded by $1$ for all $n$. By the bounded convergence theorem we thus conclude that
	\begin{equation}
		m_1 = \bP \left[ \Sigma(0) = 0 \right] = \sum\limits_{s\geq 0}\nu_s q^{2s}(1-q).
	\end{equation}
\end{proof}

Next, we will show that $m_1 = 1-  2q + O(q^2)$ as $q\downarrow 0$ be analyzing \eqref{eq:m_1}. 
Let $K_q$ denote the denominator in the expression for
$\nu$ given in~\eqref{eq:nu}. Define $t_s = \prod_{i=1}^s q^{2i-1} / (1-q^i)^2$. 
The $t_s$ satisfy for $s\geq 1$ the recursion relation 
	\begin{equation}
		t_{s+1} = t_s\cdot\frac{q^{2}}{(1-q^{s+1})^2 }\leq \frac{q^2 t_s}{(1-q)^2} \leq t_1 \left( \frac{q^2}{(1-q)^2} \right)^{s-1}  =\frac{q}{(1-q)^2}  \left( \frac{q^2}{(1-q)^2} \right)^{s-1}. 
	\end{equation}
	Thus
	\begin{align}
		K_q  = 1 + \sum\limits_{s\geq 1} t_s \leq 1 + \frac{q}{(1-q)^2} \sum\limits_{s\geq 1}  \left( \frac{q^2}{(1-q)^2} \right)^{s-1} .
	\end{align}
	If $q < 1/2$ then $q^2 / (1-q)^2 < 1$, in which case the above equals
	\begin{equation}
		1 + \frac{q}{(1-q)^2} \cdot \frac{1}{ 1 - \frac{q^2}{(1-q)^2} } = 1 + \frac{q}{(1-q)^2 - q^2}=  \frac{1-q}{1 - 2q}.
	\end{equation}
	Then 
	\begin{equation}
		\nu_0 =  \frac{1}{K_q} \geq \frac{1-2q}{1-q},
	\end{equation}
	so that by $m_1 \geq \nu_0 (1-q)$ we obtain 
	\begin{equation}\label{eq:lower_m1}
		m_1 \geq \nu_0 (1-q) \geq 1 - 2q.
	\end{equation}
	We also have the simple bound 
	\begin{equation}
		K_q \geq 1 + t_1 = 1 + \frac{q}{(1-q)^2},
	\end{equation}
	by which 
	\begin{equation}\label{eq:upper}
		m_1 \leq \frac{1-q}{  1 + \frac{q}{(1-q)^2}  } + \sum\limits_{s\geq  1} q^{2s}(1-q) = \frac{(1-q)^3}{1-q+q^2} + O(q^2) .
	\end{equation}
	The function ${(1-q)^3}/{(1-q+q^2)}$ is infinitely differentiable at $q=0$, where its first derivative equals $-2$. Thus
	\begin{equation}
		\frac{(1-q)^3}{1-q+q^2} = 1  - 2q + O(q^2).
	\end{equation}
	Together with \eqref{eq:lower_m1} and \eqref{eq:upper} this completes the the proof that $m_1 = 1 - 2q + O(q^2)$. \par

	We now proceed to show that $m_1 = (1-q)/4 + o(1-q)$ as $q\uparrow 1$. 
	Fix some $ 0 <\epsilon < \frac12$, let $0<q<1$ be so close to $1$ that $\frac{1}{q} \leq (1 + \epsilon)$ and let $S = \{s  \in \bN_{\geq 0} \mid |q^s - \frac12| < \epsilon\}$. If $q^s < \frac12  - \epsilon$, then by \eqref{eq:nu} we have the recursion
	\begin{equation}
		\nu_{s+1} = \nu_s \frac{1}{q} \left(  \frac{q^s}{1-q^s} \right)^2 \leq \nu_s \delta \qquad \text{where }\qquad \delta  \leq (1 + \epsilon) \left(\frac{ \frac12 - \epsilon }{ \frac12 + \epsilon}\right)^2 < 1.
	\end{equation}
	Then 
	\begin{equation}\label{eq:1afsafs}
		\sum\limits_{s>\max(S)}\nu_s \leq \nu_{\max(S)} \sum\limits_{ s > \max(S)} \delta^{s - \max(S)} =  \nu_{\max(S)}\frac{ \delta}{1 - \delta}.
	\end{equation}
	We now claim that $\nu_{\max (S)} = o(1)$ as $q \uparrow 1$. Let $s^* = \ceil{\log_q (\frac12(1 - \epsilon)) }$, then for $k\geq  0$ we have
	\begin{equation}
		\nu_{s^* + k + 1}\leq \nu_{s^*} \frac{1}{q} \left( \frac{q^{s^*+k+1}}{1 - q^{s^* +k + 1}} \right) \leq \nu_{s^*} (1+ \epsilon) \left( \frac{ 1 - \epsilon }{1 + \epsilon} \right) ^2 < \nu_{s^*}.
	\end{equation}
	So
	\begin{equation}
		1 \geq \sum\limits_{s= s^*}^{ \max (S) } \nu_s \geq (\max(S) - s^*)\nu _{\max(S)}.
	\end{equation}
	By
	\begin{equation}
		\max(S) -  s^* = \floor{\log_q\left( \frac{1}{2} + \epsilon \right)} - \ceil{\log_q\left(\frac{1}{2}  + \frac{\epsilon}{2}\right)} \to \infty \qquad\text{as } q \uparrow 1,
	\end{equation}
	we indeed have $\nu_{\max(S)}  = o(1)$ as $q\uparrow 1$. 

	Similarly, if $q^s > \frac12 + \epsilon$, then $\nu_{s+1} \geq \nu_s \tilde{\delta}$, where $\tilde{\delta} \geq \frac{\frac12 + \epsilon}{\frac12 - \epsilon} > 1$. Then 
	\begin{equation}\label{eq:2afsafs}
		\sum\limits_{0\leq s < \min(S)} \nu_s \leq \nu_{\min(S)} \sum\limits_{0\leq s < \min(S)} \tilde{\delta}^{ \min(S)-s}\leq \nu_{\min(S)} \frac{\tilde{\delta}}{1 - \tilde{\delta}}.
	\end{equation}
	By the same reasoning as before, $\nu_{\min(S)}$ is also $o(1)$ as $q\uparrow 1$. We conclude by Lemma \ref{lem:m_1_b} that 
	\begin{equation}
		m_1 = \sum\limits_{s\in S} (1-q)q^{2s} \nu_s + \sum\limits_{s\notin S}\nu_s q^{2s}(1-q)=\sum\limits_{s\in S} (1-q)q^{2s} \nu_s + o(1-q),\quad \text{as }q\uparrow 1.
	\end{equation}
	By the definition of $S$ we also have
	\begin{equation}
		(1 - \epsilon) (1-q)\frac14 \sum\limits_{s\in S}\nu_s \leq \sum\limits_{s\in S} (1-q)q^{2s} \nu_s  \leq (1 + \epsilon) (1-q)\frac14 \sum\limits_{s\in S}\nu_s.
	\end{equation}
	Now $\sum_{s\in S}\nu_s = 1-o(1)$ as $q\uparrow 1$, the result follows. \par
	
	We now show the limits 
	\begin{align}
		(m_1,m_2,\ldots)\label{eq:conv1} \to (0,0 ,\ldots)\qquad\text{ as }q\uparrow 1,\\
		(m_1,m_2,\ldots)\label{eq:conv2} \to (1,0,\ldots)\qquad\text{ as }q\downarrow 0.
	\end{align}

	By Part~\ref{itm:m1part2} of Theorem\ref{thm:mi} and the previous we have for $q\downarrow 0$ the expansion
	\begin{equation}
		\sum_{i=2}^\infty m_i \leq  \sum_{i=2}^{\infty} i m_i = 1 - m_1  = 2q + O(q^2).
	\end{equation}
	So necessarily all the $m_i$ converge to $0$ for $i\geq 2$ as $q\downarrow 0$. \par

	Let $\Pi_n \sim \on{Mallows}(n,q)$. For the limit as $q\uparrow 1$ we count the number of times 
	during the sampling algorithm that $i$ is the maximum element of an $r$--cycle for $r \geq 1$. 
	So suppose that during iteration $i$ of the algorithm there is an element $i' \leq i$ such that 
	having ${\Pi_{n}}(i) = i'$ would create an $r$--cycle. Then there is some $k\geq 0$ such that 
	the probability that ${\Pi_{n}}(i) = i'$ is 
	\begin{equation}
		\frac{(1-q)q^{k-1}}{1-q^{n-s+1}} \leq \frac{1-q}{1-q^{n-s+1}}. 
	\end{equation}
	For $i\leq n- \sqrt{n}$ the above is not more than $2(1-q)$ for $n$ large enough. 
	So the expected number of elements in ${\Pi_{n}}$ that are the maximum element of an $r$--cycle is 
	not more than $O(\sqrt{n}) + 2 (1-q)n$. Thus $ m_r = \lim_{n\to\infty}\bE [C_r(\Pi_n)/ n] \leq 2(1-q) $ so 
	that $\lim_{q\uparrow 1 }m_r = 0$. \par

	This shows the pointwise convergence of the two sequences in Part~\ref{itm:m1part3} of Theorem~\ref{thm:mi}. 
	That the pointwise convergence implies convergence wrt.~the $L_1$ norm follows from the 
	fact that $\sum_{j\geq i} m_j \leq \frac{1}{i}$ for all $i$ by Part~\ref{itm:m1part2} of Theorem~\ref{thm:mi}.
	So, for every $\eps>0$ we can take $i > 2/\eps$ and find a $\delta>0$ so that 
	$m_1, \dots, m_{i} < \eps/(2i)$ whenever $1-\delta<q<1$. 
	If then follows $|(m_1,m_2,\dots) - (0,0,\dots)|_1 \leq \eps$ for all $0<q<\delta$.
	The case when $q\downarrow 0$ can be dealt with analogously.

\section{The proof of Theorem~\ref{thm:mu2i}}

\subsection{The proof of Part~\ref{itm:mu2ipart1} of Theorem~\ref{thm:mu2i}}

We start by proving the existence of a coupling in the same spirit as the couplings used in previous proofs.

\begin{lem}\label{lem:compare_probs}
	Let $0<q<1$ and $\Pi\sim \on{Mallows}(\bN, q)$ and $\Sigma \sim \on{Mallows}(\bZ, q)$. 
	There exists a coupling between $\Pi$ and $\Sigma$ such that 
	\begin{equation}
		\bP\left[ \Pi(i)=\Sigma(i) \text{ for all $i\geq \log^2 n$ } \right] = 1 - o(1).
	\end{equation}
	Moreover, the $\log^2 n$ can be replaced with any function going to $\infty$ with $n$.
\end{lem}

\begin{proof}
We can run both $\Pi$ and $\Sigma$ until both processes regenerate simultaneously, in the sense that 
$\Pi(i)\geq k$ for all $i\geq k$ and $\Sigma(i)\geq k$ for all $i\geq k$. After this time, both processes 
will have the same distribution, and so can be coupled to be equal. The first time that both processes 
regenerate is equal in distribution to the hitting time of $(0,0)$ for the Markov chain defined by taking 
two independent copies of $M_i$, where we start by first running both chains until $\Sigma$ regenerates 
(so the Markov chain starts from a random state $(0,X)$, where $X$ is the state of the second copy of 
the Markov chain stopped at the first time that $\Sigma$ regenerates). 
Since the product chain is still positive recurrent,
%(it has a stationary distribution given by a product of the stationary distributions for each coordinate), 
this hitting time is finite almost surely, and the result follows.
\end{proof}

Let $q>1$ and $\Pi_n\sim\Mallows(n,q)$.
Since $0 \leq C_{2i}(\Pi_n)/n \leq 1$, Theorem~\ref{thm:clt_q_gt_one} implies

$$ \mu_{2i} = \lim_{n\to\infty} \frac{\Ee C_{2i}(\Pi_n)}{n}. $$

Let $\Pi,\Pi'\sim\Mallows(\eN,1/q)$ be independent and coupled with $\Pi_n$ by the coupling 
provided by Corollary~\ref{cor:PinPicpl2sidedq_geq_one}.
We set 

$$ 
U_n := \left|\left\{\text{$i$-cycles of $\Pi'\circ\Pi$ completely contained in $\left[\lfloor n/2\rfloor\right]$}\right\}\right|. 
$$

\noindent
By~\eqref{eq:aapie} in the proof of Theorem~\ref{thm:clt_q_gt_one} (setting $\ell=i, a_i=1, a_1=\dots=a_{i-1}=0$) 
we have that

$$ |C_{2i}(\Pi_n) - U_n| = o(n), $$

\noindent
with probability $1-o(1)$. Since also $0\leq U_n/n \leq 1$, it follows that 

\begin{equation}\label{eq:Umu} 
\lim_{n\to\infty} \frac{\Ee U_n}{n} = \mu_{2i}.  
\end{equation}

Using Lemma~\ref{lem:compare_probs}, we can couple $\Pi,\Pi'$ to $\Sigma,\Sigma'\sim\Mallows(\Zed,1/q)$ in such a way 
that $\Sigma,\Sigma'$ are independent and $\Pee( E ) = 1-o(1)$, where 

$$ E := \{ \Sigma(j) = \Pi(j), \Sigma'(j)=\Pi'(j) \text{ for all $\log^2 n \leq j \leq n -\log^2 n$}\}. $$

\noindent
We also have $\Pee( F ) = 1-o(1)$, where 

$$ F := \{ |\Sigma(j)-j|, |\Sigma'(j)-j| \leq \log^2 n \text{ for all $1\leq j \leq n$} ). $$

\noindent
by an application of Part~\ref{itm:gnedin_displacements} of Lemma~\ref{lem:gnedin} and the union bound. 
Hence, if $A_j$ denotes the event that $j$ is in an $i$-cycle of $\Pi'\circ\Pi$
that is completely contained in $\left[\lfloor n/2\rfloor\right]$, and
$B_j$ denotes the event that $j$ is in an $i$-cycle of $\Sigma'\circ\Sigma$
then 

$$ \left| \Pee(A_j) - \Pee(B_j)\right| \leq \Pee(E^c) + \Pee(F^c) = o(1), $$

\noindent
for all $2i\cdot\log^2 n \leq j \leq n/2 - 2i \log^2 n$.
It follows that 

$$ \begin{array}{rcl} 
\Ee U_n 
& = & \displaystyle \frac{1}{i} \sum_{1\leq j \leq n/2} \Pee( A_j) \\
& = & \displaystyle \frac{1}{i} \sum_{1\leq j \leq n/2} \Pee( B_j) +  o(n) \\
& = & \displaystyle \frac{1}{i} \cdot \lfloor n/2\rfloor \cdot \Pee\left(\text{$0$ in an $i$-cycle of $\Sigma'\circ\Sigma$}\right) + o(n), 
\end{array} $$

\noindent
using Part~\ref{item:gnedin_shift} of Lemma~\ref{lem:gnedin} and the remarks following that lemma for the last identity
(applied to both $\Sigma$ and $\Sigma'$).
Dividing LHS and RHS by $n$ and sending $n\to\infty$ (and recalling~\eqref{eq:Umu}) gives 
$\mu_{2i} = \Pee\left(\text{$0$ in an $i$-cycle of $\Sigma'\circ\Sigma$}\right)$.

Finally, we briefly clarify how the expression $\mu_{2i} = \sum_{i\in\Zed} \Pee( \Sigma(0)=i)^2$ is obtained. 
We have 

$$ \begin{array}{rcl} 
\Pee\left(\text{$0$ in a $1$-cycle of $\Sigma'\circ\Sigma$}\right)
& = & \displaystyle \sum_{i\in\Zed} \Pee( \Sigma(0)=i, \Sigma'(i)=0 ) \\[2ex]
& = & \displaystyle \sum_{i\in\Zed} \Pee( \Sigma(0)=i) \Pee(\Sigma(0)=-i) \\[2ex]
& = & \displaystyle \sum_{i\in\Zed} \Pee( \Sigma(0)=i)^2, \end{array} $$

\noindent 
using that $\Sigma, \Sigma'$ are i.i.d.~and Part~\ref{item:gnedin_shift} of Lemma~\ref{lem:gnedin}, and the remarks following 
that lemma, for the second identity; and in the last identity that

$$-\Sigma(0) = (r\circ\Sigma\circ r)(0) \isd \Sigma(0), $$ 

\noindent
also by the remarks following Lemma~\ref{lem:gnedin}.

\subsection{The proof of Part~\ref{itm:mu2ipart2} of Theorem~\ref{thm:mu2i}}

Let $\ell \geq 1$ and define 
	\begin{equation}
		a_i = 2i,\qquad \text{ for }i=1,\ldots, \ell.
	\end{equation}
Define $\varphi(\Pi_n) = \sum_{i=1}^\ell a_i C_{2i} (\Pi_n)$ and $X_1,Y_1$ as in the proof of Theorem~\ref{thm:clt_q_gt_one}. 
We have $0\leq \varphi(\Pi_n) / n \leq 1$, so that Theorem~\ref{thm:clt_q_gt_one} in fact implies

$$ \bE\left[\frac{\varphi(\Pi_n)}{n}\right] \to \sum_{i=1}^{\ell} a_i \mu_{2i} = \sum_{i=1}^\ell 2i \mu_{2i}, $$ 

while~\eqref{eq:aapie} implies

$$ \bE\left[\frac{\varphi(\Pi_n)}{n}\right] \to \frac{\bE Y_1}{2 \bE X_1}. $$

The function $Y_1$ counts the number of elements in $[X_1]\cup \{n-X_1+1,\ldots ,n\}$ that are contained in even cycles of 
length at most $2\ell$. If $X_1 \leq \ell$ then we have $Y_1 = 2X_1$, as all elements are in cycles of length $\leq 2\ell$.
If $X_1 > \ell$ then certainly $Y_1 \leq 2 X_1$. Let 
	\begin{equation}
		a \be \bE [X_1 1_{\{X_1 \leq \ell\}}]  \qquad\text{ and } 
		\qquad b \be\bE [X_1 1_{\{X_1 > \ell\}}].
	\end{equation}
	We have $\bE X_1 = a+ b$. We also have the bounds $ 2a \leq  \bE X_1 \leq 2a + 2b$.
	As $\bE X_1 < \infty$, for any $\epsilon>0$ we can choose $\ell_0=\ell_0(\eps)$ large enough so that 
	for every $\ell > \ell_0$ we have $b < \epsilon$. 
	As $1\leq \bE X_1$, having chosen $\ell_0$ sufficiently large, we can also ensure $a \geq 1 - \epsilon$ for all 
	$\ell > \ell_0$. In this case
	\begin{equation}
		1\geq \frac{\bE Y_1 }{2 \bE X_1} \geq \frac{a}{a + \epsilon}  = 1 - \frac{ \epsilon}{a + \epsilon} \geq  1 - \epsilon.
	\end{equation}
	So %
$$ \sum_{i=1}^{\infty}2i\, \mu_{2i}=\lim _{\ell \to \infty}   \left(   \frac{\bE Y_1}{2 \bE X_1}  \right) = 1. $$

\subsection{The proof of Part~\ref{itm:mu2ipart3} of Theorem~\ref{thm:mu2i}}

As in the proof of Part~\ref{itm:m1part3} of Theorem~\ref{thm:mi}, it suffices
to prove pointwise convergence (convergence for each $\mu_{2i}$ separately), and the convergence in $L_1$ 
will follow using that $\sum_i 2i \mu_{2i} = 1$.

We first show that $(\mu_2, \mu_4,\ldots)\to (0,0,\ldots )$ as $q\downarrow 1$. 
Let $r\geq 1$ and $\Pi_n \sim \on{Mallows}(n,q)$ where $q>1$.
By Part~\ref{itm:mu2ipart1} of Theorem~\ref{thm:mu2i}, we have %
\begin{equation}\label{eq:miesies}
		\mu_{2r} = \lim_{n\to \infty} \frac{\bE C_{2r}(\Pi_n)}{n}.
\end{equation}%

If $i_1,i_2,\ldots, i_{2r}$ form a $2r$-cycle in $\Pi_n\sim \on{Mallows}(n,q)$ with $q>1$, and
$\Pi_n^* := r_n \circ \Pi_n$ then $\Pi_n^* \isd \Mallows(n,1/q)$ and we must
have 

$$ \Pi_n^*(i_1) = r_n(i_2), \dots, \Pi_n^*(i_{2r-1}) = r_n(i_{2r}), \Pi_n^*(i_{2r}) = r_n(i_1). $$

Let us define

$$ A_j := \{\text{$j$ is the largest element of a $2r$-cycle of $\Pi_n$}\}. $$

\noindent
(To clarify $j$ being the largest element, we of course just mean $j = \max(i_1,\dots,i_{2r})$ if
$i_1,\dots,i_2r$ is a $2r$-cycle as above.)
We consider the iterative procedure for generating $\Pi_n^*$.
When we generate the image of $\Pi_n^*(j)$, having already determined 
$\Pi_n^*(1),\dots,\Pi_n^*(j-1)$ it may not be possible that $j$ is the largest element 
of some $2r$-cycle.
If it is still possible, then we need $\Pi_n^*(j)$ to be some specific value among the still available ones.
(To be precise, $(r_n\circ\Pi_n^*)^{(-2r-1)}(j)$.) 
Since we sample according to a truncated geometric distributions, the probability is thus at most

$$ \Pee( A_j ) \leq \frac{1-(1/q)}{1-(1/q)^{n+1-j}} \leq 2 \cdot (1-1/q),. $$

\noindent
where the last inequality holds for all $j \leq n-\log^2 n$.
To follows that 

$$ \Ee C_{2r}(\Pi_n) = \sum_{j=1}^n \Pee( A_j ) \leq 2n (1-1/q) + o(n). $$

\noindent
dividing by $n$, sending $n\to\infty$ and recalling~\eqref{eq:miesies}, shows
$\mu_{2r} \leq 2(1-1/q)$.
In particular $\lim_{q\downarrow 1} \mu_{2r} = 0$.

	We now proceed to show that $(\mu_2, \mu_4,\mu_6\ldots)\to (1/2,0,0,\ldots )$ as $q\to \infty$. 
	By Part~\ref{itm:mu2ipart2} of Theorem \ref{thm:mu2i} it is enough to show that 
	$\mu_2 \to 1/2$ as $q\to \infty$. 
	Let $\Sigma \sim \on{Mallows}(\bZ, 1/q)$. We have 
	\begin{equation}
		2\mu_2 = \bP \left[ \Sigma(0) = 0 \right] ^2 
		+ 2\sum\limits_{i=1}^\infty \bP \left[ \Sigma(0)=i \right] ^2 
		= \bP \left[ \Sigma(0) = 0 \right]^2 + O\left( \sum\limits_{i\geq 1} q^{-2i} \right).
	\end{equation}
	The last equality is due to Part~\ref{item:gnedin_disp} of Lemma~\ref{lem:gnedin}.  
	By Theorem~\ref{thm:mi} the above is
	\begin{equation}
		(1-2/q + O(1/q^2))^2 + O(1/q^2) = 1 - 4/q + O(1/q^2),\qquad\text{ as }q\to\infty.
	\end{equation}

	\section{The proof of Theorem~\ref{thm:ceco}}

	\subsection{The proof Part~\ref{itm:cecopart1} of Theorem~\ref{thm:ceco}}

Let $\Sigma\sim \on {Mallows}(\bZ, 1/q)$ with $0 < 1/q <1$. 
By definition of $c_e$ we have 

$$ c_e := \sum_{i\in\Zed} \Pee\left[ (r\circ\Sigma)(i) = i \right] = 
\sum_{i= -\infty}^\infty \bP \left[ \Sigma(i)  = -i \right]. $$

\noindent
By the remarks following Lemma \ref{lem:gnedin}, $\Sigma(i)-i \isd \Sigma(0)$. 
Thus 
\begin{equation}
	c_e = \sum\limits_{i= -\infty}^\infty \bP \left[ \Sigma(0)  = -2i \right] = \bP \left[ \Sigma(0)\text{ even} \right]  .
\end{equation}
Similarly 
\begin{equation}
	c_o = \sum_{i\in\Zed} \Pee\left[ (\rho\circ\Sigma)(i)=i \right]
	= \sum_{i\in\Zed} \bP \left[ \Sigma(i)  = 1-i \right]  
	= \sum_{i\in\Zed} \bP \left[ \Sigma(0)  = 1-2i \right]  
	= \bP \left[ \Sigma(0)\text{ odd} \right].
\end{equation}

\subsection{The proof of Part~\ref{itm:cecopart2} of Theorem~\ref{thm:ceco}}

We will make use of the following relatively elementary observation.

\begin{lem}\label{lem:bound_pij}
	Let $0<q<1$ and $\Sigma \sim \on {Mallows}(\bZ, q)$. Then 
	\begin{equation}\label{eq:khbrjhrgbjrehg}
		q \,\bP \left[ \Sigma(0)=j \right]  \leq 
		\bP \left[ \Sigma(0) = j+1 \right]  \leq \frac1q \, \bP \left[ \Sigma(0) = j \right].
	\end{equation}
\end{lem}

\begin{proof}
	Fix some $j \in \bZ$, let $n>j$, $I_n = \{-n,\ldots, n\}$ and $\Sigma_{2n+1}\sim \on{Mallows}(  I_n, q)$. Let $A$ be the set of all permutations $\sigma$ of $I_n$ such that $\sigma(0) = j$, and $B$ be the set of all such permutation $\sigma$ with $\sigma(0) = j+1$. Let $\sigma_{j,j+1} = (j\,\ j+1)$ be the permutation of $I_n$ swapping $j$ and $j+1$. Then the map $\phi (\sigma) = \sigma_{j,j+1} \circ \sigma$ is a bijection from $A$ to $B$. To see this, note that $\phi (\phi (\sigma)) = \sigma$ as $\sigma_{j,j+1}$ is its own inverse. Moreover, if $\sigma(0) = j$, then $\phi (\sigma (0)) = \sigma_{j,j+1}(j) = j+1$. We claim that $|\on{inv} (\sigma) - \on{inv}(\phi(\sigma)) |  = 1$: Let $k=\sigma^{-1}(j+1) \neq 0$. All inversions $(a,b)$ with $\{a,b\} \cap \{0,k\} = \emptyset$ are still inversions in $\phi(\sigma)$ as their images are unchanged in $\phi(\sigma)$. Inversions $(a,b)$ of $\sigma$ with $| \{a,b\} \cap \{0,k\}| = 1$ are also unchanged, as for all $s\notin \{0,k\}$ we have $\Sigma_{2n+1}(s) < j$ if and only if $\Sigma_{2n+1}(s) < j+1$. So the only ordered pair that can be an inversion of exactly one of $\sigma$ and $\phi(\sigma)$ is either $(0,k)$ or $(k,0)$, depending on whether or not $0 < k$. Thus we have
	\begin{equation}
		\bP \left[ \Sigma_{2n+1} (0) = j + 1\right] = \sum\limits_{\sigma \in A} \bP \left[ \phi(\sigma) \right] \leq \frac{1}{q} \, \sum\limits_{\sigma \in A} \bP \left[\sigma \right] = \frac{1}{q}\, \bP \left[ \Sigma_{2n+1}(0) = j \right] .
	\end{equation}
	The bound $\bP \left[ \Sigma_{2n+1} (0) = j + 1\right] \geq q\,  \bP \left[ \Sigma_{2n+1}(0) = j \right]  $ follows in the same manner. As $\Sigma_{I_n} \to \Sigma$ with probability $1$ by Lemma \ref{lem:gnedin}, Part \ref{item:gnedin_converge}, the inequality in \eqref{eq:khbrjhrgbjrehg} holds. 
\end{proof}

The last lemma allows us to give short proof of the following explicit bounds on $c_e,c_o$, that will immediately 
imply the value of the $q\downarrow 1$ limits equals $1/2$.

\begin{lem}\label{lem:gjrknrkjgn}
	For all $q > 1$ we have 
	\begin{align}
		\frac{1}{1+q} &\leq c_e, c_o \leq \frac{q}{1+q}.
	\end{align}
\end{lem}

\begin{proof}
	The second equality will follow from the first as $c_e+c_o = 1$. Let $q > 1$ and $\Sigma \sim \on{Mallows}(\bZ, 1/q)$. By Lemma \ref{lem:bound_pij} we have $\frac{1}{q} \bP \left[ \Sigma(0)=j \right]  \leq \bP \left[ \Sigma(0) = j+1 \right]  \leq q \, \bP \left[ \Sigma(0) = j \right]$. Thus
	\begin{align}\label{eq:rkhrgjhbfschkdbcww}
		\frac1q \bP \left[ \Sigma(0) \text{ odd} \right] & \leq \bP \left[ \Sigma(0) \text{ even} \right] \leq q \, \bP \left[ \Sigma(0) \text{ odd} \right].
	\end{align}
	Adding $\bP \left[ \Sigma(0) \text{ odd} \right] $ and rearranging gives the two inequalities
	\begin{align}
		\bP \left[ \Sigma(0)\text{ odd} \right]&  \geq \frac{1}{1+q}  \bP \left[ \Sigma(0) \text{ odd or even} \right]  = \frac{1}{1+q},\\
		\bP \left[ \Sigma(0)\text{ odd} \right]&  \leq \frac{q}{1+q}  \bP \left[ \Sigma(0) \text{ odd or even} \right]  = \frac{q}{1+q}.
	\end{align}
	The same bounds hold for $\bP \left[ \Sigma(0)\text{ even} \right]$ as 
	$\bP \left[ \Sigma(0)\text{ even}\right] = 1 - \bP \left[ \Sigma(0) \text{ odd} \right]$.
\end{proof}

As mentioned, from Lemma \ref{lem:gjrknrkjgn} it follows immediately that 
\begin{equation}
	\lim_{q\downarrow 1} c_e = \lim_{q\downarrow 1} c_o= \frac12.
\end{equation}

It remains to prove the following asymptotic expressions for $c_e,c_o$.

\begin{lem}
	$c_o = 1 - 2/q + O(1/q^2)$ and $c_e = 2/q + O(1/q^2)$ as $q\to \infty$.
\end{lem}

\begin{proof}
	The second statement again follows from the first by $c_e + c_o = 1$. Let $q>1$ and $\Sigma \sim \on{Mallows}(\bZ, 1 / q )$. We have
	\begin{equation}
		c_o =  \bP \left[ \Sigma(0) \text{ even} \right]  = \bP \left[ \Sigma(0) = 0 \right]  + \sum\limits_{i=1}^{\infty }\bP \left[ |\Sigma(0)| = 2i \right] .
	\end{equation}
	
The quantity $m_1 = m_1( 1/ q)$ is equal to $\bP \left[ \Sigma(0) = 0 \right]$ by Theorem \ref{thm:mi}. 
By Part~\ref{itm:m1part3} of Theorem~\ref{thm:mi} it has asymptotic expansion $1- 2 / q + O(1 /q^2)$ as $q \to \infty$. 
By Part~\ref{item:gnedin_disp} of Lemma~\ref{lem:gnedin}, the sum on the right 
hand side above is $O(1/q^2)$ as $q\to\infty$. Thus $c_o =  1 - 2 /q + O(1/q^2)$. 	
\end{proof}

\section{The proof of Theorem~\ref{thm:asymp}}

We will prove the following more detailed result that implies Theorem~\ref{thm:asymp}.

\begin{prop}\label{prop:asymp}
For $0<q<1$ we have, as $k\to\infty$ 
\begin{enumerate}
\item\label{itm:asymp1} $\Pee\left( C_1( \rho \circ\Sigma) \geq 2k \right) = \Omega\left( q^{{2k \choose 2}} \cdot (1-q)^{2k} \right)$,
\item\label{itm:asymp2} $\Pee\left( C_1( r \circ\Sigma) \geq 2k \right) = o\left( q^{{2k\choose 2}} \cdot (1-q)^{2k} \right)$.
\item\label{itm:asymp3} $\Pee\left( C_1( r \circ\Sigma) \geq 2k+1 \right) = \Omega\left( q^{{2k+1 \choose 2}} \cdot (1-q)^{2k+1} \right)$,
\item\label{itm:asymp4} $\Pee\left( C_1( \rho \circ\Sigma) \geq 2k+1 \right) = o\left( q^{{2k+1\choose 2}} \cdot (1-q)^{2k+1} \right)$.
\end{enumerate}
\end{prop}

\begin{proofof}{Proposition~\ref{prop:asymp}, Parts~\ref{itm:asymp1} and~\ref{itm:asymp3}}
The proofs of~\ref{itm:asymp1} and~\ref{itm:asymp3} are very similar. We start with~\ref{itm:asymp1}.
Consider $\Pi_n \sim \text{Mallows}( I_n, q )$ with $I_n := \{-n,\dots,n\}$.
By Part~\ref{item:gnedin_converge} of \ref{lem:gnedin}, it suffices to show that 
$\liminf_{n\to\infty} \Pee( (\rho \circ\Pi_n)(i)=i \text{ for $i=-k+1,\dots,k$} )$ 
is lower bounded by $\text{const} \cdot q^{{2k \choose 2}} \cdot (1-q)^k$.

For $J \subseteq \Zed$ let $S_J$ denote the set of all permutations of $J$.
Consider the set of all permutations $\pi \in S_{I_n}$ constructed as follows.
We pick arbitrary permutations $\sigma \in S_{\{-n,\dots,-k\}}, \sigma' \in S_{\{k+1,\dots,n\}}$
and set 

$$ \pi(i) := \begin{cases} 1-i & \text{ if } -k+1 \leq i \leq k, \\
              \sigma(i) & \text{ if } -n \leq i \leq -k, \\
              \sigma'(i) & \text{ if } k+1 \leq i \leq n. 
             \end{cases}
$$

Notice such a permutation satisfies 

$$ C_1(\rho \circ \pi) = 2k, \quad \text{ and } \quad
\text{inv}(\pi) = \text{inv}(\sigma) + \text{inv}(\sigma') + {2k \choose 2}. $$

We have 

$$  \begin{array}{rcl} \Pee( C_1( \rho \circ\Pi_n) \geq 2k ) 
& \geq & \Pee( \Pi_n(i) = 1-i \text{ for all $-k+1\leq i \leq k$ } ) \\[2ex]
& \geq & \displaystyle 
\frac{1}{Z(2n+1,q)} \cdot \left( \sum_{\sigma \in S_{\{-n,\dots,-k\}}} \sum_{\sigma' \in S_{\{k+1,\dots,n\}}}  
q^{\text{inv}(\sigma) + \text{inv}(\sigma') + {2k \choose 2}} \right) \\[3ex]
& = & \displaystyle 
\frac{Z(n-k+1,q) \cdot Z(n-k,q)}{Z(2n+1,q)} \cdot q^{{2k\choose 2}}. 
\end{array} $$

Now we recall that 

$$ Z(m,q) = \prod_{i=1}^m \frac{1-q^i}{1-q} = \Theta\left( \frac{1}{(1-q)^m} \right), $$

\noindent 
as $m \to \infty$. Here we use that $1 > \prod_{i=1}^m (1-q^i) > \prod_{i=1}^\infty (1-q^i) > 0$.
(This last inequality can be easily seen using the Taylor expansion 
$\log(1-x) = x + O(x^2)$.)
So 

$$ \frac{Z(n-k+1,q) \cdot Z_(n-k,q)}{Z(2n+1,q)} = \Theta\left( (1-q)^{2k} \right), $$

\noindent
and Part~\ref{itm:asymp1} follows.

The proof of~\ref{itm:asymp3} is essentially the same as the proof of~\ref{itm:asymp1}. Now we 
put the elements of $\{-k,\dots,k\}$ in reverse order, and put arbitrary permutations on 
$\{-n,\dots,-k-1\}$ and $\{k+1,\dots,n\}$ and the proof carries through with only minor adaptations in notation.
\end{proofof}

We next turn attention to Part~\ref{itm:asymp2} of Proposition~\ref{prop:asymp}.
This proof is a bit more involved, and we break it down into several steps. The first step is the following observation.

\begin{lem}
For $0<q<1$ we have 

$$ \Pee\left( \text{sign}\left(\Sigma(i)\right) \neq \text{sign}(i) \text{ for some $i$ with $|i|\geq k^3$ } \right) = 
o\left( q^{{2k\choose 2}} \cdot (1-q)^{2k} \right), $$

\noindent
as $k \to \infty$.
\end{lem}

\begin{proof}
 This immediately follows by the result on displacements 
 $\Pee( |\Sigma(i)-i| > m ) = \Theta(q^m)$ as in Lemma \ref{lem:gnedin} part \ref{item:gnedin_disp} and the union bound.
\end{proof}

By this last lemma, when determining the probability $\Pee\left( C_1( r \circ\Sigma) \geq 2k \right)$, 
we can restrict attention to the event that there are $2k$ points $x_1,\dots,x_{2k}$ in the 
interval $(-k^3, k^3)$ such that $\Sigma(x_1)=-x_1,\dots,\Sigma(x_{2k})=-x_{2k}$.

For $0<i_1<i_2<\dots<i_\ell$ and $0<j_1<j_2<\dots<j_r$ let us write 

$$ 
p_{\underline{i},\underline{j}} := 
\Pee\left( \Sigma(-i_1)=i_1,\dots,\Sigma(-i_\ell)=i_\ell, \Sigma(j_1)=-j_1,\dots,\Sigma(j_r)=-j_r \right). $$
% 
% $$ p_{\underline{i},\underline{j}} := 
% \Pee\left(\Sigma(0)=0,\Sigma(-i_1)=i_1,\dots,\Sigma(-i_\ell)=i_\ell, \Sigma(j_1)=-j_1,\dots,\Sigma(j_r)=-j_r \right). $$

We point out that, as $\Sigma^{-1}$ and $\Sigma$ follow the same distribution, we have 

\begin{equation}\label{eq:symm} p_{\underline{j},\underline{i}}=p_{\underline{i},\underline{j}}. 
%, \quad p_{\underline{j},\underline{i}}=p_{\underline{i},\underline{j}}. 
\end{equation}

\begin{lem}\label{lem:pij} 
There is a constant $c$ such that we have
 
 $$ 
 p_{\underline{i},\underline{j}} \leq 
 (1-q)^{\ell+r} \cdot q^{\Psi(\underline{i},\underline{j})},  
 $$
 
 \noindent
 where 
 
 $$ \Psi(\underline{i},\underline{j}) := 
 {\ell\choose 2}+{r\choose 2} + 
 2 \sum_{a=1}^\ell i_a + \frac12 \sum_{c\leq a\leq r-c} \min(a,r-a)\cdot (j_a-j_{a-1}-1). $$
 
\end{lem}

\begin{proof}
We are again going to consider $\Pi_n \sim \text{Mallows}(I_n,q)$ with $I_n=\{-n,\dots,n\}$ and $n$ large. 
For notational convenience we set 

$$ A_n := \{ \Pi_n(j_1)=-j_1, \dots, \Pi_n(j_r)=-j_r\}, $$

$$ B_n := \{ \Pi_n(-i_1)=i_1, \Pi_n(-i_2)=i_2, \dots, \Pi_n(-i_{\ell})=i_{\ell} \}. $$

By Mallows' iterative procedure for generating $\Pi_n$ we have 

$$ \Pee( \Pi_n(i) = j | \Pi_n(-n)=x_{-n}, \dots, \Pi_n(i-1)=x_{i-1} ) 
= \begin{cases} \left(\frac{1-q}{1-q^{n-i+1}}\right)\cdot q^{j-i+k} & \text{ if } j \not\in \{x_{-n},\dots,x_{i-1}\}, \\
   0 & \text{ otherwise.} 
  \end{cases},
 $$

\noindent 
where 

$$ k = k(x_{-n},\dots,x_{i-1}) := |\{ a : x_a > j \}|. $$

\noindent
It follows that 

$$ \begin{array}{l} \Pee( \Pi_n(-i_\ell) = i_{\ell} ) \leq \left(\frac{1-q}{1-q^{n+i_\ell+1}}\right)q^{2i_\ell}, \\
\Pee( \Pi_n(-i_{\ell-1}) = i_{\ell-1} |  \Pi_n(-i_\ell) = i_{\ell} ) 
\leq \left(\frac{1-q}{1-q^{n+i_{\ell-1} + 1}}\right)q^{2i_{\ell-1}+1}, \\
\Pee( \Pi_n(-i_{\ell-2}) = i_{\ell-2} |  \Pi_n(-i_\ell) = i_{\ell}, \Pi_n(-i_{\ell-1}) = i_{\ell-1} ) 
\leq \left(\frac{1-q}{1-q^{n+i_{\ell-1}+1}}\right)q^{2i_{\ell-2}+2}, \\
\hspace{10ex} \vdots \\ 
\Pee( \Pi_n(-i_1) = i_1 |  \Pi_n(-i_\ell) = i_{\ell}, \dots, \Pi_n(-i_2)=i_2 ) 
\leq \left(\frac{1-q}{1-q^{n+i_1+1}}\right) q^{2i_{1}+\ell-1}. 
\end{array} $$

In other words,

$$ \Pee( B_n ) 
\leq (1-q)^\ell \cdot q^{2 \sum_{a=1}^\ell i_a + {\ell\choose 2}} \cdot 
\prod_{a=1}^\ell \left(\frac{1}{1-q^{n+i_a+1}}\right). 
$$

Note that for $A_n$ to hold, it must be the case 
that 

$$ \begin{array}{l} \Pi_n[\{-n,\dots,j_1-1\}] \cap \{-j_1,\dots,-j_r\} = \emptyset,  \\
\Pi_n[\{-n,\dots,j_2-1\}] \cap \{-j_2,\dots,-j_r\} = \emptyset,\\ 
\hspace{10ex} \vdots \\
\Pi_n[\{-n,\dots,j_r-1\}] \cap \{-j_r\} = \emptyset
\end{array} $$

In particular, when choosing the image of $j_1$ we must ``skip'' at least the first $r-1$ available
numbers, when choosing the image of $j_2$ we must ``skip'' at least $r-2$ available
numbers, and so on.

We let $c$ be a large constant to be specified later on.
Suppose that, for some $c \leq a \leq r-c$, there exists a $j_{a} < j < j_{a+1}$.
If both $A_n$ and $\Pi_n(j)>-j_a$ are to hold then 
when determining $\Pi_n(j)$ we must skip over the (still available) numbers 
$-j_r, \dots, -j_{a+1}$. I.e., we skip over at least the first $r-a$ of the available numbers.
If on the other hand, $A_n$ and $\Pi_n(j)<-j_a$ are to hold then when determining each of 
$j_1,\dots,j_{a}$ must have skipped an additional available number.

It follows that 

$$ \begin{array}{rcl} 
\Pee( A_n | B_n ) & \leq & \left(\frac{1-q}{1-q^{n-j_1+1}}\right) q^{r-1} \cdot 
\left(\frac{1-q}{1-q^{n-j_2+1}}\right) q^{r-2}
\cdot \dots 
\cdot \left(\frac{1-q}{1-q^{n-j_r+1}}\right) q^{0} \\
& & \cdot \prod_{a=c}^{r-c} \prod_{j_{a-1}<j<j_a} 
\left(\frac{q^{r-a} + q^{a}}{1-q^{n-j+1}}\right) \\
& \leq & (1-q)^r q^{{r\choose 2}+\frac12 \sum_{a=c}^{r-c} \min(a,r-a) \cdot (j_a-j_{a-1}-1)} \cdot 
\prod_{a=1}^{i_r} \left(\frac{1}{1-q^{n-a+1}}\right), 
\end{array} $$

\noindent
where in the last line we assume without loss of generality that $c$ has been chosen large enough so that
$q^{c/2} < 1/2$. (Which implies that

$$ q^{r-a} + q^{a} \leq 2 q^{\min(a,r-a)} \leq q^{\frac12 \min(a,r-a)}, $$

for all $c \leq a \leq r-c$.)

Combining the bounds on $\Pee(A_n|B_n)$ and $\Pee(B_n)$ gives

$$ \Pee( A_n\cap B_n ) \leq 
(1-q)^{\ell+r} \cdot q^{\Psi(\underline{i},\underline{j})} \cdot 
\prod_{a=-i_\ell}^{j_r} \left(\frac{1}{1-q^{n-a+1}}\right). $$

\noindent
Since 

$$ \Pee( \Sigma(-i_1)=i_1, \dots, \Sigma(-i_\ell)=i_\ell, \Sigma(j_1)=-j_1,\dots, \Sigma(j_r)=-j_r) 
\leq \limsup_{n\to\infty} \Pee( A_n \cap B_n ), $$

\noindent
the inequality in the lemma follows.
\end{proof}

We make an additional definition:

$$ p_{\underline{i},\underline{j}}^{x\to y} := 
\Pee\left(\Sigma(x)=y, %%\Sigma(0)=0, %% nergens in berekening gebruikt
\Sigma(-i_1)=i_1,\dots,\Sigma(-i_\ell)=i_\ell, \Sigma(j_1)=-j_1,\dots,\Sigma(j_r)=-j_r \right).
$$ 

We observe that, analogously to~\eqref{eq:symm}, we have

\begin{equation}\label{eq:symm2}
 p_{\underline{j},\underline{i}}^{y\to x} = p_{\underline{i},\underline{j}}^{x\to y}. 
\end{equation}

\noindent
Since $ r \circ\Sigma\circ r $ (the map $i \mapsto -\Sigma(-i)$) 
also has the same distribution as $\Sigma$, we have in addition that

\begin{equation}\label{eq:symm3}
 p_{\underline{j},\underline{i}}^{-x\to -y} = p_{\underline{i},\underline{j}}^{x\to y}. 
\end{equation}

\begin{lem}\label{lem:pxy} We have 
\begin{enumerate}
\item\label{itm:pijxy1} If $x<0$ and $y>0$ are such that $|x|,y \not\in \{i_1,\dots,i_\ell\}$ then 

$$ p_{\underline{i},\underline{j}}^{x\to y} 
\leq (1-q)^{\ell+r} q^{\Psi(\underline{i},\underline{j}) + |x|+y}. $$

\item\label{itm:pijxy2} If $x>0$ and $y<0$ are such that $x,|y| \not\in \{j_1,\dots,j_r\}$ and 
$x,|y|>j_{\lceil r/2\rceil}$ then 

$$ p_{\underline{i},\underline{j}}^{x\to y} 
\leq (1-q)^{\ell+r} q^{\Psi(\underline{i},\underline{j}) + r/4}. $$

\end{enumerate}
\end{lem}

\begin{proof}
We start with the proof of~\ref{itm:pijxy1}.
The proof is nearly identical to the proof of the previous lemma, and we only mention the 
changes that need to be made.
Now, we define 

$$ C_n := \{ \Pi_n(x)=y \} \cap B_n, $$

and let $m \geq  0$ be such that 
$i_a > x$ for all $a>m$.

We have 

$$ \Pee( \Pi_n(-x) = y | \Pi_{n}(-i_a ) = i_a \text{ for  all $a > m$ } ) \leq \left(\frac{1-q}{1-q^{n+x}}\right)\cdot q^{x+y}, $$

while for $a \leq m$

$$ \Pee( \Pi_n( - i_a ) = i_a | \Pi_n(-i_\ell)=i_{\ell}, \dots, \Pi_n(-i_{a-1})=i_{a-1}, \Pi_n(-x)=y )
\leq \left(\frac{1-q}{1-q^{n+i_a}}\right)\cdot q^{2i_a + \ell-a}. $$

We can conclude that 

$$ \Pee( C_n ) \leq (1-q)^{\ell+1} \cdot q^{2 \sum_{a=1}^\ell i_a + {\ell\choose 2} + x+y} \cdot 
\prod_{a=1}^\ell \left(\frac{1}{1-q^{n+i_a}}\right) \cdot \left(\frac{1}{1-q^{n+x}}\right). $$

The same reasoning as before shows

$$ \Pee( A_n | C_n ) \leq (1-q)^r q^{{r\choose 2}+\frac12 \sum_{a=c}^{r-c} \min(a,r-a) \cdot (j_a-j_{a-1}-1)} \cdot 
\prod_{a=1}^{i_r} \left(\frac{1}{1-q^{n-a}}\right). $$

Hence

$$ 
p_{\underline{i},\underline{j}}^{x\to y} \leq \limsup_{n\to\infty} \Pee( A_n \cap C_n ) \leq 
(1-q)^{\ell+r+1} q^{\Psi(\underline{i},\underline{j}) + x+y}, $$

as claimed in~\ref{itm:pijxy1}.

For the proof of~\ref{itm:pijxy2}, we again proceed similarly.
This time we define

$$ D_n := \{ \Pi_n(j_1)=-j_1, \dots,\Pi_n(j_r)=-j_r, \Pi_n(x)=y \}. $$

Let $m\geq 0.8r$ be such that $j_1, \dots, j_m < x < j_{m+1}$ (where the upper bound is void 
if $m=r$).
We can compute $\Pee( D_n | B_n )$ in a manner analogous to the way we determined $\Pee( A_n | B_n )$ in the proof of 
Lemma~\ref{lem:pij}. When computing $\Pee( A_n | B_n )$ in the term for the $m$-th gap, the choice of $\Pi_n(x)$ contributed a 
factor $\left(\frac{q^{\frac12 \min(m,r-m)}}{1-q^{n-x}}\right)$ in case $m\leq r-c$ and contributed a factor one otherwise.
If we know that $\Pi_n(x) \leq j_{\lfloor r/2\rfloor}$ then we can replace this by $q^{r/2}$ as 
all of $j_1, \dots,j_{\lceil r/2\rceil}$ must have skipped over the additional available number $y$.
This gives

$$ \Pee( D_n | B_n ) \leq \Pee( A_n | B_n ) \cdot q^{ 0.8r - \frac12 \min(m,r-m) }
\geq q^{r/4}. $$

Hence also 

$$ p_{\underline{i},\underline{j}}^{x\to y} \leq \limsup_{n\to\infty}
\Pee( D_n \cap A_n ) \leq q^{r/2} \cdot \limsup_{n\to\infty} \Pee( A_n \cap B_n ), $$

\noindent
and the result follows.
\end{proof}

For notational convenience we introduce a notation for the gaps between consecutive entries of $\underline{i}$ and $\underline{j}$:

$$ g_1 := i_1-1, \quad g_a := i_a - (i_{a-1}+1) \quad (a=2,\dots,\ell), $$

\noindent
and 

$$ h_{1} := j_1-1, \quad h_a := j_a - (j_{a-1}+1) \quad (a=2,\dots,r). $$

\noindent
Notice that $i_a = g_1 + \dots + g_a + a$ so that

$$ \begin{array}{rcl} 
2 \sum_{a=1}^{\ell} i_a 
& = & \displaystyle 
2\cdot{\big(} (g_1 + 1) + (g_1+g_2+2) + (g_1+g_2+g_3+3) + \dots
+ (g_1+\dots+g_\ell+\ell) {\big)} \\
& = & \displaystyle
\ell(\ell+1) + 2 \sum_{a=1}^\ell (\ell+1-a) \cdot g_a. 
\end{array} $$

This gives the following alternative expression for the upper bound on $p_{\underline{i},\underline{j}}$:

$$ \begin{array}{rcl} 
\Psi(\underline{i}, \underline{j} ) 
& = & \displaystyle 
\ell(\ell+1) + {\ell\choose2} + {r\choose2} 
+ 2 \sum_{a=1}^\ell (\ell+1-a) \cdot g_a + \frac12 \sum_{c\leq a\leq r-c} \min(a,r-a)\cdot h_{a}\\
& = & \displaystyle
{\ell+r\choose 2} + \ell(\ell+1-r)
+ 2 \sum_{a=1}^\ell (\ell+1-a) \cdot g_{a} + \frac12 \sum_{c\leq a\leq r-c} \min(a,r-a)\cdot h_{a}.
\end{array} $$

Next we establish that the probability 
that $C_1( r \circ\Sigma) \geq 2k$ and $\Sigma(0)\neq 0$ is small
compared to the target expression $(1-q)^{2k} \cdot q^{{2k\choose 2}}$.

\begin{lem}\label{lem:Sigma_0_fixed}
We have

$$ \sum_{\ell+r=2k} \sum_{0<i_1<\dots<i_\ell<k^3,\atop 0<j_1<\dots<j_r<k^3} p_{\underline{i},\underline{j}}
= o\left( (1-q)^{2k} \cdot q^{{2k\choose 2}} \right), $$

as $k\to\infty$.
\end{lem}

\begin{proof}
By~\eqref{eq:symm}, the sought sum is at most $2S$ where 

$$ S := \sum_{\ell+r=2k, \atop \ell \geq r} \sum_{0<i_1<\dots<i_\ell<k^3, \atop 0<j_1<\dots<j_r<k^3} 
p_{\underline{i},\underline{j}}.$$

We have

$$ \begin{array}{rcl} 
S 
& \leq  & \displaystyle 
\sum_{\ell+r=2k, \atop \ell \geq r} 
\sum_{0\leq g_1,\dots,g_\ell<k^3, \atop
0 \leq h_1,\dots,h_r<k^3}
(1-q)^{2k} \cdot q^{ {2k\choose 2} + \ell(\ell+1-r)
+ 2 \sum_{a=1}^\ell (\ell+1-a) \cdot g_a + \frac12 \sum_{c\leq a\leq r-c} \min(a,r-a)\cdot h_a} \\
& = & \displaystyle 
(1-q)^{2k} \cdot q^{{2k\choose 2}} \cdot \left( \sum_{\ell+r=2k, \atop \ell \geq r} q^{\ell(\ell+1-r)}
\cdot \prod_{a=1}^\ell \left( \sum_{0\leq g_a <k^3} q^{(\ell+1-a)g_a} \right) 
\cdot \prod_{a=1}^{c-1} \left( \sum_{0\leq h_a < k^3} 1 \right) \right. \\
& & \displaystyle \left.
\cdot \prod_{c\leq a <r/2} \left( \sum_{0\leq h_a<k^3} q^{\frac12 ah_a} \right)
\cdot \prod_{r/2\leq a \leq r-c} \left( \sum_{0\leq h_a<k^3} q^{\frac12 (r-a)h_a} \right)
\cdot \prod_{a=r-c+1}^{r} \left( \sum_{0\leq h_a<k^3} 1 \right) \right).
\end{array} $$

Now we remark that 

$$ \prod_{a=1}^\ell \left( \sum_{0\leq g_a <k^3} q^{(\ell+1-a)g_a}\right) 
\leq \prod_{a=1}^\ell \left( \frac{1}{1-q^{\ell+1-a}}\right)
\leq \prod_{b=1}^\infty \left( \frac{1}{1-q^b} \right) < \infty, $$

\noindent 
and analogously 

$$ \prod_{c\leq a <r/2} \left( \sum_{0\leq h_a<k^3} q^{\frac12 ah_a} \right), 
\prod_{r/2\leq a \leq r-c} \left( \sum_{0\leq h_a<k^3} q^{\frac12 (r-a)h_a} \right) 
\leq \prod_{b=1}^\infty \left( \frac{1}{1-\left(\sqrt{q}\right)^b} \right) < \infty. $$

\noindent
Of course we also have $\sum_{0\leq h_a<k^3} 1 = k^3$.

This gives

$$  
S =
O\left( 
(1-q)^{2k} \cdot q^{{2k\choose 2}} \cdot k^{6c} \cdot \left( \sum_{\ell+r=2k, \atop \ell \geq r} q^{\ell(\ell+1-r)} \right) 
\right). $$

\noindent
As $\ell+r=2k, \ell\geq r$ implies that $\ell(\ell+1-r) \geq k$ we have 

$$ \sum_{\ell+r=2k, \atop \ell \geq r} q^{\ell(\ell+1-r)} \leq (2k+1) q^k. $$

\noindent
We find that

$$ S = O\left(
(2k+1) \cdot k^{6c} \cdot (1-q)^{2k} \cdot q^{{2k\choose 2}+k} \right)
= o\left( (1-q)^{2k} \cdot q^{{2k\choose 2}+k} \right), $$

% 
% \cdot \prod_{a=1}^\ell \left( \frac{1}{1-q^{\ell+1-a} \right) 
% \cdot k^{3c}  \right. \\
% & & \displaystyle \left.
% \cdot \prod_{c\leq a <r/2} \left( \frac{1}{1-q^{a h_a} \right)
% \cdot \prod_{r/2\leq a \leq r-c} \left( \sum_{0\leq h_a<k^3} q^{(r-a) h_a} \right)
% \cdot \prod_{a=r-c+1}^{r} \left( \sum_{0\leq h_a<k^3} 1 \right) \right) \\
% 
% & \leq &  \displaystyle 
% \sum_{\ell+r=2k, \atop \ell \geq r} 
% (1-q)^{2k} \cdot q^{{2k\choose 2} + \ell(\ell+1-r)} \cdot 
% k^{3 \cdot (2c+1)} \cdot \prod_{a=1}^\ell \left(\frac{1}{1-q^a}\right) \cdot \prod_{a=c}^{\lceil r/2\rceil} 
% \left(\frac{1}{1-q^a}\right) \\
% & \leq & \displaystyle 
% \sum_{\ell+r=2k, \atop \ell \geq r} 
% (1-q)^{2k} \cdot q^{\ell(\ell+1)+{\ell\choose2}+{r\choose2}} \cdot 
% k^{3 \cdot (2c+1)} \cdot \left( \prod_{a=1}^\infty \frac{1}{1-q^a}\right)^2. 
% \end{array} $$
% 
% Next, we point out that 
% 
% $$ \ell(\ell+1) + {\ell\choose2}+{r\choose2} = {\ell+r\choose 2} + \ell(\ell+1-r)
% \geq {2k\choose 2} + k, $$
% 
% when $\ell+r=k$ and $\ell\geq r$.
% We see that 
% 
% $$ S = O\left( (2k+1) \cdot k^{6c+3} \cdot (1-q)^{2k} \cdot q^{{2k\choose 2} + k} \right)
% = o\left( (1-q)^{2k} \cdot q^{{2k\choose 2}} \right), $$

\noindent
as required.
\end{proof}

We proceed by showing that the the probability that {\bf a)} $C_1( r \circ\Sigma) \geq 2k$, and; {\bf b)} $\Sigma(0)=0$, and; 
{\bf c)} the number of fixed points of $ r \circ\Sigma$ below zero differs by more than one from the 
number of fixed points of $ r \circ\Sigma$ above zero, is small
compared to the target expression $(1-q)^{2k} \cdot q^{{2k\choose 2}}$.

\begin{lem}
We have 

$$ \sum_{\ell+r=2k-1,\atop |\ell-r|>1} \sum_{0<i_1<\dots<i_\ell<k^3, \atop 0<j_1<\dots<j_r<k^3} 
p_{\underline{i},\underline{j}}
= o\left( (1-q)^{2k} \cdot q^{{2k\choose 2}} \right), $$

as $k\to\infty$.
\end{lem}

\begin{proof}
Arguing as in the previous lemma, the sought sum is at most

$$ k^{O(1)} \cdot (1-q)^{2k-1}\cdot q^{{2k-1\choose 2}} \cdot\sum_{\ell+r=2k-1,\atop \ell\geq k+1} 
 q^{\ell(\ell+1-r)}. $$

We remark that when $\ell+r=2k-1, \ell \geq k+1, r\leq k-2$ we have 

$$ {2k-1\choose 2} + \ell(\ell+1-r) 
\geq {2k-1\choose 2} + 3k+3 = {2k\choose 2} + k + 4. $$

\noindent
In particular the sought sum is at most

$$ k^{O(1)} \cdot (1-q)^{2k}  
q^{{2k\choose2}+k} = o\left( (1-q)^{2k}  
q^{{2k\choose2}}\right). $$

\noindent
(The number of choices of $\ell,r$ with $\ell+r=2k, \ell>r+1$ is $O(k)$ and is absorbed in the polynomial term $k^{O(1)}$.)
\end{proof}

By the results so far, we can restrict attention to the situation where $\Sigma(0)=0$ and $\ell=k, r=k-1$.
Next, we establish that the contribution from the situation in which $\Sigma(-i) \neq i$ for some $1\leq i \leq 0.9k$ is negligible.

\begin{lem}
We have 

$$  \sum_{{0<i_1<\dots<i_{k}<k^3, \atop 0<j_1<\dots<j_{k-1}<k^3,} \atop i_a \neq a \text{ for some $1\leq a \leq 0.9k$} } 
p_{\underline{i},\underline{j}}
= o\left( (1-q)^{2k} \cdot q^{{2k\choose 2}} \right), $$

as $k\to\infty$.

\end{lem}

\begin{proof}
If $i_a \neq a$ for some $1\leq  a \leq 0.9k$ then 
there is also a $1\leq a \leq 0.9k$ for which $g_a\neq 0$.
For this $a$ we have $2 (k+1-a) \cdot g_a \geq 0.2k$.
Also note that

$$ {2k-1\choose 2} + k(k+1-(k-1)) = {2k-1\choose 2} + 2k = {2k\choose 2} + 1. $$

\noindent
Arguing as in previous lemmas, it follows that the sought sum is at most 

$$ k^{O(1)} \cdot (1-q)^{2k-1} \cdot q^{{2k\choose 2} + 0.2k}
= o\left( (1-q)^{2k} q^{{2k\choose 2}} \right). $$

(The polynomial term $k^{O(1)}$ also absorbs the $k$ ways of choosing an index $a$ for which 
$g_a \neq 0$ and the $k^3$ ways of choosing a value for $g_a$.)
\end{proof}

\begin{lem}
We have 

$$  \sum_{{0<i_1<\dots<i_{k}<k^3, \atop 0<j_1<\dots<j_{k-1}<k^3,} \atop 
j_a \neq j_{a-1}+1 \text{ for some $0.1k\leq a \leq 0.9k$} } 
p_{\underline{i},\underline{j}}
= o\left( (1-q)^{2k} \cdot q^{{2k\choose 2}} \right), $$

as $k\to\infty$.

\end{lem}

\begin{proof}
If $j_a \neq j_{a-1}+1$ for some $0.1k\leq a \leq 0.9k$ then, provided $k$ is sufficiently large,
$c \leq a \leq k-1-c$ and moreover $\min(a,k-1-a) \cdot h_a \geq 0.05 k$.
Arguing as in previous lemmas, the sought sum is therefore at most 

$$ k^{O(1)} \cdot (1-q)^{2k-1} \cdot q^{{2k\choose 2} + 0.05 k}
= o\left( (1-q)^{2k} q^{{2k\choose 2}} \right). $$

(The polynomial term also absorbs the $0.9k$ ways of choosing $a$ and the $k^3$ ways of choosing $h_a$.)
\end{proof}

\begin{lem}
We have 

$$  \sum_{{0<i_1<\dots<i_{k}<k^3, \atop 0<j_1<\dots<j_{k-1}<k^3,} \atop h_1+\dots+h_{\lfloor0.1k\rfloor}\geq 2} 
p_{\underline{i},\underline{j}}
= o\left( (1-q)^{2k} \cdot q^{{2k\choose 2}} \right), $$

as $k\to\infty$.

\end{lem}

\begin{proof}
For each $p_{\underline{i},\underline{j}}$ occurring in the sought sum, we have 

$$ p_{\underline{i},\underline{j}} = p_{\underline{j},\underline{i}}
\leq (1-q)^{2k} q^{ {2k-1\choose 2} + 2\sum_{a=1}^{k-1} (r-a)h_a
+ \frac12 \sum_{c\leq a\leq k-c} \min(a,k-a)\cdot g_a}, $$

\noindent
by Lemma~\ref{lem:pij}. (Notice the roles of $\ell=k$ and $r=k-1$ and $g_a$ and $h_a$ are switched.
Notice in particular $(k-1)\cdot( (k-1)+1-k ) = 0$.)
The sought sum is thus at most

$$ S := \sum_{1\leq x<y \leq 0.1k} \sum_{{0\leq h_1,\dots,h_{k-1} < k^3, \atop
0\leq g_1,\dots,g_k<k^3,} \atop
h_x+h_y\geq 2}
(1-q)^{2k} q^{ {2k-1\choose 2} + 2\sum_{a=1}^{k-1} (r-a)\Delta_a
+ \frac12 \sum_{c\leq a\leq k-c} \min(a,k-a)\cdot\Delta_a'}. $$

By computations similar to those in previous proofs:

$$ \begin{array}{rcl} 
S & = & 
k^{O(1)} \cdot (1-q)^{2k} \cdot 
q^{ {2k-1\choose 2} + 4 \cdot 0.9 k } \\
& = & 
k^{O(1)} \cdot (1-q)^{2k} \cdot 
q^{ {2k\choose 2} + 1.6 k } \\
& = &  
o\left( (1-q)^{2k} \cdot 
q^{ {2k\choose 2} } \right).
\end{array} $$ 

\end{proof}

We next observe that if $\Sigma(0)=0$ and $\Sigma(-i_1)=i_1, \dots, \Sigma(-i_k)=i_k$ and 
$\Sigma(j_1)=-j_1, \dots, \Sigma(j_{k-1})=-j_{k-1}$ then 
there must be some $x \geq 1$ and $y\leq -1$ such that $\Sigma(x)=y$ and 
$x \not\in \{j_1,\dots,j_{k-1}\}$, and $y \not\in \{-j_1,\dots,-j_{k-1}\}$.
(Since $\Sigma$ is ``balanced''.)

This will allow us to improve over our previous bounds on $p_{\underline{i},\underline{j}}$.

\begin{proofof}{Proposition~\ref{prop:asymp}, Part~\ref{itm:asymp2}}
By the previous lemmas and the observation immediately preceding the present proof, it suffices to show that

$$  S:= \sum p_{\underline{i},\underline{j}}^{x\to y}
= o\left( (1-q)^{2k} \cdot q^{{2k\choose 2}} \right), $$

\noindent
as $k\to\infty$, where the sum is over all $0<j_1<\dots<j_k\leq k^3$ and $0<i_1<\dots<i_{k-1}\leq k^3$
and $0<x<k^3$ and $-k^3<y<0$
such that
\begin{itemize}
\item $i_a=a$ for all $1\leq a \leq 0.9k$, and; 
\item $j_a=j_{a-1}+1$ for all $0.1k\leq a \leq 0.9k$, and;
\item there is at most one $1\leq a \leq 0.1k$ such that 
$j_a\neq j_{a-1}+1$ and if such an $a$ exists then we have
$j_a=j_{a-1}+2$, and;
\item $x, |y| \not\in J := \{j_1,\dots,j_{k-1}\}$.
\end{itemize}

We first notice that if $\underline{i},\underline{j},x,y$ are as described then 
either {\bf a)} $\{j_1,\dots,j_{\lfloor0.9k\rfloor}\} = \{1,\dots,\lfloor0.9k\rfloor\}$ or
{\bf b)} $\{j_1,\dots,j_{\lfloor0.9k\rfloor}\}
= \{1,\dots,\lfloor0.9k\rfloor+1\}\setminus\{z\}$ for some $1\leq z\leq 0.1k$.

In case {\bf a)} we thus have $x,|y| > j_{\lfloor 0.9k\rfloor}$ and in particular

$$ p_{o,\underline{i},\underline{j}}^{x\to y}
\leq (1-q)^{2k-1} \cdot q^{\Psi(\underline{i}\underline{j})+r/4}. $$

The corresponding sum $S_{\bf a}$ thus satisfies

$$ \begin{array}{rcl}
S_{\bf a} 
& := & \displaystyle 
\sum_{{0<i_1<\dots<i_k<k^3, \atop 0<j_1<\dots<j_{k-1}<k^3,}\atop
i_a=j_a=a \text{ for $1\leq a \leq 0.9k$} } \sum_{{0.9k\leq x < k^3,\atop -k^3<y\leq0.9k,} \atop x,|y|\not\in J}
p_{\underline{i},\underline{j}}^{x\to y} \\
& \leq & \displaystyle 
\sum_{\underline{i},\underline{j}} \sum_{{0.9k\leq x < k^3,\atop -k^3<y\leq0.9k,} \atop x,|y|\not\in J}
(1-q)^{2k-1} \cdot q^{\Psi(\underline{i}\underline{j})+r/4} \\
& \leq & \displaystyle 
k^6 \sum_{\underline{i}\underline{j}} (1-q)^{2k-1} \cdot q^{\Psi(\underline{i}\underline{j})+r/4} \\
& \leq & \displaystyle 
k^6 \cdot (1-q)^{2k-1} \cdot q^{{2k\choose 2}+r/4} \sum_{0\leq g_1,\dots,g_k<k^3, \atop 0\leq h_1,\dots,h_{k-1}<k^3}
q^{ 2\sum_{a=1}^k (k+1-a)g_a + \frac12 \sum_{c\leq a\leq k-1-c} \min(a,k-1-a) h_a } \\
& \leq & \displaystyle 
k^{O(1)} (1-q)^{2k-1} q^{{2k\choose 2}+r/4} \\
& = & \displaystyle
o\left( (1-q)^{2k} q^{{2k\choose 2}} \right).
\end{array} $$

(Using the familiar observation that ${2k-1\choose 2} + k(k+1-(k-1)) = {2k\choose 2} + 1$ in the fourth line, and 
computations as in previous lemmas.)

In case {\bf b)} it is possible that either {\bf b-1)} $x,|y|\geq 0.9k$, or
{\bf b-2)} $x=z$ and $y=-z$, or {\bf b-3)} $x>0.9k$ and $y=-z$, or
{\bf b-4)} $x=z$ and $y<-0.9k$.

In the case {\bf b-1)} the same bound on 
$p_{o,\underline{i},\underline{j}}^{x\to y}$ applies as in the case {\bf a)}, and 
via similar computations we obtain that the corresponding contribution to the sum satisfies

$$ \begin{array}{rcl}
S_{\bf b-1} 
& := & \displaystyle 
\sum_{1\leq b \leq 0.1k}
\sum_{{0<i_1<\dots<i_k<k^3, \atop 0<j_1<\dots<j_{k-1}<k^3,}\atop
{i_a=a \text{ for $1\leq a \leq 0.9k$,} \atop 
{j_a=a \text{ for $a\leq b$}, \atop 
j_a=a+1 \text{ for $b< a \leq 0.9k$}}}} 
\sum_{{0.9k\leq x < k_3,\atop -k^3<y\leq0.9k,} \atop x,|y|\not\in J}
p_{\underline{i},\underline{j}}^{x\to y} \\
& & \\
& \leq & \displaystyle 
k^{O(1)} (1-q)^{2k-1} q^{{2k\choose 2}+r/4} \\
& = & \displaystyle
o\left( (1-q)^{2k} q^{{2k\choose 2}} \right).
\end{array} $$

In case {\bf b-2)} we have $p_{\underline{i},\underline{j}}^{x\to y} = p_{\underline{i},\underline{j'}}$, where 
$j_1'=1,\dots, j_{\lfloor 0.9k\rfloor}' = \lfloor 0.9k\rfloor$ and 
$j_a' = j_{a-1}$ for $0.9k \leq a \leq k-1$. 
(I.e.~we've filled the ``gap at $z$'' and made a vector of length $k$.)
So the sum corresponding to case {\bf b-2)} is 

$$ S_{\bf b-2} \leq 
\sum_{0<i_1<\dots<i_k<k^3, \atop 0<j_1'<\dots<j_k'<k^3} p_{\underline{i},\underline{j'}}
= k^{O(1)} (1-q)^{2k} q^{{2k\choose 2} + k} = o\left( (1-q)^{2k} q^{{2k\choose2}}\right). 
$$

To deal with case {\bf b-3)} we note that in this case

$$ \begin{array}{rcl} p_{\underline{i},\underline{j}}^{x\to y} 
& = & p_{\underline{j},\underline{i}}^{y\to x} \\
& \leq & (1-q)^{2k-1} q^{{2k-1\choose2} + 2(k-1-x) + x+|y| + 2\sum_{0.9k\leq a\leq k-1} (k-1-a)h_a
+ \frac12\sum_{0.9k\leq a\leq k-c} (k-a)g_a } \\
& \leq & (1-q)^{2k-1} q^{{2k-1\choose2} + 3.7k - 2 + 2\sum_{0.9k\leq a\leq k-1} (k-1-a)h_a
+ \frac12\sum_{0.9k\leq a\leq k-c} (k-a)g_a} \\
& \leq & 
(1-q)^{2k-1} q^{{2k\choose2} + 1.7k + 2\sum_{0.9k\leq a\leq k-1} (k-1-a)h_a
+ \frac12\sum_{0.9k\leq a\leq k-c} (k-a)g_a},
\end{array}
$$

Hence the corresponding sum satisfies

$$ \begin{array}{rcl}
S_{\bf b-3} 
& := & \displaystyle 
\sum_{1\leq x \leq 0.1k}
\sum_{{0<i_1<\dots<i_k<k^3, \atop 0<j_1<\dots<j_{k-1}<k^3,}\atop
{i_a=a \text{ for $1\leq a \leq 0.9k$,} \atop 
{j_a=a \text{ for $a< x$}, \atop 
j_a=a+1 \text{ for $x\leq a \leq 0.9k$}}}} 
\sum_{-k^3<y\leq0.9k, \atop |y|\not\in J}
p_{\underline{i},\underline{j}}^{x\to y} \\
& & \\
& = & \displaystyle 
o\left( (1-q)^{2k} q^{{2k\choose 2}} \right).
\end{array} $$

Finally we deal with case {\bf b-4)}.
We now use~\eqref{eq:symm3} to see that in this case

$$ \begin{array}{rcl} p_{\underline{i},\underline{j}}^{x\to y}
& = & p_{\underline{j},\underline{i}}^{-x\to -y} \\
& \leq & (1-q)^{2k-1} q^{{2k-1\choose 2} + 2(k-1-x) + x+|y| + 2\sum_{0.9k\leq a\leq k-1}(k-1-a)h_a 
+ \frac12 \sum_{c\leq a\leq k-c} \min(a,k-a)g_a } \\
& \leq &  
(1-q)^{2k-1} q^{{2k-1\choose 2} + 3.7k-2 + 2\sum_{0.9k\leq a\leq k-1}(k-1-a)h_a 
+ \frac12 \sum_{c\leq a\leq k-c} \min(a,k-a)g_a }, 
\end{array} $$

\noindent 
and hence repeating the computations bounding $S_{\bf b-3}$ we find

$$ S_{\bf b-4} 
:=
\sum_{-0.1k\leq y \leq 1}
\sum_{{0<i_1<\dots<i_k<k^3, \atop 0<j_1<\dots<j_{k-1}<k^3,}\atop
{i_a=a \text{ for $1\leq a \leq 0.9k$,} \atop 
{j_a=a \text{ for $a<|y|$}, \atop 
j_a=a+1 \text{ for $|y|\leq a \leq 0.9k$}}}} 
\sum_{0.9k\leq x \leq k^3, \atop x\not\in J}
p_{\underline{i},\underline{j}}^{x\to y} 
=
o\left( (1-q)^{2k} q^{{2k\choose 2}} \right).$$

This establishes that $S=o\left( (1-q)^{2k} q^{{2k\choose 2}} \right)$. 
Proposition~\ref{prop:asymp}, Part~\ref{itm:asymp2} is proved.
\end{proofof}

\begin{proofof}{Proposition~\ref{prop:asymp}, Part~\ref{itm:asymp4}}
   The proof proceeds in the same manner as the proof of Proposition \ref{prop:asymp} Part \ref{itm:asymp2}, we highlight here only the differences. We will now define for sequences $0 < i_1<\ldots < i_\ell$ and $1 < j_1 <\ldots < j_r$ the probability 
   \begin{equation}
      \tilde{p}_{\vec{i},\vec{j}} = \bP \left( \Sigma(-i_1) =i_1 + 1,\ldots, \Sigma(-i_\ell) = i_\ell + 1,\Sigma(j_1)= -j_1+1,\ldots, \Sigma(j_r) = -j_r+1 \right).
   \end{equation}
   For sequences satisfying the above we have $(\rho \circ \Sigma)(i_s) = i_s$ and $(\rho\circ\Sigma)(j_s) = j_s$. For $\tilde{p}_{\vec{i},\vec{j}}$ the bound in Lemma \ref{lem:pij} may be replaced by the stronger bound
   \begin{equation}\label{eq:new_bound_pij}
      \tilde{p}_{\vec{i},\vec{j}} \leq (1-q)^{\ell + r} q^{\ell +r+\Psi(\vec{i},\vec{j})},
   \end{equation}
   the proof is the same, but we now have 
   $$ \Pee( \Pi_n(i) = j + 1 | \Pi_n(-n)=x_{-n}, \dots, \Pi_n(i-1)=x_{i-1} ) 
   = \begin{cases} \left(\frac{1-q}{1-q^{n-i+1}}\right)\cdot q^{j-i+k + 1} & \text{ if } j \not\in \{x_{-n},\dots,x_{i-1}\}, \\
      0 & \text{ otherwise.} 
     \end{cases},
    $$
   which gives the additional $\ell + r$ term in the exponent of $q$ in \ref{eq:new_bound_pij}. Also, as in Lemma \ref{lem:Sigma_0_fixed}, the contribution from such sequences where $\Sigma(1) \neq 0$ or $\Sigma(0) \neq 1$ is $ o\left( q^{{2k + 1 \choose 2}} \cdot (1-q)^{2k + 1} \right)$. Then we need only consider cases with $\Sigma(1) = 0$ and $\Sigma(0) = 1$, so that we consider the case $\ell + r = 2k-1$. The remaining estimates for the different such sequences $i_1,\ldots, i_\ell$ and $j_1,\ldots , j_{r}$ follow in the same manner as the estimates for $r \circ \Sigma$, where now the $q^{\ell + r}$ term in \eqref{eq:new_bound_pij} contributes $q^{2k-1}$. In the above proofs we estimate $\bP \left[ \Sigma(0) = 0 \right] $ by a a coupling with $\Pi_n \sim \on{Mallows}(n,q)$, which we can sample element by element. If during this sampling $0$ does not yet have a preimage when determining the image of $0$, then $\bP \left[ \Sigma(0)=1 \right] \leq q \ \bP \left[ \Sigma(0) = 0 \right]$. Subsequently, when determining the image of $1$, if $\Sigma(0)=1$, then $\bP \left[ \Sigma(1) = 0 \right]$ equals the probability that $\Sigma(0)=0$ in the previous step, as $1$ now needs to skip over exactly all elements below $0$ that are not yet selected. In the estimations of the $r\circ \Sigma$ case, the additional contribution of $q^{\ell  + r}q = q^{2k}$ is exactly as needed as $q^{\binom{2k}{2} + 2k} = q^{\binom{2k+1}{2}}$. 
\end{proofof}

\section{Suggestions for further work}

We expect that much more information can be extracted about the constants $c_e, c_o, m_i, \mu_{2i}$ and 
the probability measures of $C_{2i-1}(r\circ\Sigma)$ and $C_{2i-1}(\rho\circ\Sigma)$ from the explicit 
expressions for them in terms of $q$-hypergeometric series that 
can be obtained from Theorem 5.1 in~\cite{Gnedin}.
There might well be some low-hanging fruit available to someone better versed in $q$-hypergeometric series than 
the present authors. In particular, the plot in Figure~\ref{fig:fixed} strongly 
suggests that $c_o > 1/2 > c_e$ for all $q>1$, but we have not been able to show this rigorously.
We leave it as an open problem for other teams.

A curious phenomenon we've observed is that if $q\downarrow 1$ then the expected number of 1-cycles in the limiting 
distribution tends to $1/2$, which is different from the value of 1 that we get when $q=1$ (which corresponds
to sampling a permutation uniformly at random).
At the moment we do not even have a reasonable intuitive explanation for this phenomenon.
Clearly something interesting must be going on in the phase change regime when $q=q(n)$ is a function of $n$ that 
approaches one from above.
It is also intriguing that the odd and even cycle counts behave so differently for $q>1$, the odd cycle counts 
being ``tight'' and the even ones being linear in $n$, but at $q=1$ there appears to be no trace of this difference. 
We would be very interested to see an analysis of the regime when $q=q(n)\to1$ that can shed some light on these phenomena.

Finally, given the Gaussian fluctuations for fixed $q$, and Poisson fluctuations when $q$ is close to $1$, it would be interesting to understand what happens in the intermediate regime, and whether there is a phase transition. Let us note that the techniques of the renewal theoretic techniques used in this paper are not very helpful, because the lengths between regenerations grow extremely quickly as $q\to 1$.

\subsection*{Acknowledgements}

We thank Daniel Valesin for helpful pointers to the literature. The first author would like to thank Persi Diaconis for bringing the problem to his attention.

\bibliographystyle{plain} 
\bibliography{bibfile}	

\end{document}